\documentclass[leqno,11pt]{amsart}
\usepackage{amssymb, amsmath}
\usepackage{amsthm, amsfonts,mathrsfs}
\def\inte#1{
\displaystyle\mathop{#1\kern0pt}^\circ }

\let\pa=\partial

\let\d=\delta

\let\f=\frac

\let\p=\psi

\let\D=\Delta

\let\wt=\widetilde

\def\cC{{\mathcal C}}

\def\cF{{\mathcal F}}

\def\cI{{\mathcal I}}

\def\cK{{\mathcal K}}
\def\cL{{\mathcal L}}

\def\cS{{\mathcal S}}

\def\pa{\partial}
\def\grad{\nabla}
\def\dB{\dot{B}}

\def\virgp{\raise 2pt\hbox{,}}
\def\cdotpv{\raise 2pt\hbox{;}}

\def\eqdefa{\buildrel\hbox{\footnotesize def}\over =}

\def\C{\mathop{\mathbb C\kern 0pt}\nolimits}
\def\DD{\mathop{\mathbb D\kern 0pt}\nolimits}
\def\EE{\mathop{{\mathbb E \kern 0pt}}\nolimits}
\def\K{\mathop{\mathbb K\kern 0pt}\nolimits}
\def\N{\mathop{\mathbb N\kern 0pt}\nolimits}
\def\Q{\mathop{\mathbb Q\kern 0pt}\nolimits}
\def\R{\mathop{\mathbb R\kern 0pt}\nolimits}
\def\SS{\mathop{\mathbb S\kern 0pt}\nolimits}
\def\ZZ{\mathop{\mathbb Z\kern 0pt}\nolimits}
\def\TT{\mathop{\mathbb T\kern 0pt}\nolimits}
\def\P{\mathop{\mathbb P\kern 0pt}\nolimits}

\newcommand{\la}{\lambda}

\newcommand{\Z}{{\ZZ}}

\def\dv{\mbox{div}}

\def\dive{\mathop{\rm div}\nolimits}

\def\Supp{\mathop{\rm Supp}\nolimits\ }

\def\no{\noindent}
\def\na{\nabla}
\def\p{\partial}

\newcommand{\beq}{\begin{equation}}
\newcommand{\eeq}{\end{equation}}
\newcommand{\ben}{\begin{eqnarray}}
\newcommand{\een}{\end{eqnarray}}
\newcommand{\beno}{\begin{eqnarray*}}
\newcommand{\eeno}{\end{eqnarray*}}
\newcommand{\andf}{\quad\hbox{and}\quad}
\newcommand{\with}{\quad\hbox{with}\quad}
\newtheorem{defi}{Definition}[section]
\newtheorem{thm}{Theorem}[section]
\newtheorem{lem}{Lemma}[section]
\newtheorem{rmk}{Remark}[section]
\newtheorem{cor}{Corollary}[section]
\newtheorem{prop}{Proposition}[section]
\renewcommand{\theequation}{\thesection.\arabic{equation}}

\begin{document}

\title[Well-posedness of $2$-D  inhomogeneous incompressible Navier-Stokes system ]
{On the global existence and uniqueness  of solution to $2$-D  inhomogeneous incompressible Navier-Stokes equations in  critical spaces}
\bigbreak\medbreak
\author[H.  Abidi]{Hammadi Abidi}
\address[H.  Abidi]{D\'epartement de Math\'ematiques
Facult\'e des Sciences de Tunis
Universit\'e de Tunis EI Manar
2092
Tunis
Tunisia}\email{hammadi.abidi@fst.utm.tn}
\author[G. Gui]{Guilong Gui}
\address[G. Gui]{School of Mathematics and Computational Science, Xiangtan University,  Xiangtan 411105,  China}\email{glgui@amss.ac.cn}
\author[P. Zhang]{Ping Zhang}
\address[P. Zhang]{Academy of Mathematics $\&$ Systems Science, The Chinese Academy of Sciences, Beijing 100190, China;\\
School of Mathematical Sciences, University of Chinese Academy of Sciences, Beijing 100049, China} \email{zp@amss.ac.cn}

\setcounter{equation}{0}

\maketitle
\begin{abstract} In this paper, we establish the global existence and uniqueness of solution to $2$-D inhomogeneous incompressible  Navier-Stokes equations \eqref{1.2} with initial data in the critical spaces. Precisely, under the assumption that the initial velocity  $u_0$ in $L^2 \cap\dot B^{-1+\frac{2}{p}}_{p,1}$ and the initial density $\rho_0$ in $L^\infty$ and having a positive lower bound, which satisfies
$1-\rho_0^{-1}\in \dot B^{\frac{2}{\lambda}}_{\lambda,2}\cap L^\infty,$ for $p\in[2,\infty[$ and $\lambda\in [1,\infty[$ with $\frac{1}{2}<\frac{1}{p}+\frac{1}{\lambda}\leq1,$ the system \eqref{1.2} has a global solution. The solution is unique if $p=2.$
With additional assumptions on the initial density in case $p>2,$ we can also prove the uniqueness of such solution. In particular,
this result improves the previous work  in \cite{AG2021} where  $u_{0}$ belongs to $\dot{B}_{2,1}^{0}$ and $\rho_0^{-1}-1$ belongs to $\dot{ B}_{\frac{2}{\varepsilon},1}^{\varepsilon}$, and we also remove the assumption that the initial density is close enough to a positive constant  in \cite{DW2023} yet with additional regularities on the initial density here.
\end{abstract}

\vskip 0.2cm


\renewcommand{\theequation}{\thesection.\arabic{equation}}
\setcounter{equation}{0}

\section{Introduction}
In this paper, we investigate the global well-posedness   of the following
$2$-D inhomogeneous incompressible Navier-Stokes equations:
\begin{equation}\label{1.2}
\begin{cases}
\pa_t \rho + \dv (\rho u)=0,\qquad (t,x)\in\R_+\times\R^2, \\
\rho(\pa_t u +u\cdot\nabla u )-\Delta u+\grad\Pi=0, \\
\dv\, u = 0, \\
(\rho,u)|_{t=0}=(\rho_0,u_0),
\end{cases}
\end{equation}
where the unknowns $\rho$ and $u=(u_1,u_2)^T$ stand for the density and velocity of the fluid respectively, and $\Pi$  is a  scalar pressure function, which guarantees the divergence free condition of the velocity field. Such a system can be used to describe the mixture of several immiscible fluids that are incompressible and with different densities, it can  also characterize a fluid containing a molten substance.

It is easy to observe that for any smooth enough solution $(\rho, u)$ of \eqref{1.2}, one has the following energy law:
\beq \label{ener}
\f12\int_{\R^2}\rho |u|^2\,dx+\int_0^t\int_{\R^2}|\na u|^2\,dx\,dt'=\f12\int_{\R^2}\rho_0|u_0|^2\,dx.
\eeq
Based on the energy law, Kazhikov \cite{KA} proved  that the $d$-dimensional system \eqref{1.2}  (with $d=2, 3$) has a global weak solution
provided that the initial density is bounded from above and away from vacuum, the initial velocity belongs to $H^1$ (the size of $H^1$ norm
should be sufficiently small in three space dimension).
Danchin and  Mucha \cite{DM13} solved the uniqueness problem with smoother velocity.
The uniqueness of Kazhikov
weak solution was solved in \cite{PZZ} (see \cite{CZZ, DW2023, Z20} for the improvements).
Lately Danchin and Mucha \cite{DM}
established  the existence and uniqueness of such solution even allowing the appearing of vacuum.
In general, DiPerna and Lions \cite{DL, LP} proved the global
existence of weak solutions to \eqref{1.2} in the energy space in any space
dimensions. Yet the uniqueness and regularities of such weak
solutions are listed as open questions  by Lions in \cite{LP}.


On the other hand, if the initial data of the density $\rho$ is away from zero, we
denote by $a\eqdefa{\rho^{-1}}-1$, then the system \eqref{1.2} can be
 equivalently reformulated as
\begin{equation}\label{1.3}
\quad\left\{\begin{array}{l}
\displaystyle \pa_t a + u \cdot \grad a=0,\qquad (t,x)\in \R_+\times\R^2,\\
\displaystyle \pa_t u + u \cdot \grad u+ (1+a)(\grad\Pi-\Delta\,u)=0, \\
\displaystyle \dv\, u = 0, \\
\displaystyle (a, u)|_{t=0}=(a_0, u_{0}).
\end{array}\right.
\end{equation}
Just as the classical Navier-Stokes equations, which corresponds to the case when $a=0$ in \eqref{1.3}, the system \eqref{1.3} also has
the following  scaling-invariant property: if $(a, u)$ solves \eqref{1.3} with initial data $(a_0, u_0)$, then for any $\ell>0$,
\begin{equation*} 
(a, u)_{\ell}(t, x) \eqdefa (a(\ell^2\cdot, \ell\cdot), \ell
u(\ell^2 \cdot, \ell\cdot))
\end{equation*}
is also a solution of \eqref{1.3} with initial data $(a_0(\ell\cdot),\ell
u_0(\ell\cdot))$. We call such functional spaces as critical spaces if the norms of
which are invariant under the scaling transformation $(a_0,u_0)\mapsto (a_0(\ell\cdot),\ell
u_0(\ell\cdot)).$

Danchin \cite{danchin04} first established the global well-posedness of the system \eqref{1.3} with
initial data in the almost critical Sobolev spaces. After the works \cite{A,A-P,DAN-03} in the critical framework,
 Danchin and Mucha \cite{DM1} eventually proved the global well-posedness
of  \eqref{1.3} with initial density being close enough to a positive constant in the multiplier space of
$\dot{B}^{-1+\f{d}p}_{p,1}(\R^d)$ and initial velocity being small enough in $\dot{B}^{-1+\f{d}p}_{p,1}(\R^d)$ for $1\leq p<2d.$
The work of \cite{A-G-Z-2} is the first to investigate the global well-posedness of the 3-D incompressible inhomogeneous Navier-Stokes equation with initial data in the critical spaces  and yet without   the size restriction on $a_0$. One may check \cite{DW2023} and references therein
for the recent progress in this direction.

In  two dimensions and with  initial density being bounded from above and away from vacuum,
 Danchin \cite{danchin04} proved the global well-poedness of the system \eqref{1.2} if
  $\rho_0^{-1}-1\in H^{1+\alpha}$ and $u_0\in H^\beta$ with $\alpha,\beta>0$.
  The  authors of \cite{A-Z} proved the global existence and uniqueness of the solution to the system \eqref{1.2} with variable
   viscosity when the viscosity is close enough to a positive constant,
    and  $\rho_0^{-1}-1\in\dot B^1_{2,1}\cap\dot B^{\alpha}_{\infty,\infty}$ with $\alpha>0$ and $u_0\in\dot B^0_{2,1}$ (one
    may check \cite{D97} for the existence result of the system \eqref{1.2} with $H^1$ initial data and also \cite{PZ20}
     together with the references therein for the rough density case).
     Haspot \cite{BHas} proved the global well-posedness of system \eqref{1.2} with small initial
      velocity $u_0\in \dot{B}^{\frac{2}{p_2}-1}_{p_2, r}$ and more regular initial density $\rho_0^{-1}-1 \in {B}^{\frac{2}{p_1}+\varepsilon}_{p_1, \infty}$ with
      some technical conditions on $p_1$, $p_2$, $r$ and $\varepsilon$. Recently, the first two authors of this paper
     improved the above result in \cite{AG2021} to that $u_{0}\in\dot{B}_{2,1}^{0}$ and $\rho_0^{-1}-1\in\dot{ B}_{\frac{2}{\varepsilon},1}^{\varepsilon}$ with $M_1\le\rho_0\leq M_2$ and $0<\varepsilon<1$.
      This is, to the best of our knowledge, the first global well-posedness result of \eqref{1.2}
       in the critical framework that does not require any smallness condition.
       More recently, based on  interpolation results, time weighted estimates and maximal regularity estimates for time evolutionary Stokes system in Lorentz spaces (with respect to the time variable), Danchin and Wang \cite{DW2023} obtained the existence and uniqueness of
       the system \eqref{1.2} when the initial data $\rho_0$ is close to a positive constant in $L^\infty$ and $u_0\in L^2 \cap\dot B^{-1+\frac{2}{p}}_{p,1}$ with $1<p<2$.

Inspired by \cite{AG2021},  we shall investigate the global well-posedness of the system \eqref{1.2} with initial data in the general critical spaces.
The main result  states as follows.

\begin{thm}\label{thm1.1}
{\sl Let  $M_1, M_2$ be two positive constants, $p\in [2, +\infty[$ and
$\lambda\in [1,+\infty[$ with $\frac{1}{2}<\frac{1}{p}+\frac{1}{\lambda}\le1.$
We assume that $u_0\in L^2 \cap\dot B^{-1+\frac{2}{p}}_{p,1}$ is a
solenoidal vector field and $1-\rho_0^{-1}\in
\dot B^{\frac{2}{\lambda}}_{\lambda,2}\cap L^\infty$
satisfies
\begin{equation}\label{1.4}
M_1\le\rho_0\leq M_2.
\end{equation}
Then the system \eqref{1.2} has a  global  solution $(\rho,u,\na\Pi)$  which satisfies
\begin{equation}\label{AG}
\begin{split}
&\rho^{-1}-1 \in \cC([0,\infty[;\, \dot{B}^{\frac{2}{\lambda}}_{\lambda,2} \cap L^\infty),\quad M_1\leq \rho(t, x) \leq M_2 \quad
\mbox{for all}\,\,(t, x) \in \mathbb{R}_+\times \mathbb{R}^2,
\\&
u\in  \cC([0,\infty[;\,L^2\cap\dot{B}^{-1+\frac{2}{p}}_{p,1})
\cap\widetilde L^1_{loc}(\R_+;\,\dot H^2)
\cap L^1_{loc}(\R_+;\,\dot B^{1+\frac{2}{p}}_{p,1}),
\\&
\nabla\Pi\in L^1_{loc}(\R_+;\,\dot{B}^{-1+\frac{2}{p}}_{p,1})
\cap L^1_{loc}(\R_+;\,L^2)
\andf
\partial_tu\in L^1_{loc}(\R_+;\,\dot{B}^{-1+\frac{2}{p}}_{p,1})
\cap\widetilde L^1_{loc}(\R_+;\,L^2).
\end{split}
\end{equation}
In particular, for $p=2$, this solution is unique. For $p\in]2,\infty[$,
if in addition, $\rho_0^{-1}-1\in \dot{B}^{2-\frac{2}{p}}_{\frac{p}{p-1},\infty}
\cap\dot{B}^1_{2,1},$
then the solution is unique and satisfies
$
\rho^{-1}-1\in C([0,\infty[;\,\dot{B}^{2-\frac{2}{p}}_{\frac{p}{p-1},\infty}\cap\dot B^1_{2,1}).
$
}
\end{thm}

Notice that for $1\le p<2$,
$\dot B^{-1+\frac{2}{p}}_{p,1}\hookrightarrow
 \dot B^{0}_{2,1}$
and for $p\in [4/3,2[,$ $\lambda=\frac{2p}{2-p}$, $(p,\lambda)$ satisfies $\frac{1}{2}<\frac{1}{p}+\frac{1}{\lambda}\le1,$ then
 we deduce from Theorem \ref{thm1.1} that the system \eqref{1.2}
has a unique global-in-time solution satisfying \eqref{AG}. Precisely

\begin{cor}\label{DW}
{\sl Let   $p\in [4/3,2[$ and $u_0\in\dot B^{-1+\frac{2}{p}}_{p,1}$ be a
solenoidal vector field, and $\rho_0$ satisfies \eqref{1.4} with $1-\rho_0^{-1}\in
\dot B^{-1+\frac{2}{p}}_{\frac{2p}{2-p},1}$.
Then the  system \eqref{1.2} has a unique global solution
$(\rho,u,\na\Pi)$  which satisfies
\begin{equation}\label{AGK}
\begin{aligned}
&\rho^{-1}-1 \in \cC([0,\infty[;\,\dot B^{-1+\frac{2}{p}}_{\frac{2p}{2-p},1}),\quad
u\in \cC([0,\infty[;\,\dot B^{-1+\frac{2}{p}}_{p,1})
\cap L^1_{loc}(\R_+;\,\dot B^{1+\frac{2}{p}}_{p,1}),
\\&
\nabla\Pi\in L^1_{loc}(\R_+;\,\dot B^{-1+\frac{2}{p}}_{p,1}),\quad
\partial_tu\in L^1_{loc}(\R_+;\,\dot B^{-1+\frac{2}{p}}_{p,1}).
\end{aligned}
\end{equation}
}
\end{cor}

\begin{rmk}\label{rmkthm1.1} In some sense, our result here removed the assumption in \cite{DW2023} that the initial data $\rho_0$ is close to a positive constant and also extends the case $p\in]1,2[$  to $p\in [4/3,\infty[$. We believe that Corollary \ref{DW} is correct even for $p\in ]1,4/3[,$ yet we shall not pursue this direction here.
\end{rmk}

\begin{rmk}\label{rmkthm1.2}
The main ideas used to prove the uniqueness part of  Theorem \ref{thm1.1} for the cases $p=2$ and $p \in ]2, +\infty[$ are quite different.
For the case when $p=2,$ we shall combine  the Lagrangian approach with the techniques in \cite{AG2021} to deal with the difference between any two solutions of \eqref{1.2} in the $L^2$ framework (see Proposition \ref{prop2.4} below), which is also different from the Lagrangian method in \cite{{DM1}} where the smallness of the variation of the initial density is required. While for the case when $p \in ]2, +\infty[$, without the smallness assumption on the variation of the initial density, it is difficult for us to close the estimate for the difference  in the $L^p$ framework
  if we use the Lagrangian approach. Instead, we shall perform the estimates in Euclidean coordinates and  rely on the Osgood Lemma to conclude  the uniqueness part in Section \ref{sect-global}.
\end{rmk}

The  structure of this paper lists as follows: In Section \ref{sect-prel}, we shall first collect  some basic facts on Littlewood-Paley theory,
and then to apply it to study some commutator's estimates, finally we shall apply the previous estimates to investigate the linearized equations
of \eqref{1.3}. In Section \ref{sect-Lip}, we shall derive the necessary {\it a priori} estimates used in the proof of Theorem \ref{thm1.1}. In Section \ref{sect-global}, we shall conclude the proof of Theorem \ref{thm1.1}.

\medbreak \noindent{\bf Notations:} For two operators $A, B$, we denote $[A, B]=AB-BA,$ the commutator between $A$ and $B$. For $a\lesssim b$, we mean that there is a uniform constant $C,$ which may be different on different lines, such that $a\leq Cb,$ and $C_{\rm in}$ denotes a positive constant depending  only on  the norm to the initial data. $a\thicksim b$ means that both $a\lesssim b$ and $b\lesssim a$. For $r\in [1, +\infty]$ and $\overline{\N}\eqdefa \mathbb{N}\cup\{-1\},$ we denote $\{c_{q, r}\}_{q \in \mathbb{Z}} $ (or $\{c_{q, r}\}_{q \in \overline{\N}} $) a sequence in $\ell^r(\mathbb{Z})$ (or $\ell^r(\overline{\mathbb{N}})$ ) such that $\|\{c_{q, r}\}_q\|_{\ell^{r}} =1$. In particular, we designate $c_{q,1}$ by $d_q$ and $c_{q,2}$ by $c_q.$

For $X$ a Banach space and $I$ an interval of $\R,$ we denote by ${\mathcal{C}}(I;\,X)$ the set of continuous functions on $I$ with values in $X,$ and by ${\mathcal{C}}_b(I;\,X)$ the subset of bounded functions of ${\mathcal{C}}(I;\,X).$ For $p\in[1,+\infty],$ the notation $L^p(I;\,X)$ stands for the set of measurable functions on $I$ with values in $X,$ such that $t\longmapsto\|f(t)\|_{X}$ belongs to $L^p(I).$

\renewcommand{\theequation}{\thesection.\arabic{equation}}
\setcounter{equation}{0}

\section{Preliminaries}\label{sect-prel}

\subsection{Basic facts on Littlewood-Paley theory}
The proof of Theorem \ref{thm1.1} requires  Littlewood-Paley theory. For the convenience of the readers, we briefly
recall some basic facts in the case of $x\in\R^2$ (see, e.g. \cite{BCD}).
Let $\chi(\tau)$ and $\varphi(\tau)$ be smooth functions such that
\begin{align*}
&\Supp \varphi \subset \Bigl\{\tau \in \R\,: \, \frac34 <
\tau < \frac83 \Bigr\}\quad\mbox{and}\quad \forall
 \tau>0\,,\ \sum_{q\in\Z}\varphi(2^{-q}\tau)=1;\\
& \Supp \chi \subset \Bigl\{\tau \in \R\,: \, 0\leq \tau<
\frac43 \Bigr\}\quad\mbox{and}\quad \forall
 \tau\geq 0\,,\ \chi(\tau)+ \sum_{q\geq 0}\varphi(2^{-q}\tau)=1,
\end{align*}
we define the dyadic operators as follows:
for $u\in{\mathcal S}',$
\begin{equation}\label{LP-decom-sum-1}
  \begin{aligned}
&\dot\Delta_qu\eqdefa\varphi(2^{-q}|\textnormal{D}|)u\ \ \forall q\in\Z,\hspace{1cm}\mbox{and}
\hspace{1cm}
\dot S_qu\eqdefa\sum_{j \leq q-1}\dot\Delta_{j}u,
\\
&
\Delta_qu\eqdefa\varphi(2^{-q}|\textnormal{D}|)u\ \mbox{if}\  q\geq 0,\quad
\Delta_{-1}u\eqdefa\chi(|\textnormal{D}|)u\quad\mbox{and}\quad
S_qu\eqdefa\sum_{j=-1}^{ q-1}\Delta_{j}u.
\end{aligned}
\end{equation}
The dyadic operator satisfies the
property of almost orthogonality:
\begin{equation*}
\begin{split}
&\dot\Delta_k\dot\Delta_q u\equiv 0
\quad\mbox{if}\quad\vert k-q\vert\geq 2
\quad\mbox{and}\quad\dot\Delta_k(\dot S_{q-1}u\dot\Delta_q u)
\equiv 0\quad\mbox{if}\quad\vert k-q\vert\geq 5,
\\&
\Delta_k\Delta_q u\equiv 0
\quad\mbox{if}\quad\vert k-q\vert\geq 2
\quad\mbox{and}\quad\Delta_k( S_{q-1}u\Delta_q u)
\equiv 0\quad\mbox{if}\quad\vert k-q\vert\geq 5.
\end{split}
\end{equation*}
\begin{defi}\label{def1.1}
{\sl Let $s\in\R$, $1 \leq p,r\leq +\infty$ and $\overline{\N}\eqdefa \mathbb{N}\cup\{-1\},$ we define
\begin{itemize}
\item[(1)] the inhomogeneous Besov space $B^s_{p,r}$ to be the set of  distributions $u $ in ${\mathcal S}'$ so that
\begin{equation*}
\|u\|_{B^s_{p,r}}\eqdefa\Big\|2^{qs}\|\Delta_q
u\|_{L^{p}}\Big\|_{\ell ^{r}(\overline{\N})}<\infty,
\end{equation*}

\item[(2)] the homogeneous Besov space $\dot
B^s_{p,r}$ to be the set of  distributions $u$ in ${\mathcal S}_{h}'$
 (${\mathcal S}_{h}'\eqdefa\{u\in {\mathcal S}', \ \lim\limits_{\lambda\rightarrow +\infty}\|\theta(\lambda\,D)u\|_{L^\infty}=0\,\,\text{for any}\,\,\theta \in \mathcal{D}(\mathbb{R}^2)\}$) so that
\begin{equation*}
\|u\|_{\dot B^s_{p,r}}\eqdefa\Big\|2^{qs}\|\dot\Delta_q
u\|_{L^{p}}\Big\|_{\ell ^{r}(\mathbb{Z})}<\infty.
\end{equation*}
\end{itemize}
 }
\end{defi}

\begin{rmk}\label{rmk1.1}
\begin{enumerate}
  \item We point out that if $s>0$ then $B^s_{p,r}=\dot B^s_{p,r}\cap L^p$ and
$$
\|u\|_{B^s_{p,r}}\approx \|u\|_{\dot B^s_{p,r}}+\|u\|_{L^p}.
$$

\item If $u \in B^s_{ p,\infty} \cap B^{\tilde{s}}_{
p,\infty}$ and $s < \tilde{s}$, $\theta \in (0, 1)$,  $1 \leq p\leq \infty,$ then $u
\in  B^{\theta s+ (1-\theta)\tilde{s}}_{ p,1}$ and
\begin{equation}\label{interpo-complex-1}
\|u\|_{B^{\theta
s+ (1-\theta)\tilde{s}}_{ p,1}} \leq \frac{C}{\tilde{s}-s}\bigl(\frac{1}{\theta}+\frac{1}{1-\theta}\bigr)\|u\|_{B^s_{ p,\infty}}^{\theta}
\|u\|_{B^{\tilde{s}}_{ p,\infty}}^{1-\theta}.
\end{equation}
  \item Let $s\in \mathbb{R}, 1\le p,\,r\leq+\infty$, and $u \in
\cS'_h.$ Then $u$ belongs to $\dot{B}^{s}_{p, r}$ if and
only if there exists some positive constant $C$ and some nonnegative generic element $\{c_{q, r}\}_{q \in \mathbb{Z}} $ of $\ell^r(\Z)$ such that
$\|\{c_{q, r}\}_{q\in\Z}\|_{\ell^{r}(\Z)} =1$ and for any $q\in \mathbb{Z}$
\begin{equation*} 
\|\dot{\Delta}_{q}u\|_{L^{p}}\leq C c_{q, r} \, 2^{-q s }
\|u\|_{\dot{B}^{s}_{p, r}}.
\end{equation*}
Similarly, for $u \in \cS'$, $u$ belongs to ${B}^{s}_{p, r}$ if and
only if there holds
\begin{equation*} 
\|\Delta_{q}u\|_{L^{p}}\leq C c_{q, r} \, 2^{-q s }
\|u\|_{B^{s}_{p, r}}.
\end{equation*}
\end{enumerate}
\end{rmk}

We also recall Bernstein's inequality from  \cite{BCD}:

\begin{lem}\label{lem2.1}
{\sl Let $\mathcal{B}\eqdefa \{
\xi\in\R^2,\ |\xi|\leq\frac{4}{3}\}$ be a ball   and $\mathcal{C}\eqdefa \{
\xi\in\R^2,\frac{3}{4}\leq|\xi|\leq\frac{8}{3}\}$ a ring.
 A constant $C$ exists so that for any positive real number $\lambda,$ any nonnegative
integer $k,$ any smooth homogeneous function $\sigma$ of degree $m$,
any couple of real numbers $(a, \; b)$ with $ b \geq a \geq 1$, and any function $u$ in $L^a$,
there hold
\begin{equation}
\begin{split}
&\Supp \hat{u} \subset \lambda \mathcal{B} \Rightarrow
\sup_{|\alpha|=k} \|\pa^{\alpha} u\|_{L^{b}} \leq  C^{k+1}
\lambda^{k+ 2(\frac{1}{a}-\frac{1}{b} )}\|u\|_{L^{a}},\\
& \Supp \hat{u} \subset \lambda \mathcal{C} \Rightarrow
C^{-1-k}\lambda^{ k}\|u\|_{L^{a}}\leq
\sup_{|\alpha|=k}\|\partial^{\alpha} u\|_{L^{a}}\leq
C^{1+k}\lambda^{ k}\|u\|_{L^{a}},\\
& \Supp \hat{u} \subset \lambda \mathcal{C} \Rightarrow \|\sigma(D)
u\|_{L^{b}}\leq C_{\sigma, m} \lambda^{ m+2(\frac{1}{a}-\frac{1}{b}
)}\|u\|_{L^{a}}, \end{split}\label{2.1}
\end{equation}}
with $\sigma(D)
u\eqdefa\mathcal{F}^{-1}(\sigma\,\hat{u})$.
\end{lem}

In what follows, we shall frequently use Bony's
decomposition \cite{Bony} in both  homogeneous and inhomogeneous context. The homogeneous Bony's
decomposition reads
\begin{equation}\label{bony}
uv=T_u v+T'_vu=T_u v+T_v u+R(u,v),
\end{equation}
where
\begin{equation*}
\begin{split}
&T_u v\eqdefa\sum_{q \in \mathbb{Z}}\dot S_{q-1}u\dot\Delta_q v,\,
T'_vu\eqdefa\sum_{q \in \mathbb{Z}}\dot\Delta_q u\,\dot S_{q+2}v,\,
 R(u,v)\eqdefa\sum_{q \in \mathbb{Z}}\dot\Delta_q u {\widetilde{\dot\Delta}}_{q}v \,\,\text{with} \,\, {\widetilde{\dot\Delta}}_{q}v\eqdefa
\sum_{|q'-q|\leq 1}\dot\Delta_{q'}v,
\end{split}
\end{equation*}
and the inhomogeneous Bony's decomposition can be defined in a similar manner.

We shall also use the following law of pra-product.

\begin{prop}[Theorems 2.47 and 2.52 in \cite{BCD}]\label{prop2.2}
{\sl \begin{enumerate}
\item  There exits  a constant $C$ so that for $s \in \mathbb{R}$, $t<0$, $p,\,p_1,\,p_2,\,r,\, r_1,\, r_2\in [1, +\infty]$,
\begin{equation*}
\begin{split}
&\|T_uv\|_{\dot{B}^s_{p, r}} \leq C^{|s|+1}\|u\|_{L^{\infty}}\|v\|_{\dot{B}^s_{p, r}},\\
 &\|T_uv\|_{\dot{B}^{s+t}_{p, r}} \leq \frac{C^{|s+t|+1}}{-t} \|u\|_{\dot{B}^t_{p_1, r_1}}\|v\|_{\dot{B}^s_{p_2, r_2}}\quad \mbox{with}
 \quad \frac{1}{p} \eqdefa \frac{1}{p_1}+\frac{1}{p_2},\quad \frac{1}{r} \eqdefa \min\Bigl(1, \frac{1}{r_1}+\frac{1}{r_2}\Bigr).
    \end{split}
\end{equation*}
\item
Let $(s_1, s_2)$ be in $\mathbb{R}^2$ and $(p_1, p_2, r_1, r_2)$ be in $[1,+\infty]^4$. We assume that
$\frac{1}{p} \eqdefa \frac{1}{p_1}+\frac{1}{p_2}\leq 1$ and $\frac{1}{r} \eqdefa \frac{1}{r_1}+\frac{1}{r_2}\leq 1$.
Then there exits  a constant $C$ so that
\begin{equation*}
\begin{split}
&\|R(u, v)\|_{\dot{B}^{s_1+s_2}_{p, r}} \leq  \frac{C^{s_1+s_2+1}}{s_1+s_2} \|u\|_{\dot{B}^{s_1}_{p_1, r_1}} \|v\|_{\dot{B}^{s_2}_{p_2, r_2}} \quad \text{if}\quad  s_1 + s_2 > 0,\\
&\|R(u, v)\|_{\dot{B}^{0}_{p, \infty}} \leq  C \|u\|_{\dot{B}^{s_1}_{p_1, r_1}} \|v\|_{\dot{B}^{s_2}_{p_2, r_2}} \quad \text{if}\quad r = 1 \,\,\text{and}\,\, s_1 + s_2 = 0.
    \end{split}
\end{equation*}
\end{enumerate}}
\end{prop}

In  order to obtain a better description of the regularizing effect
of the transport-diffusion equation, we shall use Chemin-Lerner type
norm from
\cite{CL}.
\begin{defi}\label{def2.2}
{\sl Let $s\in\R$,
$r,\lambda, p\in [1,+\infty]$ and $T>0$.
 we define
\begin{equation*}
\|u\|_{\widetilde L^\lambda_T(B^s_{p,r})}\eqdefa\Big\|2^{qs}\|\Delta_q
u\|_{L^\lambda_T(L^{p})}\Big\|_{\ell ^{r}(\overline{\mathbb{N}})}
\quad\mbox{and}\quad
\|u\|_{\widetilde L^\lambda_T(\dot B^s_{p,r})}\eqdefa\Big\|2^{qs}\|\dot\Delta_q
u\|_{L^\lambda_T(L^{p})}\Big\|_{\ell ^{r}(\mathbb{Z})}.
\end{equation*}
}
\end{defi}

Finally we recall  the following commutator's estimate which will be frequently used throughout this paper.

\begin{lem}[Lemma 1 in \cite{Plan}, Lemma 2.97 in \cite{BCD}](Commutator estimate)\label{lem-commutator-1} Let $(p, s, r) \in [1, +\infty]^3$
satisfy $\f1{r}=\f1p+\f1s,$ $\theta$ be a $C^1$ function on $\mathbb{R}^{d}$ such that $(1+|\cdot|) \hat{\theta} \in L^1$. There
exists a constant $C$ such that for any function a with gradient in $L^p$
and any function $b$ in $L^s$, we have, for any positive $\lambda$,
\begin{equation}\label{commutator-compact-0}
\|[\theta(\lambda^{-1} D), a]b\|_{L^r} \leq C \lambda^{-1}\|\grad a\|_{L^p}\|b\|_{L^s}.
\end{equation}
\end{lem}

\subsection{Some useful estimates}

In this subsection, we shall apply the basic facts in the previous subsection to study some
estimates, which will be used in the subsequent sections. We first present the following commutator's estimate, the proof of which is given  for the sake of completeness.

\begin{lem}\label{lem2.3}
{\sl Let    $ p\in [2,\infty[$ and $u$ be a solenoidal vector field with $\nabla u\in
L^2.$ Then  there holds
\begin{align}\label{2.11}
\sum_{q\in\mathbb{Z}}2^{q\left(-1+\frac{2}{p}\right)}\|[\dot\Delta_{q}, u\cdot \nabla ]u\|_{L^p}
\lesssim
\|\nabla u\|_{L^{2}}^2.
\end{align}
}
\end{lem}
\begin{proof}
Thanks to Bony's decomposition \eqref{bony} and the fact that $\dive\,u=0$, we decompose $[\dot\Delta_q,u\cdot\nabla] u$ into the following four terms:
\begin{align}\label{comma-1}
[\dot\Delta_q,u\cdot\nabla] u&=\dot\Delta_q\bigl(\partial_j
R(u^j, u))+\dot\Delta_q\bigl(T_{\partial_j u}u^j\bigr)
-T'_{\dot\Delta_q\partial_j u}u^j+[\dot\Delta_q, T_{u^j}]\partial_j u \eqdefa
\sum_{i=1}^4 \mathcal{R}^i_q,
\end{align}
where repeated indices means the summation of the index from $1$ to $2.$

We first deduce form  Lemma \ref{lem2.1} that
\begin{align*}
\|\mathcal{R}^1_q\|_{L^p}\lesssim & 2^{q\left(3-\f2p\right)}\sum_{k\geq q-3} \|\widetilde{\dot\Delta}_ku\|_{L^2}\|\dot\Delta_ku^j \|_{L^2}\\
\lesssim & 2^{q\left(3-\f2p\right)}\sum_{k\geq q-3} c_k^22^{-2k}\|\na u\|_{L^2}^2
\lesssim  c_q^22^{q\left(1-\f2p\right)}\|\na u\|_{L^2}^2.
\end{align*}
Here and in all that follows, we always denote $\{c_q\}_{q\in\Z}$ to be a unit generic element of $\ell^2(\Z)$ so that $\sum_{q\in\Z}c_q^2=1.$

While considering the support properties to the Fourier transform of the terms in $
T_{\partial_j u}u^j,$ we infer
\beno
\|\mathcal{R}^2_q\|_{L^p}\lesssim \sum_{\vert
q-k\vert\leq 4} \|\dot{S}_{k-1} \na u\|_{L^\infty}\|\dot\Delta_ku\|_{L^p},
\eeno
yet it follows from Lemma \ref{lem2.1} that
\beq \label{S2eq1}
\|\dot{S}_{k-1} \na u\|_{L^\infty}\lesssim \sum_{\ell\leq k-2} 2^{\ell}\|\dot{\Delta}_\ell \na
u\|_{L^2}\lesssim c_k 2^k\|\na u\|_{L^2},
\eeq
so that we infer
\begin{align*}
\|\mathcal{R}^2_q\|_{L^p}\lesssim &\sum_{\vert
q-k\vert\leq 4}c_k 2^{2k\left(1-\f1p\right)}\|\na u\|_{L^2}\|\dot\Delta_ku\|_{L^2}\\
\lesssim &\sum_{\vert
q-k\vert\leq 4}c_k^22^{k\left(1-\f2p\right)}\|\na u\|_{L^2}^2
\lesssim  c_q^22^{q\left(1-\f2p\right)}\|\na u\|_{L^2}^2.
\end{align*}

Notice that $\mathcal{R}^3_q=-\sum\limits_{k\geq
q-3}\dot S_{k+2}\dot\Delta_q\partial_j u\dot\Delta_k u^j$, one has
\begin{align*}
\Vert \mathcal{R}^3_q\Vert_{L^p}
&\lesssim \|\dot\Delta_q\na u\|_{L^\infty}\sum_{k\geq
q-3} \|\dot\Delta_k u\|_{L^p}\\
&\lesssim 2^q
\Vert\dot\Delta_q \nabla u\Vert_{L^2} \sum_{k\geq q-3}c_k2^{-\frac{2}{p}k}
\Vert\na u\Vert_{L^2}\lesssim c_q^22^{q\left(1-\f2p\right)}\|\nabla u\|_{L^2}^2.
\end{align*}

For the last term  $\mathcal{R}^4_q$ in \eqref{comma-1}, we use the property of spectral localization of the Littlewood-Paley
decomposition to write $\mathcal{R}^4_q= \sum_{\vert
k-q\vert\leq 4}[\dot\Delta_{q},\dot{S}_{k-1}u^j]\dot\Delta_k\partial_j u$, from which, Lemma \ref{lem-commutator-1} and \eqref{S2eq1}, we infer
$$
\begin{aligned}
\|\mathcal{R}^4_q\|_{L^p}
&\lesssim
\sum_{|k-q|\leq4}2^{k-q}\|\dot{S}_{k-1}\nabla u\|_{L^\infty}\|\dot\Delta_{k}u\|_{L^p}\\
&\lesssim 2^{-q}
\sum_{|k-q|\leq4}c_k^2 2^{2k\left(1-\f1p\right)}\|\na u\|_{L^2}^2\lesssim
c_q^22^{q\left(1-\f2p\right)}\|\nabla u\|_{L^2}^2.
\end{aligned}
$$

By summarizing the above estimates, we arrive at \eqref{2.11}, which ends the proof of Lemma \ref{lem2.3}.
\end{proof}



\begin{lem}\label{lem2.5}
{\sl Let  $p\in [2,\infty[,$ $\lambda\in [1,\infty[,$ $a\in
\dot{B}_{\lambda,2}^{\frac{2}{\lambda}}$ and $f\in L^2.$
Then there holds
\begin{equation}\label{2.26}
\begin{split}
\sum_{q\in\mathbb{Z}}2^{q\bigl(-1+\f2p\bigr)}\|[\dot{\Delta}_{q},a]f\|_{L^p}\lesssim
\|a\|_{\dot{B}_{\lambda,2}^{\frac{2}{\lambda}}}
\|f\|_{L^2}.
\end{split}
\end{equation}
}
\end{lem}

\begin{proof} We
first get, by applying Bony's decomposition \eqref{bony}, that
\begin{equation}\label{2.27we}
\begin{split}
[\dot{\Delta}_{q},a]f&=\dot{\Delta}_{q}R(a,f)+\dot{\Delta}_{q}T_{f}a-T'_{\dot{\Delta}_{q}f}a-[\dot{\Delta}_{q},T_a]f \eqdefa\sum_{i=1}^4 \cI_q^i.
\end{split}
\end{equation}

In case  $\lambda\le p$, we deduce from
 Lemma \ref{lem2.1} that
\begin{equation*}
\begin{split}
 \|\cI_q^1\|_{L^p}
&
\lesssim
 2^{q}\sum_{k\geq q-3}
\|\dot{\Delta}_ka\|_{L^p}
\|{\widetilde{\dot{\Delta}}}_kf\|_{L^2}\lesssim  2^{q}
\sum_{k\geq q-3}c_{k}^2 2^{-\f2pk}
\|a\|_{\dot{B}_{p,2}^{\frac{2}{p}}}
\|f\|_{L^2}\\
&\lesssim  c_{q}^22^{q\bigl(1-\f2p\bigr)}
\|a\|_{\dot{B}_{\lambda,2}^{\frac{2}{\lambda}}}
\|f\|_{L^2}\lesssim c_{q}^22^{q\bigl(1-\f2p\bigr)}
\|a\|_{\dot{B}_{\lambda,2}^{\frac{2}{\lambda}}}
\|f\|_{L^2}.
\end{split}
\end{equation*}
While for $p<\lambda<\infty$, one has $\la>2$ and $\f{2\la}{2+\la}<2\leq p,$ so that  we infer
\begin{align*}
\|\cI_q^1\|_{L^p}
\lesssim &
2^{q\bigl(1+\frac{2}{\lambda}-\f2p\bigr)}
\sum_{k\geq q-3}
\|\dot{\Delta}_k a\|_{L^{\lambda}}\|{\widetilde{\dot\Delta}}_kf
\|_{L^2}\\
&
\lesssim
2^{q\bigl(1+\frac{2}{\lambda}-\f2p\bigr)}
\sum_{k\geq q-3}c_{k}^2
2^{-\frac{2}{\lambda}k}
\|a\|_{\dot{B}_{\lambda,2}^{\frac{2}{\lambda}}}
\|f\|_{L^2}\lesssim c_{q}^22^{q\bigl(1-\f2p\bigr)}
\|a\|_{\dot{B}_{\lambda,2}^{\frac{2}{\lambda}}}
\|f\|_{L^2}.
\end{align*}

Similarly in case $\lambda\le p,$ we have
\begin{align*}
\|\cI_q^{3}\|_{L^p}
&\lesssim
\sum_{k\geq q-3}\|\dot{S}_{k+2}\dot{\Delta}_{q}f\|_{L^\infty}
\|\dot{\Delta}_k a\|_{L^p}
\\
&\lesssim
2^{q}
\|\dot{\Delta}_{q}f\|_{L^2}
\sum_{k\geq q-3} c_k2^{-\f2pk}\|a\|_{\dot{B}_{p,2}^{\frac{2}{p}}}
\lesssim c_{q}^22^{q\bigl(1-\f2p\bigr)}
\|a\|_{\dot{B}_{\lambda,2}^{\frac{2}{\lambda}}}
\|f\|_{L^2}.
\end{align*}
And in case $p\le\lambda<\infty$, one has
\begin{equation*}
\begin{split}
\|\cI_q^{3}\|_{L^p}
&\lesssim
\sum_{k\geq q-3}
\|\dot{S}_{k+2}\dot{\Delta}_{q}f\|_{L^{\frac{p\lambda}{\lambda-p}}}
\|\dot{\Delta}_k a\|_{L^\lambda}
\\&
\lesssim c_q 2^{q\bigl(1+\f2\lambda-\f2p\bigr)}
\sum_{k\geq q-3}c_k 2^{-\frac{2}{\lambda}k}\|a\|_{\dot{B}_{\lambda,2}^{\frac{2}{\lambda}}}
\|f\|_{L^2}
\lesssim  c_{q}^22^{q\bigl(1-\f2p\bigr)}
\|a\|_{\dot{B}_{\lambda,2}^{\frac{2}{\lambda}}}
\|f\|_{L^2}.
\end{split}
\end{equation*}

On the other hand, we deduce from  Lemma \ref{lem2.1} that
\begin{equation*}
\begin{split}
\|\cI_q^2\|_{L^p}&\lesssim\sum_{|q-k|\leq4} \|\dot{S}_{k-1}f\|_{L^\infty}\|\dot{\Delta}_k a\|_{L^p}\\
&\lesssim
\sum_{|q-k|\leq4} c_k^22^{k\left(1-\f2p\right)}\|a\|_{\dot{B}_{p,2}^{\frac{2}{p}}}\|f\|_{L^2}\\
&\lesssim c_{q}^22^{q\bigl(1-\f2p\bigr)}
\|a\|_{\dot{B}_{p,2}^{\frac{2}{p}}}\|f\|_{L^2}\lesssim c_{q}^22^{q\bigl(1-\f2p\bigr)} \|a\|_{\dot{B}_{\lambda,2}^{\frac{2}{\lambda}}}
\|f\|_{L^2},
\end{split}
\end{equation*}
in case $\lambda\le p$.
While for $p\le\lambda<\infty$, we get, by a  similar estimate of $\cI_q^{3}$, that
\begin{equation*}
\begin{split}
\|\cI_q^2\|_{L^p}&\lesssim
\sum_{|q-k|\leq4} \|\dot{S}_{k-1}f\|_{L^{\frac{p\lambda}{\lambda-p}}}
\|\dot{\Delta}_k a\|_{L^\lambda}\\
&\lesssim
\sum_{|q-k|\leq4} c_k^22^{k\left(1-\f2p\right)}\|a\|_{\dot{B}_{\lambda,2}^{\frac{2}{\lambda}}}\|f\|_{L^2}\lesssim c_{q}^22^{q\bigl(1-\f2p\bigr)}
\|a\|_{\dot{B}_{\lambda,2}^{\frac{2}{\lambda}}}\|f\|_{L^2},
\end{split}
\end{equation*}

Finally it follows from Lemma \ref{lem-commutator-1} that
\begin{align*}
\|\cI_q^4\|_{L^p}
&\lesssim
\sum_{|q-k|\leq4} 2^{-q}
\|\nabla \dot{S}_{k-1} a\|_{L^\infty}\|\dot{\Delta}_{k}f\|_{L^p}\\
&\lesssim
2^{-q}\sum_{|q-k|\leq4} c_k^22^{k\left(2-\f2p\right)}\|a\|_{\dot{B}_{p,2}^{\frac{2}{p}}}\|f\|_{L^2}\\
&\lesssim c_{q}^22^{q\bigl(1-\f2p\bigr)}
\|a\|_{\dot{B}_{p,2}^{\frac{2}{p}}}\|f\|_{L^2}
\lesssim c_{q}^22^{q\bigl(1-\f2p\bigr)}
\|a\|_{\dot{B}_{\lambda,2}^{\frac{2}{\lambda}}}\|f\|_{L^2},
\end{align*}
in the case when $\lambda\le p.$
While for $p<\lambda$, we infer
\begin{equation*}
\begin{split}
\|\cI_q^4\|_{L^p}
&\lesssim
\sum_{|q-k|\leq4} 2^{-q}
\|\nabla \dot{S}_{k-1} a\|_{L^\lambda}
\|\dot{\Delta}_{k}f\|_{L^{\frac{p\lambda}{\lambda-p}}}
\\
&\lesssim
2^{-q}\sum_{|q-k|\leq4}c_k^22^{k\left(2-\f2p\right)}
\|a\|_{\dot{B}_{\lambda,2}^{\frac{2}{\lambda}}}\|f\|_{L^2}\lesssim c_{q}^22^{q\bigl(1-\f2p\bigr)}
\|a\|_{\dot{B}_{\lambda,2}^{\frac{2}{\lambda}}}\|f\|_{L^2}.
\end{split}
\end{equation*}

In view of \eqref{2.27we}, we achieve \eqref{2.26} by summarizing the above estimates. This completes the proof of the Lemma \ref{lem2.5}.
\end{proof}

\begin{lem}\label{lem2.6}
{\sl Let  $\lambda\in [1,\infty[$ and $p\in ]2,\infty[$ with
$\frac{1}{2}<\frac{1}{p}+\frac{1}{\lambda}\le1.$
 Let $g\in \dot{B}_{p,1}^{-1+\frac{2}{p}}$ and
 $f\in\dot B^{\frac{2}{\lambda}}_{\lambda,\infty}.$ Then we have
\begin{equation}\label{S2eq3}
\begin{split}
&\sum_{q\in\mathbb{Z}}2^{q\bigl(-1+\f2p\bigr)}
\|[\dot{\Delta}_q\mathbb{P}, f] g\|_{L^p}
\lesssim
\|g\|_{\dot{B}_{p,1}^{-1+\frac{2}{p}}}
\|f\|_{\dot B^{\frac{2}{\lambda}}_{\lambda,\infty}},\\
&
\sum_{q\in\mathbb{Z}}
\|[\dot{\Delta}_q\mathbb{P}, f] g\|_{L^2}
\lesssim
\|g\|_{\dot{B}_{p,1}^{-1+\frac{2}{p}}}
\|f\|_{\dot B^{\frac{2}{\lambda}}_{\lambda,\infty}}.
\end{split}
\end{equation}
}
\end{lem}

\begin{proof}
Similar to \eqref{2.27we}, we write
\begin{equation}\label{2.28}
\begin{split}
[\dot{\Delta}_q\mathbb{P}, f] g=&\dot{\Delta}_q\mathbb{P}R(f, g)+\dot{\Delta}_q\mathbb{P}T_{ g}f-T'_{\dot{\Delta}_q g}f-[\dot{\Delta}_q\mathbb{P},T_{f}] g\eqdefa \sum_{i=1}^4\cK_{q}^i.
\end{split}
\end{equation}
As $\frac{1}{2}<\frac{1}{p}+\frac{1}{\lambda}\le1,$  we deduce from Lemma \ref{lem2.1} that
\begin{align*}
\|\cK_{q}^1\|_{L^p}
&
\lesssim
2^{\frac{2}{\lambda}q}
\sum_{k\geq q-3}
\|\dot{\Delta}_k f\|_{L^\lambda}
\|\widetilde{\dot{\Delta}}_k g\|_{L^p}\\
&
\lesssim
2^{\frac{2}{\lambda}q}
\sum_{k\geq q-3} d_k2^{-k\left(\f2\la+\f2p-1\right)}\|f\|_{\dot B^{\frac{2}{\lambda}}_{\lambda,\infty}}
\|g\|_{\dot{B}_{p,1}^{-1+\frac{2}{p}}}\\
&
\lesssim d_q2^{q\bigl(1-\f2p\bigr)}
\|f\|_{\dot B^{\frac{2}{\lambda}}_{\lambda,\infty}}
\|g\|_{\dot{B}_{p,1}^{-1+\frac{2}{p}}}.
\end{align*}
The same argument  gives rise to
\begin{equation*}
\begin{split}
\|\cK_{q}^1\|_{L^2}
\lesssim d_q
\|f\|_{\dot B^{\frac{2}{\lambda}}_{\lambda,\infty}}
\|g\|_{\dot{B}_{p,1}^{-1+\frac{2}{p}}}.
\end{split}
\end{equation*}
Here and in all that follows, we always designate $\{d_q\}_{q\in\Z}$ to be a generic element of $\ell^1(\Z)$ so that $\sum_{q\in\Z}d_q=1.$

Thanks to the  support properties  of Fourier transform to the terms in
$\dot{S}_{k+2}\dot{\Delta}_q g$, we get
\begin{align*}
\|\cK_{q}^3\|_{L^p}
&\lesssim
\sum_{k\geq q-3}
\|\dot{S}_{k+2}\dot{\Delta}_q g\|_{L^{\infty}}
\|\dot{\Delta}_{k}f\|_{L^p}\lesssim
2^{\frac{2}{p}q}\|\dot{\Delta}_q g\|_{L^{p}}
\sum_{k\geq q-3} 2^{-\f2p k}\|f\|_{\dot B^{\frac{2}{p}}_{p,\infty}}\\
&\lesssim 2^{\frac{2}{p}q} d_q2^{q\bigl(1-\f2p\bigr)}2^{-\f2p q}
\|f\|_{\dot B^{\frac{2}{p}}_{p,\infty}}
\|g\|_{\dot{B}_{p,1}^{-1+\frac{2}{p}}}\\
&\lesssim d_q2^{q\bigl(1-\f2p\bigr)}
\|f\|_{\dot B^{\frac{2}{\lambda}}_{\lambda,\infty}}
\|g\|_{\dot{B}_{p,1}^{-1+\frac{2}{p}}},
\end{align*}
in the case when
$\lambda\le p.$
While if $p\le\lambda$, one has
\begin{equation*}
\begin{split}
\|\cK_{q}^3\|_{L^p}
&
\lesssim
2^{\frac{2}{\lambda}q}
\|\dot{\Delta}_q g\|_{L^{p}}
\sum_{k\geq q-3}
\|\dot{\Delta}_{k}f\|_{L^\lambda}\\
&
\lesssim d_q2^{q\bigl(1-\f2p\bigr)}
\|f\|_{\dot B^{\frac{2}{\lambda}}_{\lambda,\infty}}
\|g\|_{\dot{B}_{p,1}^{-1+\frac{2}{p}}}.
\end{split}
\end{equation*}
Similarly, one has
\begin{align*}
\|\cK_{q}^3\|_{L^2}
&\lesssim d_q
\|f\|_{\dot B^1_{2,\infty}}\|g\|_{\dot{B}_{p,1}^{-1+\frac{2}{p}}}
\lesssim d_q
\|f\|_{\dot B^{\frac{2}{\lambda}}_{\lambda,\infty}}
\|g\|_{\dot{B}_{p,1}^{-1+\frac{2}{p}}},
\end{align*}
in the case when  $\lambda\le2.$
And for $2\le\lambda<\infty$, we have
\begin{equation*}
\begin{split}
\|\cK_{q}^2\|_{L^2}
&\lesssim
\sum_{k\geq q-3}
\|\dot{S}_{k+2}\dot{\Delta}_q g\|_{L^{\frac{2\lambda}{\lambda-2}}}
\|\dot{\Delta}_{k}f\|_{L^\lambda}
\lesssim
2^{q\bigl(\frac{2}{p}+\frac{2}{\lambda}-1\bigr)}
\|\dot{\Delta}_q g\|_{L^{p}}
\sum_{k\geq q-3}
\|\dot{\Delta}_{k}f\|_{L^\lambda}
\\
&\lesssim
2^{q\bigl(\frac{2}{p}+\frac{2}{\lambda}-1\bigr)}
d_q2^{q(1-\f2p)}
\sum_{k\geq q-3}2^{-\frac{2}{\lambda}k}\|f\|_{\dot B^{\frac{2}{\lambda}}_{\lambda,\infty}}
\|g\|_{\dot{B}_{p,1}^{-1+\frac{2}{p}}}\lesssim d_q \|f\|_{\dot B^{\frac{2}{\lambda}}_{\lambda,\infty}}
\|g\|_{\dot{B}_{p,1}^{-1+\frac{2}{p}}}.
\end{split}
\end{equation*}

Notice that
\begin{equation*}
\begin{split}
\|\dot{S}_{k-1} g\|_{L^\infty}
\leq
\sum_{\ell\leq k-2}2^{\frac{2}{p}\ell}\|\dot{\Delta}_{\ell} g\|_{L^p}\lesssim \|g\|_{\dot{B}_{p,1}^{-1+\frac{2}{p}}}\sum_{\ell\leq k-2}d_{\ell} 2^{\ell}\lesssim d_k2^k\|g\|_{\dot{B}_{p,1}^{-1+\frac{2}{p}}},
\end{split}
\end{equation*}
so that one has
\begin{align*}
\|\cK_{q}^2\|_{L^p}
&\lesssim
\sum_{|q-k|\leq4}
\|\dot{\Delta}_{k}f\|_{L^p}\|\dot{S}_{k-1} g\|_{L^\infty}
\\&
\lesssim  d_q2^{q\bigl(1-\f2p\bigr)}
\|f\|_{\dot B^{\frac{2}{p}}_{p,\infty}}
\|g\|_{\dot{B}_{p,1}^{-1+\frac{2}{p}}}\lesssim  d_q2^{q\bigl(1-\f2p\bigr)}
\|f\|_{\dot B^{\frac{2}{\lambda}}_{\lambda,\infty}}
\|g\|_{\dot{B}_{p,1}^{-1+\frac{2}{p}}},
\end{align*}
in case $\lambda\le p.$
For $p\le\lambda<\infty,$ one has
\beno
\|\dot{S}_{k-1} g\|_{L^{\frac{p\lambda}{\lambda-p}}}\lesssim \sum_{\ell\leq k-2}
2^{\frac{2}{\lambda}\ell}\|\dot{\Delta}_{\ell} g\|_{L^p}
\lesssim d_k2^{k\left(1+\f2\la-\f2p\right)}\|g\|_{\dot{B}_{p,1}^{-1+\frac{2}{p}}},
\eeno
from which, we infer
\begin{align*}
\|\cK_{q}^2\|_{L^p}
&\lesssim
\sum_{|q-k|\leq4}
\|\dot{\Delta}_{k}f\|_{L^\lambda}
\|\dot{S}_{k-1} g\|_{L^{\frac{p\lambda}{\lambda-p}}}
\\
&\lesssim d_q
2^{q\bigl(1-\f2p\bigr)}
\|f\|_{\dot B^{\frac{2}{\lambda}}_{\lambda,\infty}}
\|g\|_{\dot{B}_{p,1}^{-1+\frac{2}{p}}}.
\end{align*}
Similarly, one has
\begin{align*}
\|\cK_{q}^2\|_{L^2}
&\lesssim d_q
\|f\|_{\dot B^1_{2,\infty}}\|g\|_{\dot{B}_{p,1}^{-1+\frac{2}{p}}}\lesssim d_q
\|f\|_{\dot B^{\frac{2}{\lambda}}_{\lambda,\infty}}
\|g\|_{\dot{B}_{p,1}^{-1+\frac{2}{p}}},
\end{align*}
in case $\lambda\le2.$
And for $2\le\lambda<\infty$, we have
\beno
\|\dot{S}_{k-1} g\|_{L^{\frac{2\lambda}{\lambda-2}}}\lesssim \sum_{\ell\leq k-2}
2^{(\frac{2}{\lambda}+\frac{2}{p}-1)\ell}\|\dot{\Delta}_{\ell} g\|_{L^p}
\lesssim d_k 2^{\f{2}\la k}\|g\|_{\dot{B}_{p,1}^{-1+\frac{2}{p}}},
\eeno
from which, we infer
\begin{align*}
\|\cK_{q}^2\|_{L^2}
&\lesssim
\sum_{|q-k|\leq4}
\|\dot{\Delta}_{k}f\|_{L^\lambda}
\|\dot{S}_{k-1} g\|_{L^{\frac{2\lambda}{\lambda-2}}}
\\&\lesssim d_q
\|f\|_{\dot B^{\frac{2}{\lambda}}_{\lambda,\infty}}
\|g\|_{\dot{B}_{p,1}^{-1+\frac{2}{p}}}.
\end{align*}

Observing that
\beno
\|\nabla\dot S_{k-1}f\|_{L^\infty}\lesssim \sum_{\ell\le k-2}2^{k\left(1+\frac{2}{\lambda}\right)}
\|\dot\Delta_{\ell}f\|_{L^\lambda}\lesssim 2^k\|f\|_{\dot B^{\frac{2}{\lambda}}_{\lambda,\infty}},
\eeno
so that we deduce from Lemma \ref{lem-commutator-1} that
\begin{equation*}
\begin{split}
\|\cK_{q}^4\|_{L^p}
&\lesssim 2^{-q}
\sum_{|q-k|\le4}
\|\nabla\dot S_{k-1}f\|_{L^\infty}
\|\dot\Delta_k g\|_{L^p}
\\&
\lesssim d_q2^{q\bigl(1-\f2p\bigr)}
\|f\|_{\dot B^{\frac{2}{\lambda}}_{\lambda,\infty}}
\|g\|_{\dot{B}_{p,1}^{-1+\frac{2}{p}}}.
\end{split}
\end{equation*}
While due to $\frac{1}{2}<\frac{1}{p}+\frac{1}{\lambda}$ and $p<\infty,$ we get,
by using the inequality \eqref{commutator-compact-0}, that
\begin{align*}
\|\cK_{q}^4\|_{L^2}
&\lesssim
2^{-q}\sum_{|q-k|\le4}
\|\nabla\dot S_{k-1}f\|_{L^{\frac{2p}{p-2}}}
\|\dot\Delta_k g\|_{L^p}
\\&
\lesssim
\sum_{|q-k|\le4}2^{-q}2^{\f2p k}
\|\dot\Delta_k g\|_{L^p}\|f\|_{\dot B^{\frac{2}{\lambda}}_{\lambda,\infty}}
\\&
\lesssim d_q
\|f\|_{\dot B^{\frac{2}{\lambda}}_{\lambda,\infty}}
\|g\|_{\dot{B}_{p,1}^{-1+\frac{2}{p}}}.
\end{align*}

By summarizing the above estimates, we arrive at \eqref{S2eq3}.
This completes the proof of Lemma \ref{lem2.6}.
\end{proof}

\begin{lem}\label{lem2.7}
{\sl Let $p\in [2,\infty[$ and $\lambda\in [1,\infty[.$
Let $u\in \dot{B}_{p,1}^{1+\frac{2}{p}}\cap H^1$  and
$a\in \dot B^{\frac{2}{\lambda}}_{\lambda,2}.$ Then we have
\begin{equation}\label{2.29}
\begin{split}
\sum_{q\in\mathbb{Z}}2^{q\bigl(-1+\f2p\bigr)}
\|[\dot{\Delta}_q\mathbb{P},\dot{S}_m a]\Delta u\|_{L^p}
\lesssim
2^{m}\|a\|_{\dot B^{\frac{2}{\lambda}}_{\lambda,2}}
\Bigl(\|\nabla u\|_{L^2}+\|u\|_{L^2}^{\frac{1}{2}}
\|u\|_{\dot{B}_{p,1}^{1+\frac{2}{p}}}^{\frac{1}{2}}\Bigr).
\end{split}
\end{equation}
}
\end{lem}
\begin{proof} Once again similar to \eqref{2.27we}, we write
\begin{equation*} 
\begin{split}
[\dot{\Delta}_q\mathbb{P}, \dot{S}_m a]\Delta u=&\dot{\Delta}_q\mathbb{P} R(\dot{S}_m a,\Delta u)+\dot{\Delta}_q\mathbb{P}T_{\Delta u}\dot{S}_m a\\
&-T'_{\dot{\Delta}_q\Delta u}\dot{S}_m a-[\dot{\Delta}_q\mathbb{P},T_{\dot{S}_m a}]\Delta u\eqdefa \sum_{i=1}^4\cL_{q}^i.
\end{split}
\end{equation*}

It follows from  Lemma \ref{lem2.1} that
\begin{align*}
\sum_{q\in\mathbb{Z}}2^{q\bigl(-1+\f2p\bigr)}
\|\cL_{q}^1\|_{L^p}
&
\lesssim \sum_{q\in\mathbb{Z}}2^{q}\sum_{k\geq q-3}2^{-k}\|\dot{\Delta}_k \nabla^2\dot{S}_m a\|_{L^2}\|\dot{\widetilde{{\Delta}}}_{k}\nabla u\|_{L^2}\\
&
\lesssim \sum_{q\in\mathbb{Z}}2^{q}\sum_{k\geq q-3}c_k^22^{-k}\|\nabla^2\dot{S}_m a\|_{L^2}\|\nabla u\|_{L^2}\\
&
\lesssim
2^{m}\|\nabla a\|_{L^2}\|\nabla u\|_{L^2}\lesssim
2^m\|a\|_{\dot B^{\frac{2}{\lambda}}_{\lambda,2}}\|\nabla u\|_{L^2},
\end{align*}
in case $\lambda\le2.$
If $2\le\lambda<\infty,$ we have
\begin{equation*}
\begin{split}
\sum_{q\in\mathbb{Z}}2^{q\bigl(-1+\f2p\bigr)}
\|\cL_{q}^1\|_{L^p}
\le
\sum_{q\in\mathbb{Z}}2^{q\frac{2}{\lambda}}
\|\dot{\Delta}_q\mathbb{P}R(\dot{S}_m a,\Delta u)\|_{L^{\frac{2\lambda}{\lambda+2}}}
\lesssim
2^m\|a\|_{\dot B^{\frac{2}{\lambda}}_{\lambda,2}}\|\nabla u\|_{L^2}.
\end{split}
\end{equation*}

Notice that $\|\dot{S}_{k-1}\Delta u\|_{L^\infty}\lesssim c_k2^{2k}\|\nabla u\|_{L^2},$ we infer
\begin{align*}
\sum_{q\in\mathbb{Z}}2^{q\bigl(-1+\f2p\bigr)}
\|\cL_{q}^2\|_{L^p}
&\lesssim
\sum_{q\in\mathbb{Z}}2^{q\bigl(-1+\f2p\bigr)}\sum_{|q-k|\leq4}
\|\dot{S}_{k-1}\Delta u\|_{L^\infty}\|\dot{\Delta}_k\dot{S}_m a\|_{L^p}
\\
&\lesssim\sum_{q\in\mathbb{Z}}2^{q\bigl(-1+\f2p\bigr)}\sum_{|q-k|\leq4} c_k^2 2^{k\left(1-\f2p\right)}
\|\na \dot{S}_m a\|_{\dot B^{\frac{2}{p}}_{p,2}}\|\nabla u\|_{L^2}\\
&\lesssim
2^{m}\|a\|_{\dot B^{\frac{2}{p}}_{p,2}}\|\nabla u\|_{L^2}\lesssim
2^{m}\|a\|_{\dot B^{\frac{2}{\lambda}}_{\lambda,2}}\|\nabla u\|_{L^2},
\end{align*}
in case $\lambda\le p.$
When $p\le\lambda$, we have
\beno
\|\dot{S}_{k-1}\Delta u\|_{L^{\frac{p\lambda}{\lambda-p}}}\lesssim \sum_{\ell\leq k-2}2^{2\ell(1+\frac{1}{\lambda}-\frac{1}{p})}
\|\dot{\Delta}_{\ell}\nabla u\|_{L^2}\lesssim c_k2^{2k\bigl(1+\frac{1}{\lambda}-\frac{1}{p}\bigr)}\|\nabla u\|_{L^2},
\eeno
so that one has
\begin{align*}
\sum_{q\in\mathbb{Z}}2^{q\bigl(-1+\f2p\bigr)}
\|\cL_{q}^2\|_{L^p}
\lesssim&
\sum_{q\in\mathbb{Z}}2^{q\bigl(-1+\f2p\bigr)}
\sum_{|q-k|\leq4}
\|\dot{S}_{k-1}\Delta u\|_{L^{\frac{p\lambda}{\lambda-p}}}
\|\dot{\Delta}_k\dot{S}_m a\|_{L^\lambda}
\\
\lesssim&
\sum_{q\in\mathbb{Z}}2^{q\bigl(-1+\f2p\bigr)}
\sum_{|q-k|\leq4}c_k^2 2^{k\bigl(1-\f2p\bigr)}\|\nabla \dot{S}_ma\|_{\dot B^{\frac{2}{\lambda}}_{\lambda,2}}\|\nabla u\|_{L^2}\\
\lesssim&
2^{m}\|a\|_{\dot B^{\frac{2}{\lambda}}_{\lambda,2}}\|\nabla u\|_{L^2}.
\end{align*}

Considering the support properties to the Fourier transform of the terms in $\dot{S}_{k+2}\dot{\Delta}_q\Delta u$, we deduce
\begin{align*}
\sum_{q\in\mathbb{Z}}2^{q\bigl(-1+\f2p\bigr)}
\|\cL_{q}^3\|_{L^p}
&\lesssim
\sum_{q\in\mathbb{Z}}2^{q\bigl(-1+\f2p\bigr)}
\sum_{k\geq q-3}
\|\dot{\Delta}_q\Delta u\|_{L^{\infty}}
\|\dot{\Delta}_{k}\dot{S}_m a\|_{L^p}
\\&
\lesssim
\sum_{q\in\mathbb{Z}}c_q2^{q\bigl(1+\f2p\bigr)}\|\nabla u\|_{L^2}
\sum_{k\geq q-3}\|\dot{\Delta}_{k}\dot{S}_m a\|_{L^p}
\\
&\lesssim
2^{m}\|\nabla u\|_{L^2}\|a\|_{\dot B^{\frac{2}{p}}_{p,2}}\lesssim
2^{m}\|\nabla u\|_{L^2}\|a\|_{\dot B^{\frac{2}{\lambda}}_{\lambda,2}},
\end{align*}
in case $\lambda\le p.$ If $p\le\lambda$, we have
\begin{align*}
\sum_{q\in\mathbb{Z}}2^{q\bigl(-1+\f2p\bigr)}
\|\cL_{q}^3\|_{L^p}
&\lesssim
\sum_{q\in\mathbb{Z}}2^{q\bigl(-1+\f2p\bigr)}
\sum_{k\geq q-3}
\|\dot{\Delta}_q\Delta u\|_{L^{\frac{p\lambda}{\lambda-p}}}
\|\dot{\Delta}_{k}\dot{S}_m a\|_{L^\lambda}
\\&\lesssim
\sum_{q\in\mathbb{Z}}\sum_{k\geq q-3}2^{(q-k)(1+\frac{2}{\lambda})}
\|\dot{\Delta}_q\nabla u\|_{L^2}
2^{k\frac{2}{\lambda}}\|\dot{\Delta}_{k}\nabla\dot{S}_m a\|_{L^\lambda}
\\&\lesssim
2^{m}\|\nabla u\|_{L^2}\|a\|_{\dot B^{\frac{2}{\lambda}}_{\lambda,2}}.
\end{align*}

Finally we deduce from Lemma \ref{lem-commutator-1} that
\begin{align*}
\sum_{q\in\mathbb{Z}}2^{q\bigl(-1+\f2p\bigr)}
\|\cL_{q}^4\|_{L^p}
&
\lesssim
\sum_{q\in\mathbb{Z}}2^{q\bigl(-2+\f2p\bigr)}\sum_{|k-q|\leq4}
\|\nabla \dot{S}_{k-1}  \dot{S}_{m} a\|_{L^\infty}
2^{2k}\|\dot{\Delta}_k u\|_{L^p}
\\&
\lesssim
\sum_{q\in\mathbb{Z}}\sum_{|k-q|\leq4}2^{(-2+\frac{2}{p})(q-k)}
2^{k\frac{2}{p}}\|\dot{\Delta}_k\nabla u\|_{L^p}
\|\nabla \dot{S}_{m} a\|_{\dot B^{\frac{2}{\lambda}}_{\lambda,1}}
\\&
\lesssim
2^{m}\|a\|_{\dot B^{\frac{2}{\lambda}}_{\lambda,\infty}}
\|u\|_{\dot{B}_{p,1}^{\frac{2}{p}}}.
\end{align*}
While for any $N$ we observe that
\begin{align*}
 \|u\|_{\dot{B}_{p,1}^{\frac{2}{p}}}
 \lesssim \sum_{q\leq N} 2^q\|\dot{\Delta}_{q}u\|_{L^2}+\sum_{q\geq N} 2^{q(\frac{2}{p}+1)}2^{-q}\|\dot{\Delta}_{q}u\|_{L^p}
 \lesssim  2^N\|u\|_{L^2}+2^{-N}\|u\|_{\dot{B}^{1+\frac{2}{p}}_{p,\infty}}.
\end{align*}
 By taking  $N\in \mathbb{Z}$ in the above inequality so that
\begin{equation*}
\begin{split}
2^{2N}\approx \|u\|_{L^2}^{-1}\|u\|_{\dot{B}^{1+\frac{2}{p}}_{p,\infty}},
\end{split}
\end{equation*}
we obtain
\begin{equation*} 
\begin{split}\|u\|_{\dot{B}_{p,1}^{\frac{2}{p}}}\lesssim\|u\|_{L^2}^{\frac{1}{2}}\|u\|_{\dot{B}_{p,1}^{1+\frac{2}{p}}}^{\frac{1}{2}},
\end{split}
\end{equation*}
As a result, it comes out
\beno
\sum_{q\in\mathbb{Z}}2^{q\bigl(-1+\f2p\bigr)}
\|\cL_{q}^4\|_{L^p}\lesssim
2^{m}\|a\|_{\dot B^{\frac{2}{\lambda}}_{\lambda,\infty}}\|u\|_{L^2}^{\frac{1}{2}}\|u\|_{\dot{B}_{p,1}^{1+\frac{2}{p}}}^{\frac{1}{2}}.
\eeno

By summarizing the above estimates, we arrive at \eqref{2.29}, which completes the proof of this lemma.
\end{proof}

By taking the space divergence to the momentum equation of \eqref{1.3}, we find
\beq \label{S2eq2}
\dive\bigl((1+a)\nabla\Pi\bigr)=\dive(a\Delta u)-\dive(u\cdot\nabla u).
\eeq
Let us start  with the estimate of  $\|\nabla \Pi\|_{L^2}.$

\begin{lem}\label{lem2.4}
{\sl Let  $p\in [2,+\infty[$ and $\lambda\in [1, +\infty[$ satisfy
$\frac{1}{2}<\frac{1}{p}+\frac{1}{\lambda}\le1.$
Let $a\in \dot B^{\frac{2}{\lambda}}_{\lambda,2}\cap L^\infty,$ $\nabla \Pi\in L^2$, and $u\in \dot{H}^1$ with $\dive\,u=0$
satisfy the equation \eqref{S2eq2}. Then under the assumption that $1+a\geq c,$ for any $m \in \mathbb{Z}^+$,  there holds
\begin{equation}\label{2.18}
\|\nabla \Pi\|_{L^2}
\lesssim
\|a-\dot{S}_m a\|_{\dot B^{\frac{2}{\lambda}}_{\lambda,2}}
\|u\|_{\dot{B}_{p,1}^{1+\frac{2}{p}}}
+2^{m}\|\nabla u\|_{L^{2}}
\|a\|_{\dot B^{\frac{2}{\lambda}}_{\lambda,\infty}\cap{L^\infty}}
+\|\nabla u\|_{L^2}^2.
\end{equation}
}
\end{lem}
\begin{proof} In view of \eqref{S2eq2}, for any $m \in \mathbb{Z}^+$, we write
\begin{equation*}
\begin{split}
\dive((1+a)\nabla\Pi)&=\dive((a-\dot{S}_m a)\Delta u)+\dive(\dot{S}_m a\Delta u)-\dive(u\cdot\nabla u).
\end{split}
\end{equation*}
By applying Bony's decomposition \eqref{bony} and $\dive u=0,$ we obtain
\begin{equation}\label{2.19}
\begin{split}
\dive((1+a)\nabla\Pi)=&\dive T_{\Delta u}(a-\dot{S}_m a)+\dive R(\Delta u, a-\dot{S}_m a)+T_{\nabla(a-\dot{S}_m a)} \Delta u\\
&\quad+\dive(\dot{S}_m a\Delta u)-\dive(u\cdot\nabla u).
\end{split}
\end{equation}
Then we get, by taking $L^2$ inner product of \eqref{2.19} with $\Pi$ and using  $1+a\geq c$, that
\beq\label{2.20}
\begin{split}
c\|\nabla\Pi\|_{L^2}^2\leq&\|\nabla\Pi\|_{L^2}
\Bigl(\| T_{\Delta u}(a-\dot{S}_m a)\|_{L^2}+
\|R(\Delta u, a-\dot{S}_m a)\|_{L^{2}}\\
&+\|T_{\nabla(a-\dot{S}_m a)} \Delta u\|_{\dot{H}^{-1}}+
\|\dive(u\cdot\nabla)u\|_{\dot{H}^{-1}}+\|\dive(\dot{S}_ma\Delta u)\|_{\dot{H}^{-1}}\Bigr).
\end{split}\eeq
It follows from the law of product in Besov spaces, Proposition \ref{prop2.2},  and $\frac{1}{2}<\frac{1}{p}+\frac{1}{\lambda}\le1$ that
\begin{align*}
&\|T_{\Delta u}(a-\dot{S}_m a)\|_{L^2}
\lesssim
\|\Delta u\|_{\dot{B}_{p,1}^{-1+\frac{2}{p}}}
\|a-\dot{S}_m a\|_{\dot B^{1-\frac{2}{p}}_{\frac{2p}{p-2},2}}
\lesssim
\|u\|_{\dot{B}_{p,1}^{1+\frac{2}{p}}}
\|a-\dot{S}_m a\|_{\dot B^{\frac{2}{\lambda}}_{\lambda,2}},\\
&\|R(\Delta u, a-\dot{S}_m a)\|_{L^{2}}
\lesssim
\|R(\Delta u, a-\dot{S}_m a)
\|_{\dot B^{\frac{2}{p}+\frac{2}{\lambda}-1}_{\frac{p\lambda}{p+\lambda},2}}
\lesssim
\|u\|_{\dot{B}_{p,1}^{1+\frac{2}{p}}}
\|a-\dot{S}_m a\|_{\dot B^{\frac{2}{\lambda}}_{\lambda,\infty}},\\
&\|T_{\nabla(a-\dot{S}_m a)} \Delta u\|_{\dot{H}^{-1}}
\lesssim
\|\Delta u\|_{\dot{B}_{p,2}^{-1+\frac{2}{p}}}
\|\nabla(a-\dot{S}_m a)\|_{\dot B^{-\frac{2}{p}}_{\frac{2p}{p-2},\infty}}
\lesssim
\|u\|_{\dot{B}_{p,1}^{1+\frac{2}{p}}}
\|a-\dot{S}_m a\|_{\dot B^{\frac{2}{\lambda}}_{\lambda,\infty}}.
\end{align*}

While
again by using Bony's decomposition \eqref{bony} and $\dive u=0,$ we find
\begin{align*}
\|\dive(\dot{S}_ma\Delta u) \|_{\dot{H}^{-1}}&\lesssim
\|T_{\nabla\dot{S}_m a}\Delta u\|_{\dot{H}^{-1}}
+\|T_{\Delta u}\nabla\dot{S}_m a\|_{\dot{H}^{-1}}
+\|R(\dot{S}_m a,\Delta u)\|_{L^2}\\&\lesssim
\|\nabla\dot{S}_m a\|_{L^\infty}\|\nabla u\|_{L^{2}}
+\|\dot{S}_m a\|_{\dot B^{1+\f2\la}_{\la,\infty}}\|\D u\|_{\dot B^{\f2p-2}_{p,2}}\\
&\lesssim
2^{m}\|\nabla u\|_{L^{2}}
\bigl(\|a\|_{L^\infty}+\|a\|_{\dot B^{\frac{2}{\lambda}}_{\lambda,\infty}}\bigr).
\end{align*}

Finally again due to $\dive u=0,$ we have $\dive(u\cdot\nabla u)=\sum_{i,j=1}^2\p_iu_j\p_ju_i.$ As a result,
it comes out
\beno
\|\dive(u\cdot\nabla u)\|_{\dot{H}^{-1}}\lesssim\|T_{\partial u}\partial u\|_{\dot{H}^{-1}}+\|R(u,\na u)\|_{L^2}\lesssim\|\nabla u\|_{L^2}^2.
\eeno

By
substituting the above estimates  into \eqref{2.20}, we obtain \eqref{2.18}. This completes the proof of Lemma \ref{lem2.4}.
\end{proof}

\subsection{The estimate of the linearized equation}
The goal of  this subsection is to present some {\it a priori} estimates to the linearized equations of \eqref{1.2}.
We start with the following estimates concerning the solution of transport equation.

\begin{prop}\label{prop2.3}
{\sl Let  $m\in\mathbb{Z}$ and  $\lambda\in [1,\infty],$
let the convection velocity $u$ satisfy  $\nabla u\in L_{T}^1(L^\infty)\cap\widetilde L^1_T(\dot H^2)$ and $\dive u=0$. Then for
$a_{0}\in\dot B^{\frac{2}{\lambda}}_{\lambda,2},$  the following equation
\begin{equation}\label{2.30}
\begin{cases}
\partial_t a+u\cdot\nabla a=0,\\
a|_{t=0}=a_0.
\end{cases}
\end{equation} has a unique solution $a\in C([0,T]; \dot B^{\frac{2}{\lambda}}_{\lambda,2})$ so
 that for all $t\,\in(0,T]$,
\begin{align}\label{2.31q}
\|a\|_{\widetilde L^\infty_t(\dot B^{\frac{2}{\lambda}}_{\lambda,2})}
\lesssim
\|a_0\|_{\dot B^{\frac{2}{\lambda}}_{\lambda,2}}
\bigl(1+\|u\|_{\widetilde L^1_t(\dot H^2)})e^{C\|\nabla u\|_{L^1_t(L^\infty)}},
\end{align}
and
\begin{equation}\label{2.32}
\begin{aligned}
\|a-\dot S_ma&\|_{\widetilde L^\infty_t(\dot B^{\frac{2}{\lambda}}_{\lambda,2})}
\lesssim
\Bigl(\sum_{q\geq m}2^{\frac{4}{\lambda}q}
\|\dot\Delta_qa_0\|_{L^{\lambda}}^2\Bigr)^{\frac{1}{2}}\\
&+\|a_0\|_{\dot B^{\frac{2}{\lambda}}_{\lambda,2}}
\bigl(e^{C\|\nabla u\|_{L^1_t(L^\infty)}}-1\bigr)
+\|a_0\|_{\dot B^{\frac{2}{\lambda}}_{\lambda,2}}
\|u\|_{\widetilde L^1_t(\dot H^2)}e^{C\|\nabla u\|_{L^1_t(L^\infty)}}.
\end{aligned}
\end{equation}
}
\end{prop}
\begin{proof} The proof of the above proposition is similar  to that of Proposition 2.3 in \cite{AG2021}, for
completeness, we just outline it here.
We divide the proof into the following three cases: $\lambda>2$, $1<\lambda\leq2$, and $\lambda=1$.

\no {\bf Case 1: $\lambda>2$. }  By applying the dyadic operator $\dot\Delta_q$ to the transport equation of \eqref{2.30}, we write
$$
\partial_t\dot\Delta_qa+u\cdot\nabla\dot\Delta_qa=-[\dot\Delta_q,u\cdot\nabla]a,
$$
from which and $\dive u=0,$ we infer
\begin{equation}\label{est-transport-commu-1}
\|\dot\Delta_qa\|_{L^\infty_t(L^\lambda)}
\leq
\|\dot\Delta_qa_0\|_{L^\lambda}
+\int_0^t\|[\dot\Delta_q,u\cdot\nabla]a\|_{L^\lambda}\,d\tau.
\end{equation}
It follows from classical commutator's estimate (see Lemma 2.100 and Remark 2.102 in \cite{BCD})  that
\begin{equation*}
\|[\dot\Delta_q,u\cdot\nabla]a(t)\|_{L^\lambda} \lesssim c_{q}(t)2^{-\frac{2}{\lambda}q}\|\nabla\,u\|_{L^\infty} \|a\|_{\dot B^{\frac{2}{\lambda}}_{\lambda,2}},
\end{equation*}
 which along with \eqref{est-transport-commu-1} ensures that
\begin{equation}\label{est-transport-commu-3}
\|\dot\Delta_qa\|_{L^\infty_t(L^\lambda)}
\leq
\|\dot\Delta_qa_0\|_{L^\lambda}
+C2^{-\frac{2}{\lambda}q}\int_0^tc_{q}(\tau)\|\nabla\,u\|_{L^\infty} \|a\|_{\dot B^{\frac{2}{\lambda}}_{\lambda,2}}\,d\tau.
\end{equation}
Hence, one has
\begin{equation*}
\|a\|_{\widetilde L^\infty_t(\dot B^{\frac{2}{\lambda}}_{\lambda,2})}
\leq
\|a_0\|_{\dot B^{\frac{2}{\lambda}}_{\lambda,2}}+C\int_0^t \|\nabla\,u\|_{L^\infty} \|a\|_{\dot B^{\frac{2}{\lambda}}_{\lambda,2}}\,d\tau.
\end{equation*}
from which, we get, by applying Gronwall's inequality, that
\begin{equation}\label{est-transport-commu-4}
\|a\|_{\widetilde L^\infty_t(\dot B^{\frac{2}{\lambda}}_{\lambda,2})}
\leq
\|a_0\|_{\dot B^{\frac{2}{\lambda}}_{\lambda,2}}
e^{C\|\nabla u\|_{L^1_t(L^\infty)}}.
\end{equation}

On the other hand, considering $q \geq m$ in \eqref{est-transport-commu-3}, we infer
\begin{equation}\label{est-transport-commu-5}
\|a-\dot S_ma\|_{\widetilde L^\infty_t(\dot B^{\frac{2}{\lambda}}_{\lambda,2})}
\le
\Bigl(\sum_{q\geq m}2^{\frac{4}{\lambda}q}\|\dot\Delta_qa_0\|_{L^{\lambda}}^2\Bigr)^{\frac{1}{2}}+C\int_0^t \|\nabla\,u\|_{L^\infty} \|a\|_{\dot B^{\frac{2}{\lambda}}_{\lambda,2}}\,d\tau.
\end{equation}
By plugging \eqref{est-transport-commu-4} into \eqref{est-transport-commu-5}, we obtain
\begin{equation*}
\|a-\dot S_ma\|_{\widetilde L^\infty_t(\dot B^{\frac{2}{\lambda}}_{\lambda,2})}
\le
\Bigl(\sum_{q\geq m}2^{\frac{4}{\lambda}q}\|\dot\Delta_qa_0\|_{L^{\lambda}}^2\Bigr)^{\frac{1}{2}}
+\|a_0\|_{\dot B^{\frac{2}{\lambda}}_{\lambda,2}}
\Bigl(e^{C\|\nabla u\|_{L^1_t(L^\infty)}}-1\Bigr).
\end{equation*}

\no {\bf Case 2: $1<\lambda\leq2$. } It is easy to observe that
 $\partial_ja$ satisfies
\begin{align}\label{2.33}
\partial_t\partial_j a+u\cdot\nabla \partial_j a=-\partial_j u\cdot\nabla a.
\end{align}
Applying the operator $\dot{\Delta}_q$ to \eqref{2.33} gives
\begin{align*}
\partial_t\dot{\Delta}_q\partial_j a+u\cdot\nabla \partial_j\dot{\Delta}_q a=-\dot{\Delta}_q(\partial_j u\cdot\nabla a)
-[\dot{\Delta}_q, u\cdot\nabla]\partial_j a,
\end{align*}
from which and $\dive u=0,$ we infer
\begin{align*}
\|\dot{\Delta}_q\partial_j a\|_{L^\lambda}
\leq
\|\dot{\Delta}_q\partial_j a_0\|_{L^\lambda}
+\int_0^t\|\dot{\Delta}_q(\partial_j u\cdot\nabla a)\|_{L^\lambda}\mathrm{d}\tau
+\int_{0}^t\|[\dot{\Delta}_q, u\cdot\nabla]\partial_j a\|_{L^\lambda}\mathrm{d}\tau.
\end{align*}
Since $0\le\frac{2}{\lambda}-1<1$,  we deduce from classical commutator's estimate (see Lemma 2.100 and Remark 2.102 in \cite{BCD})
 that
\beq\label{S2eq5}
\|a\|_{\widetilde L^\infty_t(\dot B^{\frac{2}{\lambda}}_{\lambda,2})}
\leq
\|a_0\|_{\dot B^{\frac{2}{\lambda}}_{\lambda,2}}+\int_0^t\|\partial_j u\cdot\nabla a\|_{\dot B^{\frac{2}{\lambda}-1}_{\lambda,2}}\,d\tau
+C\int_0^t\|a\|_{\dot B^{\frac{2}{\lambda}}_{\lambda,2}}\|\nabla u\|_{L^\infty}\,d\tau.
\eeq
Notice that
\begin{align*}
\|\partial_j u\cdot\nabla a\|_{\dot B^{\frac{2}{\lambda}-1}_{\lambda,2}}
&\lesssim
\|T_{\partial_ju}\nabla a\|_{\dot B^{\frac{2}{\lambda}-1}_{\lambda,2}}
+\|R(\partial_ju,a)\|_{\dot B^{\frac{2}{\lambda}}_{\lambda,2}}
+\|T_{\nabla a}\partial_ju\|_{\dot B^{\frac{2}{\lambda}-1}_{\lambda,2}}
\\&
\lesssim
\|a\|_{\dot B^{\frac{2}{\lambda}}_{\lambda,2}}\|\nabla u\|_{L^\infty}
+\|u\|_{\dot H^2}
\|\nabla a\|_{\dot B^{-2+\frac{2}{\lambda}}_{\frac{2\lambda}
{2-\lambda},\infty}},
\end{align*}
and
\beno
\|\nabla a\|_{L^\infty_t(\dot B^{-2+\frac{2}{\lambda}}_{\frac{2\lambda}
{2-\lambda},\infty})}\lesssim \|\nabla a\|_{L^\infty_t(\dot B^0_{2,\infty})}\lesssim\|\nabla a_0\|_{\dot B^{0}_{2,\infty}}
e^{C\|\nabla u\|_{L^1_t(L^\infty)}}\lesssim \|a_0\|_{\dot B^{\frac{2}{\lambda}}_{\lambda,2}}e^{C\|\nabla u\|_{L^1_t(L^\infty)}},
\eeno
we infer
\begin{align*}
\|\partial_j u\cdot\nabla a\|_{\dot B^{\frac{2}{\lambda}-1}_{\lambda,2}}\le
\|a\|_{\dot B^{\frac{2}{\lambda}}_{\lambda,2}}\|\nabla u\|_{L^\infty}
+C\|a_0\|_{\dot B^{\frac{2}{\lambda}}_{\lambda,2}}\|u\|_{\dot H^2}
e^{C\|\nabla u\|_{L^1_t(L^\infty)}}.
\end{align*}
By inserting the above estimate into \eqref{S2eq5} and then applying Gronwall's inequality, we obtain \eqref{2.31q}.
Exactly along the same line, we deduce \eqref{2.32}.

\no {\bf Case 3: $\lambda=1$. }  Taking one more space derivative the equation \eqref{2.33} gives
$$
\partial_t\partial^2_{ij} a+u\cdot\nabla \partial^2_{ij} a
=-\partial_i u\cdot\nabla \partial_j a
-\partial_j u\cdot\nabla \partial_i a
-\partial^2_{ij}u\cdot\nabla a,
$$
from which, we deduce from classical commutator's estimate (see Lemma 2.100 and Remark 2.102 in \cite{BCD}), Bony's decomposition
and $\dv\,u=0$, that
\begin{align*}
\|\partial_{ij}^2a\|_{\widetilde L^\infty_t(\dot B^0_{1,2})}
&\le
\|\partial_{ij}^2a_0\|_{\dot B^0_{1,2}}
+C\int_0^t\|\nabla u\|_{L^\infty}
\|a\|_{\dot B^2_{1,2}}d\tau
+\|T_{\partial\nabla a}\partial u\|_{\widetilde L^\infty_t(\dot B^0_{1,2})}
+\|T_{\nabla a}\partial^2 u\|_{\widetilde L^1_t(\dot B^0_{1,2})}
\\&
\le
\|\partial_{ij}^2a_0\|_{\dot B^0_{1,2}}
+C\int_0^t\big(\|\nabla u\|_{L^\infty}
\|a\|_{\dot B^2_{1,2}}
+\|\nabla a\|_{L^2}\|u\|_{\dot H^2}\bigr)\,d\tau.
\end{align*}
Yet notice that
\beno
\|\nabla a\|_{\widetilde L^\infty_t(L^2)}\lesssim \|a_0\|_{\dot H^1}
e^{C\|\nabla u\|_{L^1_t(L^\infty)}}\lesssim \|a_0\|_{\dot B^{2}_{1,2}}e^{C\|\nabla u\|_{L^1_t(L^\infty)}},
\eeno
we find
\begin{align*}
\|\partial_{ij}^2a\|_{\widetilde L^\infty_t(\dot B^0_{1,2})}
\le
\|\partial_{ij}^2a_0\|_{\dot B^0_{1,2}}
+C\int_0^t\big(\|\nabla u\|_{L^\infty}\|a\|_{\dot B^2_{1,2}}
+\|a_0\|_{\dot B^{2}_{1,2}}
\|u\|_{\dot H^2}
e^{C\|\nabla u\|_{L^1_\tau(L^\infty)}}\bigr)\,d\tau.
\end{align*}
Applying Gronwall's inequality leads to \eqref{2.31q}. Similar argument yields \eqref{2.32}.
This completes the proof of Proposition \ref{prop2.3}.
\end{proof}

In order to prove the uniqueness part of Theorem \ref{thm1.1}
for $p=2,$ it is necessary for us to study the following Stokes system:
\begin{equation}\label{2.36}
\begin{cases}
&\rho_0\partial_tu-\Delta u+\nabla \Pi=f,\quad (t,x)\in \R^+\times\R^2,\\
&\nabla \cdot u=\dv\, g,\\
&\Delta\Phi=\dv\, g,\\
&u|_{t=0}=u_0.
\end{cases}
\end{equation}
\begin{prop}\label{prop2.4}
{\sl Let $p\in [2,+\infty[$ and $\lambda\in [1,+\infty[$
satisfy $\frac{1}{2}<\frac{1}{p}+\frac{1}{\lambda}\le1.$
 We assume that $1+a_0\eqdefa\f1{\rho_0}\geq \frac{1}{M}>0$ for some positive constant $M$,
$u \in \widetilde{L}^{\infty}_T( \dot{B}^{-1+\frac{2}{p}}_{p,1})
\cap L^1_T( \dot{B}^{1+\frac{2}{p}}_{p,1}),$ and $\grad\Pi\in
 L^1_T( \dot{B}^{-1+\frac{2}{p}}_{p,1})$  solves the system \eqref{2.36} for smooth enough $f$ and $g.$
Then there exists a large enough integer $m_0,$ which depends only on $\|a_0\|_{\dot B^{\f2\lambda}_{\lambda,2}},$ so that for any $m\geq m_0,$  one has
\begin{equation}\label{estimate-uniqueness-1q}
\begin{aligned}
&\|u\|_{\widetilde{L}_t^\infty(\dot{B}_{p,1}^{-1+\frac{2}{p}})}
+\|(u_t,\nabla^2u,\nabla\Pi)\|_{{L}_t^1(\dot{B}_{p,1}^{-1+\frac{2}{p}})} \lesssim \|\nabla\Phi\|_{\widetilde{L}_t^\infty(\dot{B}_{p,1}^{-1+\frac{2}{p}})
\cap {L}_t^1(\dot{B}_{p,1}^{1+\frac{2}{p}})}\\
&\qquad+
\bigl(2^m\sqrt{t}+2^{2m}t\bigr) \Bigl(\|(u_0,\nabla \Phi_0)\|_{\dot B^{-1+\frac{2}{p}}_{p,1}\cap\,L^2}+\|(f,\,\nabla \dv\, g,\,\nabla\Phi_t)\|_{L_t^1(\dot B^{-1+\frac{2}{p}}_{p,1}\cap L^2)}\Bigr).
\end{aligned}
\end{equation}
}
\end{prop}
\begin{proof}
Let $v\eqdefa u-\nabla\Phi.$ Then in view of \eqref{2.36}, one has
\begin{equation}\label{2.37}
\begin{cases}
\partial_tv-(1+a_0)(\Delta v-\nabla \Pi)=K,\qquad (t,x)\in \R^+\times\R^2,\\
\nabla \cdot v=0,\\
v|_{t=0}=v_0,
\end{cases}
\end{equation}
where $K\eqdefa (1+a_0)f-\partial_t\nabla\Phi+(1+a_0)\nabla \dv\, g$ and $1+a_0\eqdefa\f1{\rho_0}.$

We first get, by using  $L^2$ energy estimate to the equation \eqref{2.36}, that
\begin{equation}\label{basic-2.37aa}
\begin{split}
\|\sqrt{\rho_0}v\|_{L^\infty_t(L^2)}+\|\nabla v\|_{L^2_t(L^2)}
\lesssim
\|\sqrt{\rho_0}v_0\|_{L^2}
+\bigl\|(\sqrt{\rho_0}\partial_t\nabla\Phi,
f,\nabla\dv\, g)\bigr\|_{L^1_t(L^2)}.
\end{split}
\end{equation}

In what follows, we separate the proof of \eqref{estimate-uniqueness-1q} into the following two steps:

\no{\bf Step 1.}
The estimate of $\|v\|_{\widetilde{L}^\infty_T(\dot{B}^{-1+\frac{2}{p}}_{p,1})}
+\|v\|_{ L^1_T(\dot{B}^{1+\frac{2}{p}}_{p,1})}$.

By applying the operator $\mathbb{P}\dot{\Delta}_q$ to \eqref{2.37} and
using a commutator process, we write
\begin{equation*}
\begin{split}
\partial_t\dot{\Delta}_q v- (1+a_0)\Delta\dot{\Delta}_q v=[\mathbb{P}\dot{\Delta}_q, a_0](\Delta v-\nabla\Pi)+\mathbb{P}\dot{\Delta}_qK,
\end{split}
\end{equation*}
which can be equivalently written as
\begin{equation*}
\begin{split}
\rho_0\partial_t\dot{\Delta}_q v- \Delta\dot{\Delta}_q v=\rho_0[\mathbb{P}\dot{\Delta}_q, a_0](\Delta v-\nabla\Pi)+\rho_0\mathbb{P}\dot{\Delta}_qK,
\end{split}
\end{equation*}
By taking $L^2$ product of the above equation with $|\dot{\Delta}_q v|^{p-2}\dot{\Delta}_q v$ and using integration by parts, we obtain
\begin{align*}
&\frac{1}{p}\frac{d}{dt}\|\rho_0^{\frac{1}{p}}\dot{\Delta}_q v\|_{L^p}^p
-\int_{\R^2}\nabla\cdot(\nabla\dot{\Delta}_q v)\cdot|\dot{\Delta}_q v|^{p-2}\dot{\Delta}_q v\,dx\\
&\leq M\|\dot{\Delta}_q v\|_{L^p}^{p-1}\big(\|[\mathbb{P}\dot{\Delta}_q, a_0](\Delta v-\nabla\Pi)\|_{L^p}+\|\mathbb{P}\dot{\Delta}_qK\|_{L^p}\big).
\end{align*}
It follows from Lemma 8 in appendix of \cite{D2010} that
\begin{equation*}
\begin{split}
&-\int_{\R^2}\nabla\cdot(\nabla\dot{\Delta}_q v)\cdot|\dot{\Delta}_q v|^{p-2}\dot{\Delta}_q v\,dx\geq c_p2^{2q}\|\dot{\Delta}_{q}v\|_{L^p}^p,
\end{split}
\end{equation*}
so that we have
\begin{equation}\label{2.39-aa1}
\begin{split}
&\frac{d}{dt}\|\rho_0^{\frac{1}{p}}\dot{\Delta}_q v\|_{L^p}^p+pc_pM^{-1}2^{2q}\|\rho_0^{\frac{1}{p}}\dot{\Delta}_{j}v\|_{L^p}^p\\
&\lesssim\|\rho_0^{\frac{1}{p}}\dot{\Delta}_{j}v\|_{L^p}^{p-1}\big(\|[\mathbb{P}\dot{\Delta}_q, a_0](\Delta v-\nabla\Pi)\|_{L^p}+\|\mathbb{P}\dot{\Delta}_q K\|_{L^p}\big),
\end{split}
\end{equation}
which implies
\begin{equation*}
\begin{split}
\|\dot{\Delta}_q v(t)\|_{L^p} \lesssim  &e^{-c_p2^{2q}t}\|\dot{\Delta}_q v_0\|_{L^p}\\
&+
\int_{0}^{t}e^{-c_p2^{2q}(t-\tau)}\big(\|[\mathbb{P}\dot{\Delta}_q, a_0](\Delta v-\nabla\Pi)\|_{L^p}+\|\mathbb{P}\dot{\Delta}_q K\|_{L^p}\big)\,d\tau.
\end{split}
\end{equation*}
Then for any $r \in [1, +\infty]$,  by taking $L^r_t$ norm to the above inequality, we achieve
\begin{equation*}
\begin{split}
 \|\dot{\Delta}_q v\|_{L^r_t(L^p)} \lesssim & 2^{-\frac{2}{r}q}\|\dot{\Delta}_q v_0\|_{L^p}+2^{-\frac{2}{r}q}\big(\|[\mathbb{P}\dot{\Delta}_q, a_0](\Delta v-\nabla\Pi)\|_{L^1_t(L^p)}+\|\mathbb{P}\dot{\Delta}_q K\|_{L^1_t(L^p)}\big).
\end{split}
\end{equation*}
 By multiplying the above inequality by $2^{q\bigl(-1+\frac{2}{p}+\frac{2}{r}\bigr)}$ and taking the $\ell^1$ norm of the resulting inequalities, we find
\begin{equation*} 
\begin{split}
\|v\|_{\widetilde{L}_t^r(\dot{B}_{p,1}^{-1+\frac{2}{p}+\frac{2}{r}})}\lesssim & \|v_0\|_{\dot{B}_{p,1}^{-1+\frac{2}{p}}}+\|K\|_{\widetilde{L}_t^1(\dot{B}_{p,1}^{-1+\frac{2}{p}})}\\
&+\sum_{q\in\mathbb{Z}}2^{q\bigl(\frac{2}{p}-1\bigr)}\big(\|[\mathbb{P}\dot{\Delta}_q,a_0]\nabla \Pi\|_{L_t^1(L^p)}+\|[\mathbb{P}\dot{\Delta}_q,a_0]\Delta v\|_{L_t^1(L^p)}\big).
\end{split}
\end{equation*}
It follows from  Lemma \ref{lem2.5} that
\begin{equation*}
\begin{split}
\sum_{q\in\mathbb{Z}}2^{q\bigl(\frac{2}{p}-1\bigr)}\|[\mathbb{P}\dot{\Delta}_q,a_0]
\nabla \Pi\|_{L_t^1(L^p)}
\lesssim
\|a_0\|_{\dot{B}_{\lambda,2}^{\frac{2}{\lambda}}}\|\nabla \Pi\|_{L_t^1(L^2)}.
\end{split}
\end{equation*}
 While we get, by using Lemmas \ref{lem2.6}-\ref{lem2.7}, that
\begin{equation*}
\begin{split}
\sum_{q\in\mathbb{Z}}2^{q\bigl(\frac{2}{p}-1\bigr)}
\|[\mathbb{P}\dot{\Delta}_q,a_0]\Delta v\|_{L_t^1(L^p)}
\lesssim &
\|a_0-\dot{S}_m a_0\|_{\dot{B}_{\lambda,\infty}^{\frac{2}{\lambda}}}
\|v\|_{L_t^1(\dot{B}_{p,1}^{1+\frac{2}{p}})}\\
&+2^m\sqrt{t}\|a_0\|_{\dot{B}_{\lambda,2}^{\frac{2}{\lambda}}}
\bigl(\|\nabla v\|_{L_t^2(L^2)}
+\|v\|_{L_t^\infty(L^2)}^{\frac{1}{2}}
\|v\|_{L_{t}^1(\dot{B}_{p,1}^{1+\frac{2}{p}})}^{\frac{1}{2}}\bigr).
\end{split}
\end{equation*}
Finally we deduce from Lemma \ref{lem2.4} that
\begin{equation*}
\begin{split}
\|\nabla \Pi\|_{L^2}
\lesssim
\|a_0-\dot{S}_m a_0\|_{\dot B^{\frac{2}{\lambda}}_{\lambda,2}}
\|v\|_{\dot{B}_{p,1}^{1+\frac{2}{p}}}
+2^{m}\|\nabla v\|_{L^{2}}
\|a_0\|_{\dot B^{\frac{2}{\lambda}}_{\lambda,\infty}\cap{L^\infty}}
+\|K\|_{L^2}.
\end{split}
\end{equation*}
As a consequence, we obtain
\begin{align*}
\|v\|_{\widetilde{L}_t^r(\dot{B}_{p,1}^{\frac{2}{p}-1+\frac{2}{r}})}
\lesssim &
\|v_0\|_{\dot B^{-1+\frac{2}{p}}_{p,1}}
+(1+\|a_0\|_{\dot{B}_{\lambda,2}^{\frac{2}{\lambda}}})\big(\|K\|_{L_t^1(\dot{B}_{p,1}^{-1+\frac{2}{p}}\cap\,L^2)}
+\|a_0-\dot{S}_m a_0\|_{\dot{B}_{\lambda,\infty}^{\frac{2}{\lambda}}}
\|v\|_{L_t^1(\dot{B}_{p,1}^{1+\frac{2}{p}})}\big)
\\&
+2^m\sqrt{t}\|a_0\|_{\dot{B}_{\lambda,2}^{\frac{2}{\lambda}}\cap\,L^{\infty}}
\bigl(\|\nabla v\|_{L_t^2(L^2)}
+\|v\|_{L_t^\infty(L^2)}^{\frac{1}{2}}
\|v\|_{L_{t}^1(\dot{B}_{p,1}^{1+\frac{2}{p}})}^{\frac{1}{2}}\bigr).
\end{align*}
By taking $r=\infty$  and $r=1$ and using \eqref{basic-2.37aa},
  we deduce that for $m\geq m_0$ with $m_0$ being large enough
\begin{equation}\label{2.42}
\begin{split}
\|v\|_{\widetilde{L}_t^\infty(\dot{B}_{p,1}^{-1+\frac{2}{p}})}
&+\|v\|_{{L}_t^1(\dot{B}_{p,1}^{1+\frac{2}{p}})}
 \lesssim
\|v_0\|_{\dot B^{-1+\frac{2}{p}}_{p,1}}
+\|K\|_{L_t^1(\dot{B}_{p,1}^{-1+\frac{2}{p}}\cap L^2)}
\\&
+\bigl(2^m\sqrt{t}+2^{2m}t\bigr)\bigl(\|v_0\|_{L^2}
+\|\nabla\Phi_t\|_{L^1_t(L^2)}
+\|f\|_{L^1_t(L^2)}+\|\nabla\dv\,g\|_{L^1_t(L^2)}\bigr).
\end{split}
\end{equation}

\no{\bf Step 2.} The estimate of $\|\nabla \Pi\|_{ L^1_T(\dot{B}^{-1+\frac{2}{p}}_{p,1})}+\|v_t\|_{L^1_T(\dot{B}^{-1+\frac{2}{p}}_{p,1})}$.

By applying the operator $\dive\dot{\Delta}_q$ to the equation \eqref{2.37}, and using  a standard commutator's argument, we find
\begin{equation}\label{2.42}
\begin{split}
\dv\bigl((1+a_0)\nabla\dot\Delta_q\Pi\bigr) =-\dot\Delta_q\dv K
+\dot\Delta_q\dv(a_0\Delta v)+\dv[\dot\Delta_q,a_0]\nabla\Pi.
\end{split}
\end{equation}
By taking $L^2$ inner product of
 \eqref{2.42} with
$|\dot\Delta_q\Pi|^{p-2}\dot\Delta_q\Pi$ and using
 $\dv\, u=0,$
we find
\begin{equation*}
\begin{split}
2^{2q}\|\dot\Delta_q\Pi\|_{L^p}^p
&\lesssim
-\int_{\R^3}\dv((1+a_0)\dot\Delta_q\nabla\Pi)
 |\dot\Delta_q\Pi|^{p-2}\dot\Delta_q\Pi \,dx \\
&\lesssim
2^q\|\dot\Delta_q\Pi\|_{L^p}^{p-1}
\bigl(\|\dot\Delta_q K\|_{L^p}
+\|\dot\Delta_q(a_0\Delta v)\|_{L^p}+\|[\dot\Delta_q,a_0]\nabla\Pi\|_{L^p}\bigr),
\end{split}
\end{equation*}
from which, we infer
\begin{equation*}
\begin{split}
\|\nabla\Pi\|_{\dot B^{-1+{2\over p}}_{p,1}}
\lesssim &
\|K\|_{\dot B^{-1+\frac{2}{p}}_{p,1}}
+\|a_0\Delta v\|_{\dot B^{-1+\frac{2}{p}}_{p,1}}
+\sum_{q\in\mathbb{Z}}2^{q\bigl(-1+\f2p\bigr)}\|[\dot\Delta_q,a_0]\nabla\Pi\|_{L^p}.
\end{split}
\end{equation*}
It follows from Proposition \ref{prop2.2} that
\begin{equation*}
\begin{split}
&\|T_{a_0}\Delta v\|_{\dot B^{-1+\frac{2}{p}}_{p,1}}
\lesssim
\|a_0\|_{L^\infty}\|v\|_{\dot B^{1+\frac{2}{p}}_{p,1}},
\end{split}
\end{equation*}
and
\begin{equation*}
\|T_{\Delta v}a_0\|_{\dot B^{-1+\frac{2}{p}}_{p,1}}
\lesssim
\begin{cases}
\|a_0\|_{\dot B^{\frac{2}{\lambda}}_{\lambda,\infty}}\|\D v\|_{\dot B^{-1}_{\infty,1}},
\qquad\qquad\mbox{if $\lambda\le p$}, \\
\|a_0\|_{\dot B^{\frac{2}{\lambda}}_{\lambda,\infty}}
\|\Delta v\|_{\dot B^{-1-\frac{2}{\lambda}+\frac{2}{p}}_{\frac{p\lambda}{\lambda-p},1}},
\qquad\;\mbox{if $p\le\lambda$}.
\end{cases}
\end{equation*}
Whereas as $\frac{1}{2}<\frac{1}{p}+\frac{1}{\lambda}\le1$, one has
\begin{equation*}
\begin{split}
\|T_{\Delta v}a_0\|_{\dot B^{-1+\frac{2}{p}}_{p,1}}
\lesssim
\|a_0\|_{\dot B^{\frac{2}{\lambda}}_{\lambda,\infty}}
\|v\|_{\dot B^{1+\frac{2}{p}}_{p,1}},
\end{split}
\end{equation*}
and
\begin{equation*}
\begin{split}
\|R(a_0,\Delta v)\|_{\dot B^{-1+\frac{2}{p}}_{p,1}}
\lesssim
\|R(a_0,\Delta v)\|_{\dot B^{-1+\frac{2}{p}+\frac{2}{\lambda}}_{\frac{p\lambda}{p+\lambda},1}}
\lesssim
\|a_0\|_{\dot B^{\frac{2}{\lambda}}_{\lambda,\infty}}
\|v\|_{\dot B^{1+\frac{2}{p}}_{p,1}},
\end{split}
\end{equation*}
which yields
\begin{equation*}
\begin{split}
\|a_0\Delta v\|_{\dot B^{-1+\frac{2}{p}}_{p,1}}&\lesssim
\bigl(\|a_0\|_{L^\infty}
+\|a_0\|_{\dot B^{\frac{2}{\lambda}}_{\lambda,\infty}}\bigr)
\|v\|_{\dot B^{1+\frac{2}{p}}_{p,1}}.
\end{split}
\end{equation*}
While applying Lemma \ref{lem2.5} gives rise to
\begin{equation*}
\begin{split}
\sum_{q \in \mathbb{Z}}2^{q\bigl(-1+\f2p\bigr)}\|[\dot\Delta_q,a_0]\nabla\Pi\|_{L^p}
\lesssim
\|a_0\|_{\dot{B}_{\lambda,2}^{\frac{2}{\lambda}}}
\|\nabla\Pi\|_{L^2}.
\end{split}
\end{equation*}
Consequently, we obtain
\begin{equation*}
\begin{split}
\|\nabla\Pi\|_{L^1_t(\dot B^{-1+{2\over p}}_{p,1})}
\lesssim
\|K\|_{L^1_t(\dot B^{-1+\frac{2}{p}}_{p,1})}
+
\|v\|_{L^1_t(\dot B^{1+\frac{2}{p}}_{p,1})}
+\|\nabla\Pi\|_{L^1_t(L^2)}.
\end{split}
\end{equation*}

On the other hand, we deduce from
the momentum equation in \eqref{2.37} that
\begin{equation*}
\begin{split}
\|v_t\|_{L^1_t(\dot B^{-1+\frac{2}{p}}_{p,1})}
\lesssim
\bigl(1+\|a_0\|_{\dot B^{\frac{2}{\lambda}}_{\lambda,2}\cap\,L^\infty}\bigr)
\bigl(\|v\|_{L_t^1(\dot B^{1+\frac{2}{p}}_{p,1})}
+\|\nabla\Pi\|_{L_t^1(\dot B^{-1+\frac{2}{p}}_{p,1})}\bigr)
+\|K\|_{L^1_t(\dot B^{-1+\frac{2}{p}}_{p,1})}.
\end{split}
\end{equation*}
Therefore, we obtain
\begin{equation}\label{v-crit-toy-1}
\begin{split}
&\|v\|_{\widetilde{L}_t^\infty(\dot{B}_{p,1}^{-1+\frac{2}{p}})}+\|(\nabla^2v,\,\nabla\Pi,\,v_t)\|_{L^1_t(\dot B^{-1+\frac{2}{p}}_{p,1})}\lesssim\|v_0\|_{\dot B^{-1+\frac{2}{p}}_{p,1}}\\
&\qquad+\|K\|_{L_t^1(\dot{B}_{p,1}^{-1+\frac{2}{p}}\cap L^2)}
+\bigl(2^m\sqrt{t}+2^{2m}t\bigr)\bigl(\|v_0\|_{L^2}
+\|\partial_t\nabla\Phi\|_{L^1_t(L^2)}
+\|f\|_{L^1_t(L^2)}\bigr).
\end{split}
\end{equation}
It follows from the law of product that
\begin{equation*}
\begin{split}
\|K\|_{L_t^1(\dot B^{-1+\frac{2}{p}}_{p,1}\cap L^2)}
\lesssim
\bigl(1+\|a_0\|_{\dot B^{\frac{2}{\lambda}}_{\lambda,2}\cap L^\infty}\bigr)
\bigl(\|f\|_{L_t^1(\dot B^{-1+\frac{2}{p}}_{p,1}\cap L^2)}
&+\|\nabla\,\dv\,g\|_{L_t^1(\dot B^{-1+\frac{2}{p}}_{p,1}\cap L^2)}\bigr)\\
&+\|\nabla\Phi_t\|_{L_t^1(\dot B^{-1+\frac{2}{p}}_{p,1}\cap L^2)}.
\end{split}
\end{equation*}
By inserting the above estimate into \eqref{v-crit-toy-1}, we achieve
 \eqref{estimate-uniqueness-1q}. This completes the proof of Proposition \ref{prop2.4}.
\end{proof}

To prove the uniqueness part of Theorem \ref{thm1.1}
for  $p\in ]2,\infty[$, we recall the following propositions from \cite{AG2021}.

\begin{prop}\label{prop2.9}(See \cite{AG2021})
{\sl Let $u_0\in B^{-1}_{2,\infty}$ and $v\in L^1_T(B^1_{\infty,1})$ be a solenoidal
vector field. Let
$f\in\widetilde L^1_T( B^{-1}_{2,\infty}),$
and
$a \in\widetilde L^\infty_T( B^{2}_{2,1})$ with $
1+a\geq \f1M>0$ for some positive constant $M$. We assume that
$u \in L^{\infty}_T( B^{-1}_{2,\infty})
\cap\widetilde L^1_T( B^{1}_{2, \infty})$ and $\grad\Pi\in
\widetilde L^1_T( B^{-1}_{2,\infty})$  solves
\begin{equation*} 
\begin{cases}
\displaystyle\partial_t u +v \cdot \grad u -(1+a)( \Delta u-\nabla \Pi) = f,\quad (t, x) \in \R_+\times\R^2,\\
\displaystyle{\mathop{\rm div}}\, u=0,\\
\displaystyle u_{|t=0}=u_0.
\end{cases}
\end{equation*}
Then there holds:
\begin{equation*}\label{estimate-uniqueness-1}
\begin{split}
\| u\|_{L^\infty_T(B^{-1}_{2,\infty})}
+\|u\|_{\widetilde L^1_T(B^{1}_{2, \infty})}\leq&
Ce^{C\bigl(T+T\|\nabla\,a\|_{\widetilde L^\infty_T(B^1_{2,1})}^2+\|v\|_{L^1_T(B^1_{\infty,1})}\bigr)}
\\&\times\bigl(\|u_0\|_{B^{-1}_{2,\infty}}
+\|f\|_{\widetilde L^1_T(B^{-1}_{2,\infty})}
+\|a\|_{\widetilde L^\infty_T(\dot{B}^{1}_{2,1})}
\|\nabla\Pi\|_{\widetilde L^1_T(B^{-1}_{2,\infty})}\bigr).
\end{split}
\end{equation*}
}
\end{prop}
\begin{prop}\label{prop-H-negtive}(See \cite{AG2021})
 {\sl We assume that $a\in \dot{B}^{1}_{2,1}$ satisfies $ 0<{\underline{b}}\leq
1+a\leq {\bar{b}}$ for two positive constants $\underline{b}$ and $\bar{b}$ , and
\begin{equation}\label{est-gen-unique-1}
\|a-\dot{S}_k a\|_{\dot{B}^{1}_{2,1}}\leq c
\end{equation}
for some sufficiently small positive constant $c$ and some integer
$k \in \N.$  Let $F \in B^{-1}_{2,\infty}$  and $\grad \Pi \eqdefa
\mathcal{H}_{b}({F})\in B^{-1}_{2,\infty}$ solve
\begin{equation*}\label{unique-per-gen-model}
\dive\,((1+a) \grad \Pi) =\dive\,F.
\end{equation*}
Then there holds
\begin{equation*}\label{estimate-unique-per-model}
\|\grad \Pi\|_{B^{-1}_{2,\infty}}
 \lesssim
\bigl(1+2^{k}\|a\|_{B^{0}_{\infty,1}}(1+\|a\|_{B^{0}_{\infty,1}})\bigr)
 \bigl(\| F\|_{B^{-2}_{2,\infty}} +\|\dive\, F\|_{B^{-2}_{2,\infty}}\bigr).
\end{equation*}
}
\end{prop}


\renewcommand{\theequation}{\thesection.\arabic{equation}}
\setcounter{equation}{0}

\section{{\it A prior} estimates}\label{sect-Lip}

The goal of this section is to present the {\it a priori} estimates which will be the most crucial ingredient used to prove Theorem \ref{thm1.1}.
The main result states as follows.

\begin{prop}\label{prop3.2}
{\sl  Let $p\in [2,+\infty[$ and $\lambda\in [1,+\infty[$ such that
$\frac{1}{2}<\frac{1}{p}+\frac{1}{\lambda}\leq 1.$
Let $a_0\in L^\infty\cap\dot{B}_{\lambda,2}^{\frac{2}{\lambda}}$ with $\rho_0=\f1{1+a_0}$ satisfying \eqref{1.4} and $u_0\in\dot{B}_{p,1}^{-1+\frac{2}{p}}\cap L^2.$ Let $(a,u,\na\Pi)$ be a  smooth enough
solution of \eqref{1.3}. Then there are a positive constant $C$, a large integer $m_0 \in \mathbb{N}$ and a small positive time $T_1$ so that
for $m\geq m_0,$
\begin{equation}\label{3.2}
\begin{split}
&\|\nabla u\|_{L_{T_1}^2(L^2)}+\|u\|_{\widetilde L^1_{T_1}(\dot H^2)}
+\|u\|_{{L}_{T_1}^{1}(\dot{B}_{p,1}^{1+\frac{2}{p}})}
+\|\nabla \Pi\|_{L^1_{T_1}(L^2)}\\
&\le  C\sum_{q\in\mathbb{Z}}2^{q\bigl(\frac{2}{p}-1\bigr)}\bigl(1-e^{-c_p2^{2q}\,T_1}\bigr)
 \|\dot{\Delta}_{q}u_0\|_{L^p}
 +C\Bigl(\sum_{q\in\mathbb{Z}}\bigl(1-e^{-c2^{2q}\,T_1}\bigr)
\|\dot\Delta_qu_0\|_{L^2}^2\Bigr)^{\frac{1}{2}}
\\&\quad
+C \bigg(2^{2m}\,T_1+ \bigl(\sum_{q\geq m}2^{\frac{4}{\lambda}q}
 \|\dot{\Delta}_qa_0\|_{L^\lambda}^2\bigr)^{\frac{1}{2}}\bigg).
 \end{split}
 \end{equation}
}
\end{prop}
\begin{proof} We divide the proof into the following three steps.

\no{\bf Step 1.} The estimate of $\|\na u\|_{L^2_t(L^2)}.$

We first observe that there holds \eqref{ener} and it follows from \eqref{1.4} that
\begin{equation*}
M_1\leq \rho(t,x)\leq M_2.
\end{equation*}
While similar to the proof of Proposition \ref{prop2.4}, we get, by
first dividing the momentum equation of \eqref{1.2} by $\rho$  and then applying Leray projector $\mathbb{P}$  to
the resulting equation, that
\begin{equation*}
\begin{split}
\partial_t u+\mathbb{P}\bigl(u\cdot\nabla u\bigr)
-\mathbb{P}\bigl({\rho}^{-1}(\Delta u-\nabla\Pi)\bigr)
=0.
\end{split}
\end{equation*}
By
applying $\dot\Delta_q$ to the above equation and using a standard
commutator's process, we write
\begin{equation*}\label{3.3}
\begin{split}
\rho\partial_t\dot\Delta_qu+\rho
u\cdot\nabla\dot\Delta_qu-\D\dot\D_ju
&=-\rho[\dot\Delta_q\mathbb{P}, u\cdot\nabla]u
+\rho[\dot\Delta_q\mathbb{P},
{\rho}^{-1}]\bigl(\Delta u-\nabla\Pi\bigr).
\end{split}
\end{equation*}
By taking $L^2$ inner product of the above equation with
$\dot\Delta_qu$ and using the transport equation of \eqref{1.2}, we obtain
\begin{align*}
&\frac{1}{2}\frac{\mathrm{d}}{\mathrm{d}t}\int_{\R^2}\rho|\dot\Delta_q u|^2\,\mathrm{d}x-
\int_{\R^2}\D\dot\Delta_qu\cdot\dot\Delta_q u\,\mathrm{d}x
\\&
\leq \|\dot\Delta_qu\|_{L^2}
\bigl(\|\rho[\dot\Delta_q\mathbb{P}, u\cdot\nabla]u\|_{L^2}
+\|\rho[\dot\Delta_q\mathbb{P}, {\rho}^{-1}]
\bigl(\Delta u-\nabla\Pi\bigr)\|_{L^2}\bigr).
\end{align*}
Observing from   Lemma \ref{lem2.1} that
 \begin{align*}
 -\int_{\R^2}\D\dot\Delta_qu\cdot
\dot\Delta_q u\,\mathrm{d}x =&\int_{\R^2}|\nabla\dot\D_ju|^2\,\mathrm{d}x
\geq
\bar{c}2^{2q}\|\dot\Delta_qu\|_{L^2}^2.
\end{align*}
 we deduce  that for $c=\overline{c}/M$,
\begin{equation*} 
\begin{split}
&\frac{\mathrm{d}}{\mathrm{d}t}\|\sqrt{\rho}\dot\Delta_qu\|_{L^2}^2
+2c2^{2q}\|\sqrt{\rho}\dot\Delta_qu\|_{L^2}^2\\
&\lesssim\|\sqrt{\rho}\dot\Delta_qu\|_{L^2}\bigl(
\|[\dot\Delta_q\mathbb{P}, u\cdot\nabla]u\|_{L^2}
+\|[\dot\Delta_q\mathbb{P},{\rho}^{-1}]\bigl(\Delta u-\nabla\Pi\bigr)\|_{L^2}\bigr),
\end{split}
\end{equation*}
from which, we infer
\begin{equation}\label{3.5}
\begin{split}
&\|\sqrt{\rho}\dot\Delta_qu(t)\|_{L^2}
\lesssim
e^{-c2^{2q}t}\|\sqrt{\rho_0}\dot\Delta_qu_0\|_{L^2}
\\&
+\int_0^te^{-c2^{2q}(t-t')}
\big(\|[\dot\Delta_q\mathbb{P}, u\cdot\nabla]u\|_{L^2}
+\|[\dot\Delta_q\mathbb{P}, {\rho}^{-1}]\bigl(\Delta u-\nabla\Pi\bigr)\|_{L^2}\big)(\tau)\,\mathrm{d}\tau.
\end{split}
\end{equation}
As a consequence,  we deduce from Definition \ref{def2.2} that
\begin{equation}\label{3.6}
\begin{split}
\|\nabla u\|_{L_t^2(L^2)}\lesssim&\Bigl(\sum_{q\in\mathbb{Z}}\bigl(1-e^{-c2^{2q}t}\bigr)\|\dot\Delta_qu_0\|_{L^2}^2\Bigr)^{\frac{1}{2}}+\Bigr(\sum_{j\in\mathbb{Z}}\|[\dot\Delta_q\mathbb{P}, u\cdot\nabla]u\|_{L_t^1(L^2)}^2\Bigl)^{\frac{1}{2}}\\
&+\Bigl(\sum_{j\in\mathbb{Z}}\|[\dot\Delta_q\mathbb{P}, {\rho}^{-1}]\bigl(\Delta u-\nabla\Pi\bigr)\|_{L_t^1(L^2)}^2\Bigr)^{\frac{1}{2}}.
\end{split}
\end{equation}

In what follows, we shall handle term by term above.
We first get, by applying  \eqref{2.11} with $p=2,$ that
\beno
\Bigl(\sum_{q\in\mathbb{Z}}\|[\dot\Delta_q\mathbb{P}, u\cdot\nabla]u\|_{L_t^1(L^2)}^2\Bigr)^{\frac{1}{2}}
\lesssim\|\nabla u\|_{L^2_t(L^2)}^2.
\eeno
While it follows from  \eqref{2.26} that
\begin{align*}
\Bigl(\sum_{q\in\mathbb{Z}}
\|[\dot\Delta_q\mathbb{P}, {\rho}^{-1}]
\nabla\Pi\|_{L_t^1(L^2)}^2\Bigr)^{\frac{1}{2}}
\le
\int_0^t\sum_{q\in\mathbb{Z}}\|[\dot\Delta_q\mathbb{P}, {\rho}^{-1}]
\nabla\Pi\|_{L^2}d\tau
\lesssim
\|a\|_{L^\infty_t(\dot B^{\frac{2}{\lambda}}_{\lambda,2})}
\|\nabla\Pi\|_{L^1_t(L^2)},
\end{align*}
which together with  Lemma \ref{lem2.4} ensures that
 \begin{align*}
\Bigl(\sum_{q\in\mathbb{Z}}\|[\dot\Delta_q\mathbb{P}, {\rho}^{-1}]\nabla\Pi\|_{L_t^1(L^2)}^2\Bigr)^{\frac{1}{2}}
&\lesssim
\|a\|_{L^\infty_t(\dot B^{\frac{2}{\lambda}}_{\lambda,2})}
\Bigl(\|a-\dot S_ma\|_{\widetilde L^\infty_t(\dot{B}^{\frac{2}{\lambda}}_{\lambda,2})}
\|u\|_{L_{t}^{1}(\dot{B}_{p,1}^{1+\frac{2}{p}})}\notag\\
&+2^{m}\sqrt{t}\big(\| a\|_{L^{\infty}_t(L^\infty)}+\|a\|_{L^\infty_t(\dot{B}_{\lambda,\infty}^{\frac{2}{\lambda}})}\big)\|\nabla u\|_{L^{2}_{t}(L^2)}
+\|\nabla u\|_{L^2_t(L^2)}^2\Bigr).
\end{align*}

On the other hand, notice that
\begin{equation}\label{3.9}
[\dot\Delta_q\mathbb{P}, {\rho}^{-1}]\Delta u=[\dot\Delta_q\mathbb{P},a]\Delta u
=[\dot\Delta_q\mathbb{P},a-\dot S_ma]\Delta u+[\dot\Delta_q\mathbb{P},\dot S_ma]\Delta u.
\end{equation}
we deduce  from Lemma \ref{lem2.6} that
\begin{equation*}\label{3.10}
\bigl(\sum_{q\in\mathbb{Z}}\|[\dot\Delta_q\mathbb{P}, a-\dot S_ma]\Delta u\|_{L_t^1(L^2)}^2\bigr)^{\frac{1}{2}}\lesssim
\|a-\dot S_ma\|_{\widetilde L^\infty_t(\dot{B}_{\lambda,\infty}^{\frac{2}{\lambda}})}
\| u\|_{L_{t}^{1}(\dot{B}_{p,1}^{1+\frac{2}{p}})}.
\end{equation*}
Whereas it follows from Lemma \ref{lem2.7} that
 \begin{equation*}
 \begin{split}
\Bigl(\sum_{q\in\mathbb{Z}}\|[\dot\Delta_q\mathbb{P}, \dot S_ma]\Delta u\|_{L_t^1(L^2)}^2\Bigr)^{\frac{1}{2}}\lesssim 2^{m}\sqrt{t}\|a\|_{L^{\infty}_t(\dot{B}^{\frac{2}{\lambda}}_{\lambda,2})}\big(\|\nabla u\|_{L^{2}_{t}(L^2)}+\|u\|_{L_t^\infty(L^2)}^{\frac{1}{2}}\|u\|_{L_t^1(\dot{B}_{p,1}^{1+\frac{2}{p}})}^{\frac{1}{2}}\big).
\end{split}
\end{equation*}

 By substituting the above estimates into \eqref{3.6}, we obtain that for $2\leq p<\infty$,
 \begin{align*}
 \|\nabla u\|_{{L}_t^2(L^2)}
 \le& \Bigl(\sum_{q\in\mathbb{Z}}\bigl(1-e^{-c2^{2q}t}\bigr)\|\dot\Delta_qu_0\|_{L^2}^2\Bigr)^{\frac{1}{2}}+\bigl(1+\|a\|_{L^\infty_t(\dot{B}^{\frac{2}{\lambda}}_{\lambda,2} )}\bigr)\Bigl(
 \|a-\dot S_ma\|_{\widetilde L^\infty_t(\dot{B}^{\frac{2}{\lambda}}_{\lambda,2})}
\|u\|_{L_{t}^{1}(\dot{B}_{p,1}^{1+\frac{2}{p}})}
\\&\qquad\qquad
+2^{m}\sqrt{t}\| a\|_{L^{\infty}_t(\dot{B}^{\frac{2}{\lambda}}_{\lambda,2})}\bigl(\|\nabla u\|_{L^{2}_{t}(L^2)}+\|u\|_{L_t^\infty(L^2)}^{\frac{1}{2}}\|u\|_{L_t^1(\dot{B}_{p,1}^{1+\frac{2}{p}})}^{\frac{1}{2}}\bigr)
+\|\nabla u\|_{L^2_t(L^2)}^2\Bigr).
 \end{align*}
By inserting the estimates \eqref{2.31q} and \eqref{2.32} into the above inequality, for any $\eta>0,$ we find
\begin{equation}\label{3.14}
 \begin{split}
 \|\nabla u\|_{{L}_t^2(L^2)}
 \leq
& C\Bigl(\sum_{q\in\mathbb{Z}}\bigl(1-e^{-c2^{2q}t}\bigr)\|\dot\Delta_qu_0\|_{L^2}^2\Bigr)^{\frac{1}{2}}
+\eta\|u\|_{L_{t}^{1}(\dot{B}_{p,1}^{1+\frac{2}{p}})}\\&
 +C_\eta\bigl(1+\|u\|_{\widetilde{L}_t^1(\dot{H}^2)}\bigr)e^{C\|\nabla u\|_{L_t^1(L^\infty)}}\Bigl(
 \bigl(2^{2m}t+\|\nabla u\|_{L^2_t(L^2)}^2\bigr)\\
 &
 +
\Bigl(\bigl(\sum_{q\geq m}2^{\frac{4}{\lambda}q}
 \|\dot{\Delta}_qa_0\|_{L^\lambda}^2\bigr)^{\frac{1}{2}}
 +\bigl(e^{C\|\nabla u\|_{L_t^1(L^\infty)}}-1\bigr)\Bigr)
\|u\|_{L_{t}^{1}(\dot{B}_{p,1}^{1+\frac{2}{p}})}\Bigr).
\end{split}
 \end{equation}

 \no {\bf Step 2.} The estimate of
  $\|u\|_{L_{t}^{1}(\dot{B}_{p,1}^{1+\frac{2}{p}})}$.

 We first get, by a similar derivation of   \eqref{2.39-aa1}, that
\begin{align*}
\frac{d}{dt}
\|\rho^{\frac{1}{p}}\dot{\Delta}_{q}u\|_{L^p}^p+& pc_p2^{2j}\|\dot{\Delta}_{q}u\|_{L^p}^p
\leq
\|\dot{\Delta}_{q}u\|_{L^p}^{p-1}
\bigl(\bigl\|[\dot\Delta_q\mathbb{P}, u\cdot\nabla]u\bigr\|_{L^p}
+\|[\dot\Delta_q\mathbb{P}, {\rho}^{-1}]
\bigl(\Delta u-\nabla\Pi\bigr)\|_{L^p}\bigr).
\end{align*} from which and a similar derivation of \eqref{3.5}, we infer
\beq\label{3.17}
\begin{split}
\|\dot{\Delta}_{q}u(t)\|_{L^p}
\lesssim e^{-c_p2^{2j}t}\|\dot{\Delta}_{q}u_0\|_{L^p}
&+\int_{0}^te^{-c_p2^{2j}(t-\tau)}\|[\dot\Delta_q\mathbb{P}, u\cdot\nabla]u\|_{L^p}\,d\tau\\&
+\int_{0}^te^{-c_p2^{2j}(t-\tau)}\|[\dot\Delta_q\mathbb{P}, {\rho}^{-1}]\bigl(\Delta u-\nabla\Pi\bigr)\|_{L^p}\,d\tau.
\end{split}\eeq
Then we deduce from Definition \ref{def2.2} that
\begin{equation}\label{3.18}
\begin{split}
\|u\|_{{L}_{t}^{1}(\dot{B}_{p,1}^{1+\frac{2}{p}})}
\lesssim &
\sum_{q\in\mathbb{Z}}2^{q\bigl(\frac{2}{p}-1\bigr)}\bigl(1-e^{-c_p2^{2q}t}\bigr)\|\dot{\Delta}_{q}u_0\|_{L^p}\\ &+\sum_{q\in\mathbb{Z}}2^{q\bigl(\frac{2}{p}-1\bigr)}\Bigr(\|[\dot\Delta_q\mathbb{P}, u\cdot\nabla]u\|_{L_t^1(L^p)}
+\|[\dot\Delta_q\mathbb{P}, {\rho}^{-1}]\bigl(\Delta u-\nabla\Pi\bigr)\|_{L_t^1(L^p)}\Bigl).
\end{split}
 \end{equation}
It follows  from \eqref{2.11} that
\beno
\sum_{q\in\mathbb{Z}}2^{q\bigl(-1+\frac{2}{p}\bigr)}\|[\dot\Delta_q\mathbb{P}, u\cdot\nabla]u\|_{L_t^1(L^p)}\lesssim\|\nabla u\|_{L_{t}^2(L^2)}^2.
\eeno
Whereas we get, by applying Lemmas \ref{lem2.4}-\ref{lem2.5}, that
\beno
\sum_{q\in\mathbb{Z}}2^{q(-1+\frac{2}{p})}
\|[\dot\Delta_q\mathbb{P}, {\rho}^{-1}]\nabla\Pi\|_{L_t^1(L^p)}
\lesssim
\|a\|_{\widetilde{L}_t^{\infty}(\dot{B}_{\lambda,2}^{\frac{2}{\lambda}})}
\|\nabla\Pi\|_{L_{t}^1(L^2)},
\eeno
and
\beq\label{3.21}
\begin{split}
\|\nabla\Pi\|_{L_{t}^1(L^2)}\lesssim&\|a-\dot{S}_m a\|_{\widetilde{L}_{t}^{\infty}(\dot{B}_{\lambda,2}^{\frac{2}{\lambda}})}\|u\|_{L_{t}^1(\dot{B}_{p,1}^{1+\frac{2}{p}})}+\|\nabla u\|_{L_{t}^2(L^2)}^2\\
&+2^{m}\sqrt{t}\bigl(\| a\|_{\widetilde{L}_{t}^{\infty}(L^\infty)}+\|a\|_{\widetilde{L}_{t}^{\infty}(\dot{B}_{\lambda,\infty}^{\frac{2}{\lambda}})}\bigr)
\|\nabla u\|_{L_{t}^2(L^2)}.
\end{split}\eeq
 And it follows from Lemmas \ref{lem2.6}-\ref{lem2.7} that
\beno
\sum_{q\in\mathbb{Z}}2^{q\bigl(-1+\f2p\bigr)}\|[\dot{\Delta}_q\mathbb{P}, a-\dot{S}_m a]\Delta u\|_{L_t^1(L^p)}
\lesssim
\|u\|_{L^1_t(\dot{B}_{p,1}^{1+\frac{2}{p}})}
\|a-\dot{S}_m a\|_{{L}_{t}^{\infty}(\dot{B}_{\lambda,\infty}^{\frac{2}{\lambda}})}
\eeno
and
\begin{align*}
\sum_{q\in\mathbb{Z}}2^{q\bigl(-1+\f2p\bigr)}
\|[\dot{\Delta}_q\mathbb{P},\dot{S}_m a]\Delta u\|_{L_{t}^1(L^p)}
\lesssim &2^{m}\sqrt{t}\| a\|_{{L}_t^\infty(\dot{B}_{\lambda,2}^{\frac{2}{\lambda}})}\bigl(\|\nabla u\|_{L_t^2(L^2)}+\|u\|_{{L}^\infty_t(L^2)}^{\frac{1}{2}}
\|u\|_{L_{t}^{1}(\dot{B}_{p,1}^{1+\frac{2}{p}})}^{\frac{1}{2}}\bigr).
\end{align*}

By substituting the above estimates  into \eqref{3.18}, we achieve
\begin{align*}
\begin{split}
&\|u\|_{{L}_{t}^{1}(\dot{B}_{p,1}^{1+\frac{2}{p}})}
\leq C
\sum_{q\in\mathbb{Z}}2^{q\bigl(\frac{2}{p}-1\bigr)}\bigl(1-e^{-c_p2^{2j}t}\bigr)
\|\dot{\Delta}_{q}u_0\|_{L^p}
+\eta\|u\|_{L_{t}^{1}(\dot{B}_{p,1}^{1+\frac{2}{p}})}
\\&+
C_\eta\bigl(1+\|a\|_{\widetilde L_t^{\infty}(\dot{B}_{\lambda,2}^{\frac{2}{\lambda}})}\bigr)
\Bigl(\|a-\dot{S}_m a\|_{\widetilde L_{t}^{\infty}(\dot{B}_{\lambda,2}^{\frac{2}{\lambda}})}
\|u\|_{L_{t}^1(\dot{B}_{p,1}^{1+\frac{2}{p}})}
+2^{2m}t\| a\|_{\widetilde L_{t}^{\infty}(\dot{B}_{\lambda,2}^{\frac{2}{\lambda}})}^2
+\|\nabla u\|_{L_t^2(L^2)}^2\Bigr).
\end{split}
\end{align*}
By taking $\eta$ to be small enough and substituting  the estimates \eqref{2.31q} and \eqref{2.32} into the resulting inequaity, we arrive at
\begin{equation}\label{3.25}
\begin{split}
\|u\|_{{L}_{t}^{1}(\dot{B}_{p,1}^{1+\frac{2}{p}})}
\lesssim & \sum_{q\in\mathbb{Z}}2^{q\bigl(\frac{2}{p}-1\bigr)}\bigl(1-e^{-c_p2^{2j}t}\bigr)
\|\dot{\Delta}_{q}u_0\|_{L^p}\\
&+\bigl(1+\|u\|_{\widetilde{L}_t^1(\dot{H}^2)}\bigr)
e^{C\|\nabla u\|_{L_t^1(L^\infty)}}
 \Bigl(\bigl(2^{2m}t+\|\nabla u\|_{L^2_t(L^2)}^2\bigr)\\
 &+
\bigl(\bigl(\sum_{q\geq m}2^{\frac{4}{\lambda}q}
 \|\dot{\Delta}_q a_0\|_{L^\lambda}^2\bigr)^{\frac{1}{2}}+ \bigl(e^{C\|\nabla u\|_{L_t^1(L^\infty)}}-1\bigr)\bigr)
 \|u\|_{L_{t}^{1}(\dot{B}_{p,1}^{1+\frac{2}{p}})}\Bigr).
\end{split}
 \end{equation}

  \no {\bf Step 3.} The closing of the estimate.

In view of \eqref{3.5},  we get, by using a similar derivation of
  \eqref{3.14}, that
\begin{equation}\label{3.26}
\begin{split}
\|u\|_{\widetilde{L}_t^1(\dot{H}^2)}
 \leq& C \Bigl(\sum_{q\in\mathbb{Z}}\bigl(1-e^{-c_p2^{2q}t}\bigr)
\|\dot{\Delta}_{q}u_0\|_{L^2}^2\Bigr)^{\frac{1}{2}}+\eta\|u\|_{L_{t}^{1}(\dot{B}_{p,1}^{1+\frac{2}{p}})}\\
&
 +C_\eta\bigl(1+\|u\|_{\widetilde{L}_t^1(\dot{H}^2)}\bigr)
e^{C\|\nabla u\|_{L_t^1(L^\infty)}}
 \Bigl(2^{2m}t+\|\nabla u\|_{L^2_t(L^2)}^2\\
 &+
\bigl(\bigl(\sum_{q\geq m}2^{\frac{4}{\lambda}q}
 \|\dot{\Delta}_q a_0\|_{L^\lambda}^2\bigr)^{\frac{1}{2}}+ \bigl(e^{C\|\nabla u\|_{L_t^1(L^\infty)}}-1\bigr)\bigr)
 \|u\|_{L_{t}^{1}(\dot{B}_{p,1}^{1+\frac{2}{p}})}\Bigr).
 \end{split}
 \end{equation}

Let us denote
\begin{equation*}
\begin{split}
U(t)\eqdefa
\|\nabla u\|_{{L}_t^2(L^2)}
 +\|u\|_{{L}_{t}^{1}(\dot{B}_{p,1}^{1+\frac{2}{p}})}
 +\|u\|_{\widetilde{L}_t^1(\dot{H}^2)}.
\end{split}
 \end{equation*}
Then by summarizing  the estimates \eqref{3.14}, \eqref{3.25} and \eqref{3.26}, we achieve
\begin{equation}\label{3.27}
\begin{split}
 U(t)
 &\leq
 C\sum_{q\in\mathbb{Z}}2^{q\bigl(\frac{2}{p}-1\bigr)}\bigl(1-e^{-c_p2^{2q}t}\bigr)
 \|\dot{\Delta}_{q}u_0\|_{L^p}
 +C\Bigl(\sum_{q\in\mathbb{Z}}\bigl(1-e^{-c2^{2q}t}\bigr)
\|\dot\Delta_qu_0\|_{L^2}^2\Bigr)^{\frac{1}{2}}
\\&
+C_\eta e^{CU(t)}\bigl(1+U(t)\bigr)
\Bigl(2^{2m}t+U^2(t)
+
 \big(\bigl(\sum_{q\geq m}2^{\frac{4}{\lambda}q}
 \|\dot{\Delta}_qa_0\|_{L^\lambda}^2\bigr)^{\frac{1}{2}}
 +\bigl(e^{CU(t)}-1\bigr)\big)
U(t)\Bigr).
 \end{split}
 \end{equation}
  Then by first taking $m\geq m_0$ with $m_0$ being large enough, and then $T_1$  being sufficiently small  in \eqref{3.27},
   we deduce that there exists a sufficiently small constant $C_0$  so that
\begin{equation}\label{3.27-bb}
\begin{split}
 &\|\nabla u\|_{{L}_{T_1}^2(L^2)}
 +\|u\|_{\widetilde L^1_{T_1}(\dot H^2)}
 +\|u\|_{{L}_{T_1}^{1}(\dot{B}_{p,1}^{1+\frac{2}{p}})}
 \le C_0,
\end{split}
 \end{equation}
which along with \eqref{3.27} ensures
\begin{equation}\label{3.27-aa}
\begin{split}
&\|\nabla u\|_{L_{T_1}^2(L^2)}+\|u\|_{\widetilde L^1_{T_1}(\dot H^2)}
+\|u\|_{{L}_{T_1}^{1}(\dot{B}_{p,1}^{1+\frac{2}{p}})}\\
&\le  C\sum_{q\in\mathbb{Z}}2^{q\bigl(\frac{2}{p}-1\bigr)}\bigl(1-e^{-c_p2^{2q}\,T_1}\bigr)
 \|\dot{\Delta}_{q}u_0\|_{L^p}
 +C\Bigl(\sum_{q\in\mathbb{Z}}\bigl(1-e^{-c2^{2q}\,T_1}\bigr)
\|\dot\Delta_qu_0\|_{L^2}^2\Bigr)^{\frac{1}{2}}
\\&\quad
+C \bigg(2^{2m}\,T_1+ \bigl(\sum_{q\geq m}2^{\frac{4}{\lambda}q}
 \|\dot{\Delta}_qa_0\|_{L^\lambda}^2\bigr)^{\frac{1}{2}}\bigg).
 \end{split}
 \end{equation}

By summarizing the estimates \eqref{3.21}, \eqref{2.31q}, \eqref{2.32} and \eqref{3.27-aa}, we obtain \eqref{3.2}
and
\begin{equation}\label{3.27-cc}
\begin{split}
 &\|\nabla u\|_{{L}_{T_1}^2(L^2)}
 +\|u\|_{\widetilde L^1_{T_1}(\dot H^2)}
 +\|u\|_{{L}_{T_1}^{1}(\dot{B}_{p,1}^{1+\frac{2}{p}})}+\|\nabla\Pi\|_{{L}_{T_1}^1(L^2)}
 \le C\,C_0,
\end{split}
 \end{equation}
which completes the proof  of Proposition \ref{prop3.2}.
 \end{proof}

\begin{prop}\label{prop3.3}
{\sl Under the assumptions of Proposition \ref{prop3.2}, for any $t\in[0,T_1]$ with $T_1$ being determined by Proposition \ref{prop3.2},
we have
\begin{align}\label{3.28}
&\|u\|_{\widetilde{L}_{t}^{\infty}(\dot{B}_{p,1}^{-1+\frac{2}{p}})}
+\|\p_tu\|_{\widetilde L_{t}^1(L^2)}\leq C_{\rm in}
\andf \|\nabla\Pi\|_{L^1_t(\dot{B}_{p,1}^{-1+\frac{2}{p}})}
+\|\p_tu\|_{L^1_t(\dot{B}_{p,1}^{-1+\frac{2}{p}})}
\leq
C_{\rm in},
\end{align}
where the constant $C_{\rm in}$ depends only on the initial data $(\rho_0, u_0).$}
\end{prop}

\begin{proof} We first deduce from
 \eqref{3.17} and Definition \ref{def2.2}, that
\begin{align*}
&\|u\|_{\widetilde{L}_{t}^{\infty}(\dot{B}_{p,1}^{-1+\frac{2}{p}})}\le\|u_0\|_{\dot{B}_{p,1}^{-1+\frac{2}{p}}}
+\sum_{q\in\mathbb{Z}}2^{q\bigl(\frac{2}{p}-1\bigr)}\big(\|[\dot\Delta_q\mathbb{P}, u\cdot\nabla]u\|_{L_t^1(L^p)}
+\|[\dot\Delta_q\mathbb{P}, {\rho}^{-1}]\bigl(\Delta u-\nabla\Pi\bigr)\|_{L_t^1(L^p)}\big),
\end{align*}
from which, we get, by a similar derivation of \eqref{3.25}, that
\begin{align*}
\|u\|_{\widetilde{L}_{t}^{\infty}(\dot{B}_{p,1}^{-1+\frac{2}{p}})}
\lesssim &
 \|u_0\|_{\dot{B}_{p,1}^{-1+\frac{2}{p}}}+
\|u\|_{L_{t}^{1}(\dot{B}_{p,1}^{1+\frac{2}{p}})}\notag\\
&+\bigl(1+\|u\|_{\widetilde{L}_t^1(\dot{H}^2)}\bigr)e^{C\|\nabla u\|_{L_t^1(L^\infty)}}
 \Bigl(\bigl(2^{2m}t+\|\nabla u\|_{L^2_t(L^2)}^2\bigr)\\
 &+\big(
\bigl(\sum_{q\geq m}2^{\frac{4}{\lambda}q}
 \|\dot{\Delta}_q\nabla a_0\|_{L^\lambda}^2\bigr)^{\frac{1}{2}}+\bigl(e^{C\|\nabla u\|_{L_t^1(L^\infty)}}-1\bigr)\big)
\|u\|_{L_{t}^{1}(\dot{B}_{p,1}^{1+\frac{2}{p}})}\Bigr),
\end{align*}
which together with \eqref{ener} and \eqref{3.27-cc} ensures that
\begin{align*}
\|u\|_{\widetilde{L}_{t}^{\infty}(\dot{B}_{p,1}^{-1+\frac{2}{p}})}\leq C_{\rm in}.
\end{align*}

While in view of  the momentum equations in \eqref{1.2}, \eqref{ener} and \eqref{3.2}, we infer
\begin{align*}
\|\p_tu\|_{\widetilde L_t^1(L^2)}
\lesssim
(1+\|a\|_{L_t^\infty(L^\infty)})\bigl(\|\Delta u\|_{\widetilde L_t^1(L^2)}+\|\nabla \Pi\|_{L^1_t(L^2)}+\|u\|_{L_t^2(L^\infty)}\|\nabla u\|_{L^2_t(L^2)}\bigr)\leq C_{\rm in},
\end{align*}
where we used  the fact that
\begin{align*}
\|u\|_{L_t^2(L^\infty)}\leq\|u\|_{L_t^2(\dot{B}_{p,1}^{\frac{2}{p}})}
\lesssim
\|u\|_{L_t^\infty(\dot{B}_{p,1}^{-1+\frac{2}{p}})}
+\|u\|_{L_t^1(\dot{B}_{p,1}^{1+\frac{2}{p}})}.
\end{align*}
This leads to the first inequality of \eqref{3.28}.

To estimate the pressure function $\Pi,$ we get, by applying the
operator $\dv$ to the momentum equation of \eqref{1.3} and using $\dv\, u=0,$ that
$$
\dv[(1+a)\nabla\Pi]=-\dv[(u\cdot\nabla)u] +\dv(a\Delta u),
$$
from which, we infer
\begin{equation*}\label{3.31}
\begin{split}
\dv[(1+a)\nabla\dot\Delta_q\Pi] =-\dot\Delta_q\dv[(u \cdot \grad)u]
+\dot\Delta_q\dv(a\Delta u)+\dv[\dot\Delta_q,a]\nabla\Pi.
\end{split}
\end{equation*}
By taking $L^2$ inner product of the  above equation with
$|\dot\Delta_q\Pi|^{p-2}\dot\Delta_q\Pi$ and using
 $\dv\, u=0,$
we deduce  from \eqref{1.4} that for $p>1$
\begin{equation*}
\begin{split}
2^{2q}\|\dot\Delta_q\Pi\|_{L^p}^p
&\lesssim
-\int_{\R^3}\dv((1+a)\dot\Delta_q\nabla\Pi)
 |\dot\Delta_q\Pi|^{p-2}\dot\Delta_q\Pi \,dx \\
&\lesssim
2^q\|\dot\Delta_q\Pi\|_{L^p}^{p-1}
\bigl(2^q\|\dot\Delta_q(u\otimes u)\|_{L^p}
+\|\dot\Delta_q(a\Delta u)\|_{L^p}+\|[\dot\Delta_q,a]\nabla\Pi\|_{L^p}\bigr),
\end{split}
\end{equation*}
which implies
\begin{equation}\label{presure-crit-1}
\begin{split}
\|\nabla\Pi\|_{\dot B^{{-1+\f2p}}_{p,1}}
\lesssim &
\|u\otimes u\|_{\dot B^{\frac{2}{p}}_{p,1}}
+\|a\Delta u\|_{\dot B^{-1+\frac{2}{p}}_{p,1}}
+\sum_{q\in\mathbb{Z}}2^{q\bigl(-1+\f2p\bigr)}\|[\dot\Delta_q,a]\nabla\Pi\|_{L^p}.
\end{split}
\end{equation}
Since $\dot B^{\frac{2}{p}}_{p,1}(\mathbb{R}^2)$ is an Banach algebra for $p<+\infty$,  one has
$$
\|u\otimes u\|_{\dot B^{\frac{2}{p}}_{p,1}}
\lesssim
\|u\|_{\dot B^{\frac{2}{p}}_{p,1}}^2.
$$
While it follows from the law of product in Besov space and Lemma \ref{lem2.5} that
\begin{align*}
\|a\Delta u\|_{\dot B^{-1+\frac{2}{p}}_{p,1}}
+\sum_{q\in\mathbb{Z}}2^{q\bigl(-1+\f2p\bigr)}\|[\dot\Delta_q,a]\nabla\Pi\|_{L^p}
\leq &\|a\|_{L^\infty_t(L^\infty\cap\dot B^{\frac{2}{\lambda}}_{\lambda,\infty})}
\|u\|_{L^1_t(\dot B^{1+\frac{2}{p}}_{p,1})}\\
&+\|a\|_{L^\infty_t(\dot B^{\frac{2}{\lambda}}_{\lambda,2})}\|\nabla\Pi\|_{L^1_t(L^2)}.\end{align*}
By substituting the above estimates and \eqref{3.27-cc} into \eqref{presure-crit-1},  we deduce from \eqref{2.31q} that
\begin{equation}\label{PMPZ}
\begin{aligned}
\|\nabla\Pi\|_{L^1_t(\dot B^{{-1+\f2p}}_{p,1})}
&\lesssim
\|u\|_{L^2_t(\dot B^{\frac{2}{p}}_{p,1})}^2
+\|a\|_{L^\infty_t(L^\infty\cap\dot B^{\frac{2}{\lambda}}_{\lambda,\infty})}
\|u\|_{L^1_t(\dot B^{1+\frac{2}{p}}_{p,1})}
+\|a\|_{L^\infty_t(\dot B^{\frac{2}{\lambda}}_{\lambda,2})}\|\nabla\Pi\|_{L^1_t(L^2)}
\\&
\lesssim
\|u\|_{\widetilde L^\infty_t(\dot B^{-1+\frac{2}{p}}_{p,1})}
\|u\|_{L^1_t(\dot B^{1+\frac{2}{p}}_{p,1})}\\
&\qquad+\|a\|_{L^\infty_t(\dot B^{\frac{2}{\lambda}}_{\lambda,2}\cap L^\infty)}\bigl(
\|u\|_{L^1_t(\dot B^{1+\frac{2}{p}}_{p,1})}
+\|\nabla\Pi\|_{L^1_t(L^2)}\bigr)
\lesssim
C_{\rm in}.
\end{aligned}
\end{equation}

Finally we deduce from the
 momentum equation of \eqref{1.3} and the law of product in Besov spaces that
\begin{equation}\label{UMPZ}
\begin{aligned}
\|\p_tu\|_{L^1_t(\dot B^{-1+\frac{2}{p}}_{p,1})}
&\lesssim
\|u\|_{L^2_t(\dot B^{\frac{2}{p}}_{p,1})}^2
+\bigl(1+\|a\|_{L^\infty_t(L^\infty\cap\dot B^{\frac{2}{\lambda}}_{\lambda,\infty})}\bigr)
\bigl(\|u\|_{L^1_t(\dot B^{1+\frac{2}{p}}_{p,1})}
+\|\nabla\Pi\|_{L^1_t(\dot B^{-1+\frac{2}{p}}_{p,1})}\bigr)\\
&\lesssim
C_{\rm in}.
\end{aligned}
\end{equation}
which together with \eqref{PMPZ} ensures the second inequality of \eqref{3.28}.
This completes the proof of Proposition \ref{prop3.3}.
\end{proof}

\renewcommand{\theequation}{\thesection.\arabic{equation}}
\setcounter{equation}{0}

\section{The proof of Theorem \ref{thm1.1}}\label{sect-global}

In this section, we present the proof of  Theorem \ref{thm1.1}.

\begin{proof}[Proof  of Theorem \ref{thm1.1}] We divide the proof of  Theorem \ref{thm1.1}
into two steps.

\no{\bf Step 1.} Existence of strong solutions.

We first mollify the initial data
to be
\beq \label{o.1} a_{0,n}\eqdefa\,a_0\ast j_n,\andf
u_{0,n}\eqdefa u_0\ast j_n,\eeq where
$j_n(|x|)=n^{2}j(|x|/n)$ is the standard Friedrich's mollifier. Then
we deduce  from the standard well-posedness theory of inhomogeneous
Navier-Stokes system (see \cite{danchin04, PZZ} for instance) that
\eqref{1.2}  has a unique global
solution $(\rho_n,u_n,\na\Pi_n).$  It is easy to observe from \eqref{o.1} that
\beno
\|a_{0,n}\|_{L^\infty\cap\dot B^{\frac{2}{\lambda}}_{\lambda,2}}
\le
C\|a_0\|_{L^\infty\cap\dot B^{\frac{2}{\lambda}}_{\lambda,2}}
\andf
\|u_{0,n}\|_{L^2\cap\dot B^{-1+\frac{2}{p}}_{p,1}}\leq C\|u_0\|_{ L^2\cap\dot{B}^{-1+\frac{2}{p}}_{p,1}}.
\eeno
Then under the assumptions of Theorem \ref{thm1.1}, we deduce  from the proof of Propositions
\ref{prop3.2} and  \ref{prop3.3} that there exists a positive time $T_1$  so that
\beq\label{unif-appro-soln-1}
\begin{aligned}
&\|u_n\|_{\widetilde L^\infty_{T_1}(\dot B^{-1+\frac{2}{p}}_{p,1})
\cap L^1_{T_1}(\dot B^{1+\frac{2}{p}}_{p,1})\cap\widetilde L^1_{T_1}(\dot H^2)}
+\|\partial_tu_n\|_{\widetilde L^1_{T_1}(L^2)\cap L^1_{T_1}(\dot B^{-1+\frac{2}{p}}_{p,1})}+\|\nabla\Pi_n\|_{L^1_{T_1}(L^2)\cap L^1_{T_1}(\dot B^{-1+\frac{2}{p}}_{p,1})}\\
&+
\|a_{n}\|_{\widetilde L^\infty_{T_1}(\dot B^{\frac{2}{\lambda}}_{\lambda,2})\cap L^\infty_{T_1}(L^\infty)}
\leq C_{\rm in}\andf M_1\leq \rho_n(t,x)\leq M_2.
\end{aligned}
\eeq

 With \eqref{unif-appro-soln-1}, we get, by using a standard
compactness argument (see \cite{DAN-03} for instance), that \eqref{1.2} has a solution
$(\rho,u,\na\Pi)$ on $[0,T_1]$ so that
\begin{equation}\label{space-soln-exist}
  \begin{split}
 &a\in C([0,T_1];\,\dot B^{\frac{2}{\lambda}}_{\lambda,2}\cap\,L^\infty),\quad
 u\in C([0,T_1];\,\dot B^{-1+\frac{2}{p}}_{p,1})
 \cap  L^1_{T_1}(\dot B^{1+\frac{2}{p}}_{p,1})
 \cap\widetilde L^1_{T_1}(\dot H^2),\\
 &\p_tu\in\widetilde L^1_{T_1}(L^2)
 \cap  L^1_{T_1}(\dot B^{-1+\frac{2}{p}}_{p,1}),
 \,
 \na\Pi\in L^1_{T_1}(L^2)\cap
 L^1_{T_1}(\dot B^{-1+\frac{2}{p}}_{p,1}) \andf M_1\leq \rho(t,x)\leq M_2.
\end{split}
\end{equation}
Then there exists   $t_0\in
(0,T_1)$ such that $u(t_0)\in H^1,$ and it follows from \eqref{space-soln-exist} that $a(t_0)\in \dot B^{\frac{2}{\lambda}}_{\lambda,2}\cap L^\infty$ and $u(t_0)\in B^{-1+\frac{2}{p}}_{p,1}.$ As a consequence, with initial data at time $t_0,$  \eqref{1.2} has a  global solution (see
\cite{DM13, PZZ} for instance) $(\rho, u, \na \Pi)$ so that
\beq \label{S4eq1}
\begin{split}
&(\partial_tu,\nabla^2u,\nabla\Pi)\in \bigl(L^2([t_0,+\infty[;\,L^2)\bigr)^3,\ \
u\in L^\infty([t_0,+\infty[;\,H^1), \\
& \nabla u \in L^2([t_0,+\infty[;\,L^2)\cap L^1_{\mbox{loc}}([t_0,\infty[;L^\infty).
\end{split}
\eeq

On the other hand, we deduce from the inequality \eqref{3.17} that for $t\geq t_0$
$$
\begin{aligned}
&\|u\|_{\widetilde{L}^{\infty}([t_0,t);\,\dot{B}_{p,1}^{-1+\frac{2}{p}})}
+\|u\|_{{L}^{1}([t_0,t);\,\dot{B}_{p,1}^{1+\frac{2}{p}})}
\lesssim
\sum_{q\in\mathbb{Z}}2^{q\bigl(\frac{2}{p}-1\bigr)}(1-e^{-c_p2^{2j}t})
\|\dot{\Delta}_{q}u(t_0)\|_{L^p}\notag\\ &
+\sum_{q\in\mathbb{Z}}2^{q\bigl(\frac{2}{p}-1\bigr)}\bigr(\|[\dot\Delta_q\mathbb{P}, u\cdot\nabla]u\|_{L^1([t_0,t);\,L^p)}
+\|[\dot\Delta_q\mathbb{P}, {\rho}^{-1}]
\bigl(\Delta u-\nabla\Pi\bigr)\|_{L^1([t_0,t);\,L^p)}\bigl),
\end{aligned}
$$
from which, Lemmas \ref{lem2.5} and \ref{lem2.3} and \eqref{S4eq1}, we deduce that
$$
\begin{aligned}
\|u\|_{\widetilde{L}^{\infty}([t_0,t);\,\dot{B}_{p,1}^{-1+\frac{2}{p}})}
+\|u\|_{{L}^{1}([t_0,t);\,\dot{B}_{p,1}^{1+\frac{2}{p}})}
&\lesssim
\|u_0\|_{\dot{B}_{p,1}^{-1+\frac{2}{p}}}
+\|\nabla u\|_{L^2([t_0,t);\,L^2)}^2
\\&
+\|a\|_{\widetilde L^\infty([t_0,t);\,\dot B^{\frac{2}{\lambda}}_{\lambda,2})}
\|\Delta u-\nabla\Pi\bigr\|_{L^1([t_0,t);\,L^2)}\leq C(t),
\end{aligned}
$$
which along with Proposition \ref{prop2.3} ensures that
$$
\|a\|_{\widetilde L^\infty([t_0,t);\,\dot B^{\frac{2}{\lambda}}_{\lambda,2})}
\lesssim
\|a(t_0)\|_{\dot B^{\frac{2}{\lambda}}_{\lambda,2}}
\bigl(1+\|u\|_{\widetilde L^1([t_0,t);\,\dot H^2)})
e^{C\|\nabla u\|_{L^1([t_0,t);\,L^\infty)}}\leq C(t).
$$
Hence we deduce from inequalities \eqref{PMPZ} and \eqref{UMPZ} that
\beno
\|(\partial_tu,\nabla\Pi)\|_{L^1_{loc}([t_0,+\infty);\,\dot{B}_{p,1}^{-1+\frac{2}{p}})}\leq C(t).
\eeno
This completes the  existence part of Theorem \ref{thm1.1}.

\no{\bf Step 2.} Uniqueness of strong solutions.

We divide the uniqueness part further  into the following three sub-steps:

\no{\bf Step 2.1} Propagation of regularity of the density $a$ for $p\in [2,\infty[$.

It follows from Theorem 3.14 of \cite{BCD} that
$$
\|a\|_{L^\infty_t(\dot B^1_{2,1})}
\le
\|a_0\|_{\dot B^1_{2,1}}e^{C\|u\|_{L^1_t(\dot B^1_{\infty,1})}}.
$$
When $p\in ]2,\infty[,$ we deduce from the transport equation of \eqref{1.2} that
$$
\partial_t\partial_j\dot\Delta_qa+u\cdot\nabla\partial_j\dot\Delta_qa
=-\dot\Delta_q(\partial_ju\cdot\nabla a)-[\dot\Delta_q,u\cdot\nabla]\partial_ja,
$$
from which and $\dive u=0,$ we infer
$$
\|\dot\Delta_q\partial_ja(t)\|_{L^{\frac{p}{p-1}}}
\le
\|\dot\Delta_q\partial_ja_0\|_{L^{\frac{p}{p-1}}}
+\int_0^t\|\dot\Delta_q(\partial_ju\cdot\nabla a)\|_{L^{\frac{p}{p-1}}}\,d\tau
+\int_0^t\|[\dot\Delta_q,u\cdot\nabla]\partial_ja\|_{L^{\frac{p}{p-1}}}\,d\tau.
$$
Since $0\le1-\frac{2}{p}<1,$ we get,  by applying  classical commutator's estimate (see Lemma 2.100 and Remark 2.102 in \cite{BCD}), that
\beq\label{S4eq2}
\|\partial_j a\|_{L^\infty_t(\dot{B}^{1-\frac{2}{p}}_{\frac{p}{p-1},\infty})}
\le
\|\partial_j a_0\|_{\dot{B}^{1-\frac{2}{p}}_{\frac{p}{p-1},\infty}}
+\|\partial_j u\cdot\nabla a\|_{\widetilde L^1_t(\dot{B}^{1-\frac{2}{p}}_{\frac{p}{p-1},\infty})}
+C\int_0^t\|\nabla u\|_{L^\infty}
\|\partial_j a\|_{\dot{B}^{1-\frac{2}{p}}_{\frac{p}{p-1},\infty}}\mathrm{d}\tau.
\eeq
It follows from Proposition \ref{prop2.2}, $\dv\,u=0$ and $p\in ]2,\infty[$ that
\begin{equation*}
  \begin{split}
 &\|T_{\partial_ju}\nabla a\|_{L^1_t(\dot{B}^{1-\frac{2}{p}}_{\frac{p}{p-1},\infty})}\lesssim \int_0^t\|\nabla u\|_{L^\infty}
\|\nabla\,a\|_{\dot{B}^{1-\frac{2}{p}}_{\frac{p}{p-1},\infty}}\,d\tau
  \end{split}
\end{equation*}
and
\begin{equation}\label{propog-a-est-1}
  \begin{split}
&\|T_{\nabla a}\partial_ju\|_{\widetilde{L}^1_t(\dot{B}^{1-\frac{2}{p}}_{\frac{p}{p-1},\infty})}+\|R(\partial_ju,a)\|_{\widetilde L^1_t(\dot{B}^{2-\frac{2}{p}}_{\frac{p}{p-1},\infty})}\\
&\lesssim \|\nabla a\|_{L^\infty_t(\dot{B}^{-\frac{2}{p}}_{\frac{2p}{p-2},\infty})}\|\nabla u\|_{\widetilde{L}^1_t(\dot{H}^1)}\lesssim \|\nabla a\|_{L^\infty_t(\dot{B}^{1-\frac{2}{p}}_{\frac{p}{p-1},\infty})}\|\nabla u\|_{\widetilde{L}^1_t(\dot{H}^1)}.
  \end{split}
\end{equation}
By using Bony's decomposition and substituting the above estimates into \eqref{S4eq2}, we deduce from \eqref{3.27-cc}
  that
for $t\leq T_1,$
\beno
\|\nabla a\|_{L^\infty_t(\dot{B}^{1-\frac{2}{p}}_{\frac{p}{p-1},\infty})}
\lesssim
\|\nabla a_0\|_{\dot{B}^{1-\frac{2}{p}}_{\frac{p}{p-1},\infty}}
+C\int_0^t\|\nabla u\|_{L^\infty}\|\na a\|_{\dot{B}^{1-\frac{2}{p}}_{\frac{p}{p-1},\infty}}\,d\tau.
\eeno
Applying Gronwall's inequality gives
\begin{align*}
\|\nabla a\|_{L^\infty_t(\dot{B}^{1-\frac{2}{p}}_{\frac{p}{p-1},\infty})}
\lesssim
\|\nabla a_0\|_{\dot{B}^{1-\frac{2}{p}}_{\frac{p}{p-1},\infty}}e^{C\|\na u\|_{L^1_t(L^\infty)}}\ \ \mbox{for}\ \ t\leq T_1.
\end{align*}
On the other hand, we deduce from \eqref{propog-a-est-1} that
\begin{align*}
\|T_{\nabla a}\partial_ju\|_{ L^1_t(\dot{B}^{1-\frac{2}{p}}_{\frac{p}{p-1},\infty})}+\|R(\partial_ju,a)\|_{\widetilde L^1_t(\dot{B}^{2-\frac{2}{p}}_{\frac{p}{p-1},\infty})}
\lesssim
\int_0^t\bigl \|u\|_{\dot{H}^2}
\|\na a\|_{\dot{B}^{1-\frac{2}{p}}_{\frac{p}{p-1},\infty}}\,d\tau.
\end{align*}
By inserting the above estimate into \eqref{S4eq2}, we infer for $t>t_0$ with $t_0\leq T_1$
\begin{equation}\label{T-2-11}
\begin{split}
\|\nabla a\|_{L^\infty_t(\dot{B}^{1-\frac{2}{p}}_{\frac{p}{p-1},\infty})}
\lesssim&
\|\nabla a(t_0)\|_{\dot{B}^{1-\frac{2}{p}}_{\frac{p}{p-1},\infty}}+\int_{t_0}^t\bigl(\|\nabla u\|_{L^\infty}+\|u\|_{\dot H^2}\bigr)
 \|\na a\|_{\dot{B}^{1-\frac{2}{p}}_{\frac{p}{p-1},\infty}}\,d\tau.
\end{split}
\end{equation}
 We thus get, by applying Gronwall's inequality to \eqref{T-2-11}, that for $t\geq t_0$
\begin{equation*}
\|\nabla a\|_{L^\infty_t(\dot{B}^{1-\frac{2}{p}}_{\frac{p}{p-1},\infty})}
\lesssim
 \|\nabla a(t_0)\|_{\dot{B}^{1-\frac{2}{p}}_{\frac{p}{p-1},\infty}}
\,e^{C(\|\nabla u\|_{L^1_t(L^\infty)}+\|u\|_{L^1_t(\dot H^2)})}\leq C(t),
\end{equation*}
where we used \eqref{S4eq1} in the last step.

\no{\bf Step 2.2}
The uniqueness of solution in case $p=2$.

We shall use the Lagrangian approach to prove the uniqueness (see \cite{DM1, PZZ} for instance).
Let $(\rho,u,\na\Pi)$ be a global solution of \eqref{1.2} obtained in
Theorem \ref{thm1.1}. Then due to $u\in L^1_{loc}(\R^+;Lip),$ we
can define the trajectory $X(t,y)$ of $u(t,x)$ by
$$
\partial_t X(t,y)=u(t,X(t,y)),\qquad X(0,y)=y,
$$
which leads to the following relation between the Eulerian
coordinates $x$ and the Lagrangian coordinates $y$:
\beq\label{u.14}
x=X(t,y)=y+\int_0^tu(\tau, X(\tau,y))\mathrm{d}\tau.
\eeq
Moreover, we can take $T$ to be so small that
\begin{align}
\label{u.15} \int_0^T\|\na
u(t,\cdot)\|_{L^\infty}\mathrm{d} t\leq \f12.
\end{align}
Then for $t\leq T,$ $X(t,y): \R^2\rightarrow\R^2,$
is invertible with respect to $y$ variables, and we denote
$Y(t,\cdot)$ to be its inverse mapping.
Let
\begin{align*}
\bar{u}(t,y)\eqdefa u(t,x)=u(t,X(t,y))\quad\mbox{and}\quad
\bar{\Pi}(t,y)\eqdefa \Pi(t,X(t,y)).
\end{align*}
Then similar to \cite{DM1}, one has
\beq\label{u.15ag}
\bar{u}\in
\wt{L}^\infty_{\mbox{loc}}(\R^+;\dB^0_{2,1}) \quad\mbox{and}\quad
\p^2\bar u, \p_t\bar u, \na\bar{\Pi}\in
L^1_{\mbox{loc}}(\R^+;\dB^0_{2,1}),
\eeq
and
\begin{equation}\label{u.16}
\begin{split}
&\partial_t
\bar{u}(t,y)=(\partial_t u+ u\cdot\nabla u)(t,x),\quad \partial_{x_i} u_j(t,x)= \sum_{k=1}^2\partial_{y_k}
\bar{u}_j(t,y)\partial_{x_i} y_k,
 \end{split}
 \end{equation}
for $ x=X(t,y),\
y=Y(t,x).$

Let $A(t,y)\eqdefa (\na
X(t,y))^{-1}=\na_x Y(t,x),$ then we have
\begin{align*}
 \nabla_x u(t,x)= A(t,y)^T\nabla_y \bar{u}(t,y)\quad\mbox{ and}\quad
\dive_x u(t,x)=\dive( A(t,y) \bar{u}(t,y)),
\end{align*}
and $(\bar{u},
\na_y\bar{\Pi})$
solves
\begin{equation}\label{u.18}
\left\{\begin{array}{l} \displaystyle
\rho_{0}\pa_t\bar{u}-\D_y\bar{u}
+\grad_y\bar{\Pi}=\dv\bigl((A A^T-Id)\na_y\bar{u}\bigr)+(Id-A)^T\na_y\bar{\Pi}, \\
\displaystyle \dv_y\,\bar{u} = \dv\bigl((Id- A)\bar{u}\bigr).
\end{array}\right.
\end{equation}

Before proceeding, we recall the following two lemmas from \cite{DM1}:
\begin{lem}\label{lem2.9}
{\sl Let $p\in[1,\infty[$, then under the assumption that $\int_{0}^{T}\|\na_y\bar{u}\|_{\dot{B}_{p,1}^{\frac{2}{p}}}\mathrm{d}\tau\leq c$ ($c=c(p)>0 $ is a constant), for all $t\in[0,T],$ one has
\begin{align*}
&\|Id-A\|_{\widetilde L^\infty_t(\dot{B}_{p,1}^{\frac{2}{p}})}
\lesssim
\|\na_y\bar{u}\|_{L_{t}^{1}(\dot{B}_{p,1}^{\frac{2}{p}})},\
\|\partial_{t}A\|_{\dot{B}_{p,1}^{\frac{2}{p}}}
\lesssim
\|\na_y \bar{u}\|_{\dot{B}_{p,1}^{\frac{2}{p}}},\
\|AA^{T}-Id\|_{\dot{B}_{p,1}^{\frac{2}{p}}}
\lesssim
\|\na_y\bar{u}\|_{L_{t}^{1}(\dot{B}_{p,1}^{\frac{2}{p}})}.
\end{align*}
}
\end{lem}
\begin{lem}\label{lem2.10}
{\sl Let $p\in[1,\infty[$,  and $\bar{u}^1$ and $\bar{u}^2$ be two vector fields satisfying $\int_{0}^{T}\|\na_y\bar{u}^{i}(\tau,\cdot)\|_{\dot{B}_{p,1}^{\frac{2}{p}}}\mathrm{d}\tau\leq c$ ($c=c(p)>0 $ is a constant and $i=1,2$) and
$\delta\bar{u}\eqdefa \bar{u}^2-\bar{u}^1$, then  for all $t\in[0,T]$, we have
\beno
\|A^2-A^1\|_{\widetilde L_{t}^{\infty}(\dot{B}_{p,1}^{\frac{2}{p}})}
\lesssim
\|\na_y\delta\bar{u}\|_{L_{t}^{1}(\dot{B}_{p,1}^{\frac{2}{p}})}
\andf
\|\partial_{t}(A^2-A^1)\|_{L_{t}^{1}(\dot{B}_{p,1}^{\frac{2}{p}})}
\lesssim
\|\na_y\delta\bar{u}\|_{L_{t}^{1}(\dot{B}_{p,1}^{\frac{2}{p}})}.
\eeno
}
\end{lem}

Now let $(\rho^i, u^i, \na\Pi^i),$ $i=1,2,$ be two solutions of
\eqref{1.2} which satisfy the regularity properties of
Theorem \ref{thm1.1}. Let $(\bar{u}^i, A^i, \bar{\Pi}^i),$ $i=1,2,$
be given by \eqref{u.14}-\eqref{u.16},  we set
$$
(\delta A,\delta\bar u,\nabla\delta\bar\Pi) \eqdefa (A^2-A^1,\bar
u^2-\bar u^1,\nabla\bar\Pi^2-\nabla\bar\Pi^1).
$$
Then it follows from \eqref{u.18} that the system for $(\delta \bar
u, \nabla\delta\bar \Pi)$ reads
\begin{equation*}
\begin{cases}
&\pa_t\delta\bar u-
\bigl(1+a_0\bigr)
\bigl(\D_y\d\bar u+\grad_y\delta\bar\Pi\bigr)=\d\bar F,\\
&\dv_y\,\delta\bar u = \na\delta\bar u:(Id-A^2)-\na
u^1:\delta A =\dv_y\bigl((Id-A^2)\delta \bar u-\delta A \bar u^1\bigr)\eqdefa\dv_yg,\\
& \Delta\Phi=\dv_yg,\\
&\d\bar{u}|_{t=0}=0,
\end{cases}
\end{equation*}
where
$$
\begin{aligned}
\d\bar F\eqdefa &
 \bigl(1+a_0\bigr)
(Id-A^2)^T\grad_y\delta\bar\Pi
- \bigl(1+a_0\bigr)\bigl(\delta A^T\nabla_y\bar\Pi^1\bigr)
\\&
+\bigl(1+a_0\bigr)
\dv_y\bigl((A^2 (A^2)^T-Id) \nabla_y\delta \bar u + (A^2
(A^2)^T-A^1 (A^1)^T)\nabla_y\bar u_1\bigr).
\end{aligned}
$$
We get, by applying Proposition \ref{prop2.4} with $p=2$, that
\beq\label{S4eq4}
\begin{split}
\|\delta&\bar u\|_{\widetilde{L}_t^\infty(\dot{B}_{2,1}^0)}
+\|(\partial_t \delta\bar u,\nabla_y^2\delta\bar u,\nabla_y\delta\bar\Pi)\|_{{L}_t^1(\dot{B}_{2,1}^0)}
\lesssim
\|\d\bar F\|_{L_t^1(\dot B^0_{2,1})}
+\|\nabla\dv_yg\|_{L_t^1(\dot B^0_{2,1})}
\\
&+\|\partial_tg\|_{L_t^1(\dot B^0_{2,1})}+\bigl(2^m\sqrt{t}+2^{2m}t\bigr)\bigl(\|\partial_tg\|_{L^1_t(L^2)}
+\|\d\bar F\|_{L^1_t(L^2)}\bigr)+\|\nabla\Phi\|_{\widetilde{L}_t^\infty(\dot{B}_{2,1}^{0})\cap {L}_t^1(\dot{B}_{2,1}^{2})}.
\end{split}
\eeq
It follows from Lemmas \ref{lem2.9} and \ref{lem2.10}, $\frac{1}{2}<\frac{1}{2}+\frac{1}{\lambda} \leq 1$, and the law of product in Besov spaces in Proposition \ref{prop2.2} that
\begin{align*}
\|\d\bar F\|_{L_t^1(\dot B^0_{2,1})}
\lesssim &\bigl(1+\|a_0\|_{\dot B^{\frac{2}{\lambda}}_{\lambda,\infty}\cap L^\infty}\bigr)\Bigl(\|(Id-A^2)^T\|_{L_t^\infty(\dot B^1_{2,1})}\|\nabla_y\delta\bar\Pi\|_{L_t^1(\dot B^0_{2,1})}
\\&+\|\delta A^T\|_{L_t^\infty(\dot B^1_{2,1})}\|\nabla_y\bar\Pi^1\|_{L_t^1(\dot B^0_{2,1})}
+\|(A^2 (A^2)^T-Id)\|_{L_t^\infty(\dot B^1_{2,1})}\| \nabla_y\delta \bar u\|_{L_t^1(\dot B^1_{2,1})}\\
& + \|(A^2
(A^2)^T-A^1 (A^1)^T)\|_{L_t^\infty(\dot B^1_{2,1})}\|\nabla_y\bar u_1\|_{L_t^1(\dot B^1_{2,1})}\Bigr)\\
\lesssim
&\|\nabla_y \bar u^2\|_{L_t^1(\dot B^1_{2,1})}\|\nabla_y\delta\bar\Pi\|_{L_t^1(\dot B^0_{2,1})}
+\|(\nabla_y \bar u^1,\nabla_y \bar u^2)\|_{L_t^1(\dot B^1_{2,1})}\| \nabla_y\delta \bar u\|_{L_t^1(\dot B^1_{2,1})}\\
&+\|\nabla_y \bar \Pi^1\|_{L_t^1(\dot B^0_{2,1})}\| \nabla_y\delta \bar u\|_{L_t^1(\dot B^1_{2,1})},
\end{align*}
and
\begin{align*}
 \|\nabla\dv_yg\|_{L_t^1(\dot B^0_{2,1})}&=\|\dv_yg\|_{L_t^1(\dot B^1_{2,1})}\\
&\lesssim\|\na_y\delta\bar u\|_{L_t^1(\dot B^1_{2,1})}\|Id-A^2\|_{L_t^\infty(\dot B^1_{2,1})}+\|\na_y
u^1\|_{L_t^1(\dot B^1_{2,1})}\|\delta A\|_{L_t^\infty(\dot B^1_{2,1})}\\
&\lesssim\|(\nabla_y \bar u^1,\nabla_y \bar u^2)\|_{L_t^1(\dot B^1_{2,1})}\| \nabla_y\delta \bar u\|_{L_t^1(\dot B^1_{2,1})}
\end{align*}
and
$$
\begin{aligned}
\|\partial_tg\|_{L_t^1(\dot B^0_{2,1})}
\lesssim
&\|\partial_tA^2\,\delta\bar u\|_{L_t^1(\dot B^0_{2,1})}
+\|(Id-A^2)\,\partial_t\delta\bar u\|_{L_t^1(\dot B^0_{2,1})}
+\|\partial_t\delta A\,\bar u^1\|_{L_t^1(\dot B^0_{2,1})}
+\|\delta A\,\partial_t\bar u^1\|_{L_t^1(\dot B^0_{2,1})}
\\
\lesssim
&\Bigl(\|(\nabla^2\bar u^2,\,\partial_t\bar u^1)\|_{L_t^1(\dot B^0_{2,1})}
+\| \bar u^1\|_{L_t^\infty(\dot B^0_{2,1})}^{\frac{1}{2}}
\| \bar u^1\|_{L_t^1(\dot B^2_{2,1})}^{\frac{1}{2}}\Bigr)
\\&
\times
\Bigl(\| \delta \bar u\|_{L_t^\infty(\dot B^0_{2,1})}
+\|(\nabla^2\delta \bar u,\,\partial_t\delta\bar u)\|_{L_t^1(\dot B^0_{2,1})}\Bigr).
\end{aligned}
$$
Observing that $\dot{B}_{2,1}^0\hookrightarrow\dot{B}_{2,2}^0$, we have
$$
\|\partial_tg\|_{L^1_t(L^2)}
+\|\d\bar F\|_{L^1_t(L^2)}\leq \|\partial_tg\|_{L^1_t(\dot{B}_{2,1}^0)}
+\|\d\bar F\|_{L^1_t(\dot{B}_{2,1}^0)}.
$$
As $\Phi(0)=0,$ there holds
$
\Phi(t)=\int_0^t\partial_{\tau}\Phi(\tau)d\tau,
$
so that one has
$$
\|\nabla\Phi\|_{\widetilde{L}_t^\infty(\dot{B}_{2,1}^{0})}
\lesssim
\|\partial_tg\|_{L^1_t(\dot B^0_{2,1})}.
$$
Finally let us turn to the estimate of
$\|\nabla\Phi\|_{{L}_t^1(\dot{B}_{2,1}^{2})}.$ Indeed it follows from the law of product that
$$
\begin{aligned}
\|\nabla\Phi\|_{{L}_t^1(\dot{B}_{2,1}^{2})}&\lesssim \|\dv_yg\|_{{L}_t^1(\dot{B}_{2,1}^{1})}
\lesssim
\|\na\delta\bar u\|_{L^1_t(\dot{B}_{2,1}^{1})}
\|Id-A^2\|_{L^\infty_t(\dot{B}_{2,1}^{1})}
+\|\na\bar u^1\|_{L^1_t(\dot{B}_{2,1}^{1})}
\|\delta A\|_{L^\infty_t(\dot{B}_{2,1}^{1})}
\\&
\lesssim
\|\na\bar u^2\|_{L^1_t(\dot{B}_{2,1}^{1})}
\|\na\delta\bar u\|_{L^1_t(\dot{B}_{2,1}^{1})}
+\|\na\bar u^1\|_{L^1_t(\dot{B}_{2,1}^{1})}
\|\na\delta\bar u\|_{L^1_t(\dot{B}_{2,1}^{1})}.
\end{aligned}
$$

By substituting the above estimates into \eqref{S4eq4}, we obtain
\begin{align*}
\|\delta\bar u&\|_{\widetilde{L}_t^\infty(\dot{B}_{2,1}^0)}
+\|(\partial_t \delta\bar u,\nabla_y^2\delta\bar u,\nabla_y\delta\bar\Pi)\|_{{L}_t^1(\dot{B}_{2,1}^0)}
\\
\leq
&C\bigl(\|(\nabla_y^2 \bar u^1,\,\nabla_y^2 \bar u^2,\,\nabla_y \bar \Pi^1,\,\partial_t\bar u^1)\|_{L_t^1(\dot B^0_{2,1})}
+\| \bar u^1\|_{L_t^\infty(\dot B^0_{2,1})}^{\frac{1}{2}}
\| \bar u^1\|_{L_t^1(\dot B^1_{2,1})}^{\frac{1}{2}}\bigr)
\\&
\qquad\times\bigl(\|\delta \bar u\|_{\widetilde L_t^\infty(\dot B^0_{2,1})}
+\|(\partial_t \delta\bar u,\nabla_y^2\delta\bar u,\nabla_y\delta\bar\Pi)\|_{{L}_t^1(\dot{B}_{2,1}^0)}\bigr).
\end{align*}
Then by taking $t$ to be so small that
$$C\bigl(\|(\nabla_y^2 \bar u^1,\,\nabla_y^2 \bar u^2,\,\nabla_y \bar \Pi^1,\,\partial_t\bar u^1)\|_{L_t^1(\dot B^0_{2,1})}+\| \bar u^1\|_{L_t^\infty(\dot B^0_{2,1})}^{\frac{1}{2}}
\| \bar u^1\|_{L_t^1(\dot B^1_{2,1})}^{\frac{1}{2}}\bigr)<1,$$
 we deduce the uniqueness part of Theorem \ref{thm1.1} for
the case when $p=2.$

\no{\bf Step 2.2}
The uniqueness of solution in case  $p\in]2,\infty[.$

Let $(\rho^i, u^i, \na\Pi^i),$ $i=1,2,$ be two solutions of
\eqref{1.2} which satisfy the regularity properties of
Theorem \ref{thm1.1}. We set $\rho\eqdefa\frac{1}{1+a}$ and denote \beno (\delta a,\delta
u,\nabla\delta\Pi) \eqdefa
(a^2-a^1,u^2-u^1,\nabla\Pi^2-\nabla\Pi^1).
\eeno
Then in view of \eqref{1.2}, the system
for $(\delta a,\delta u,\nabla\delta\Pi)$ reads
\beq\label{u.3}
\left\{\begin{array}{l}
\displaystyle \pa_t\delta a+u^2\cdot\nabla\delta a=-\delta u\cdot\nabla a^1,\\
\displaystyle \pa_t\delta u+(u^2\cdot\nabla)\delta u
-(1+a^2)(\Delta\delta u-\grad\delta\Pi)=\d F, \\
\displaystyle \dv\,\delta u = 0, \\
\displaystyle (\delta a,\delta u)|_{t=0}=(0,0),
\end{array}
\right. \eeq
where $\d F$ is determined by
$\d F\eqdefa -(\delta u\cdot\nabla)u^1+\delta a(\Delta u^1-\nabla\Pi^1).$

To estimate $\delta u,$ for integer $k\in\N,$ we first write the momentum equation of \eqref{u.3} as
\beq\label{Pre-1}
\begin{split}
&\partial_t\delta u+(u^2\cdot\nabla)\delta
u-(1+S_k a^2)(\Delta\delta u-\nabla\delta \Pi) =\d \cF_k
\with\\
&\qquad \d \cF_k
\eqdefa(a^2-S_k a^2)(\Delta\delta u-\nabla\delta \Pi) -\delta
u\cdot\nabla u^1+\delta a(\Delta u^1-\nabla\Pi^1).
\end{split}
\eeq
Then we get, by
applying Proposition
\ref{prop2.9} to \eqref{Pre-1}, that for all
$t\in ]0,T]$,
\begin{equation}\label{estimate-uniqueness-velosity-1a}
\begin{aligned}
\| \delta\,u\|_{L^\infty_t(B^{-1}_{2,\infty})}
+\|\delta\,u\|_{\widetilde L^1_t(B^{1}_{2, \infty})}\leq &
Ce^{C\bigl(t+t\|\nabla\,S_ka^2\|_{\widetilde L^\infty_t(B^1_{2,1})}^2+\|u^2\|_{L^1_t(B^1_{\infty,1})}\bigr)}
\\&\times\bigl(\|\d\cF_k\|_{\widetilde L^1_t(B^{-1}_{2,\infty})}
+\|S_k\nabla a^2\|_{\widetilde L^\infty_t(\dot{B}^{1}_{2,1})}
\|\nabla\delta\Pi\|_{\widetilde L^1_t(B^{-1}_{2,\infty})}\bigr).
\end{aligned}
\end{equation}
Notice that
\begin{equation*}
  \begin{split}
  \|\nabla\,S_ka^2\|_{\widetilde L^\infty_t(B^1_{2,1})}
&\lesssim
\|\nabla\,S_ka^2\|_{\widetilde L^\infty_t(L^2)}+
\|\nabla\,S_ka^2\|_{\widetilde L^\infty_t(\dot{B}^1_{2,1})}
 \\
&\lesssim \|a^2\|_{\widetilde L^\infty_t(\dot{B}^1_{2,1})}+ 2^k\|a^2\|_{\widetilde L^\infty_t(\dot{B}^1_{2,1})}\lesssim 2^k\|a^2\|_{\widetilde L^\infty_t(\dot{B}^1_{2,1})}.
  \end{split}
\end{equation*}
By substituting the above estimate into \eqref{estimate-uniqueness-velosity-1a} and using \eqref{space-soln-exist},  we obtain
\begin{equation}\label{estimate-uniqueness-velosity-1}
\begin{aligned}
\|\delta u\|_{{L}^{\infty}_t(B^{-1}_{2,\infty})}&+\|\delta
u\|_{\widetilde{L}^1_t(B^{1}_{2, \infty})}
\leq
Ce^{Ct2^k}
\bigl(\|\nabla\delta\Pi\|_{\widetilde L^1_t(B^{-1}_{2,\infty})}
+\|\d\cF_k\|_{\widetilde L^1_t(B^{-1}_{2,\infty})}\bigr).
\end{aligned}
\end{equation}

On the other hand, we get, by applying  $\dv$ to the momentum equation of \eqref{u.3}, that
\begin{equation}\label{est-uniq-2}
\dv\bigl((1+a^2)\nabla\delta\Pi\bigr)=\dv\,G
\end{equation}
with
$$
G\eqdefa (a^2-S_ma^2)\Delta\delta u+S_ma^2\Delta\delta u-\delta u\cdot\nabla
u^1-u^2\cdot\nabla\delta u +\delta a(\Delta u^1-\nabla\Pi^1)\eqdefa \sum_{\ell=1}^5\mbox{I}_{\ell}.$$

 We deduce from Propositions \ref{prop2.3} and \ref{prop3.2} that, for any small constant $c_0>0$,
  there exist sufficiently large $j_0 \in \mathbb{N}$ and a positive existence time $T_2$ such that
\beq \label{small-refe-1}
\|a^2-S_ja^2\|_{\widetilde L^\infty_{T_2}(B^1_{2,1})}\leq\|a^2-\dot S_ja^2\|_{\widetilde L^\infty_{T_2}(\dot B^1_{2,1})}\leq c_0,\quad\mbox{
for any}\ j\geq\,j_0.\eeq
Then we get, by applying Proposition \ref{prop-H-negtive} to \eqref{est-uniq-2}, that
\begin{equation}\label{4.10}
\begin{split}
\|\grad\delta\Pi\|_{\widetilde{L}^{1}_{t} ({B}^{-1}_{2, \infty})}
\lesssim
 \bigl(1+2^{j}\|a^2\|_{\widetilde L^\infty_t(\dot{B}^{1}_{2,1})}
 \bigl(1+\|a^2\|_{\widetilde L^\infty_t(\dot{B}^{1}_{2,1})}\bigr)\bigr)
 \bigl(\|G\|_{\widetilde L^1_t(B^{-2}_{2,\infty})}
 +\|\dv\,G\|_{\widetilde L^1_t(B^{-2}_{2,\infty})}\bigr).
\end{split}
\end{equation}
While it follows from  Lemma \ref{lem2.1} and product laws in Besov spaces in Proposition \ref{prop2.2} that
\begin{align*}
\|\mbox{I}_1\|_{\widetilde L^1_t(B^{-2}_{2,\infty})}
+\|\dv\,\mbox{I}_1\|_{\widetilde L^1_t(B^{-2}_{2,\infty})}
\lesssim&
\|\mbox{I}_1\|_{\widetilde L^1_t(B^{-1}_{2,\infty})}
\lesssim
\|a^2-S_ma^2\|_{\widetilde L^\infty_t(B^1_{2,1})}
\|\delta u\|_{\widetilde L^1_t(B^1_{2,\infty})}
\end{align*}
and
\begin{align*}
&\|\mbox{I}_2\|_{\widetilde L^1_t(B^{-2}_{2,\infty})}
+\|\dv\,\mbox{I}_2\|_{\widetilde L^1_t(B^{-2}_{2,\infty})}\\
&\lesssim
\|T_{S_ma^2}\Delta\delta u\|_{\widetilde L^1_t(B^{-2}_{2,\infty})}
+\|T_{\Delta\delta u}S_ma^2\|_{\widetilde L^1_t(B^{-2}_{2,\infty})}\\
&\qquad +\|R(S_ma^2,\Delta\delta u)\|_{\widetilde L^1_t(B^{-1}_{2,\infty})}
+\|T_{\nabla S_ma^2}\Delta\delta u\|_{\widetilde L^1_t(B^{-2}_{2,\infty})}+\|T_{\Delta\delta u}\nabla S_ma^2\|_{\widetilde L^1_t(B^{-2}_{2,\infty})}\\
&
\lesssim
(\|S_ma^2\|_{L^\infty_t(L^\infty)}+  2^{m}\|\nabla\,S_ma^2\|_{L^\infty_t(B^0_{2, 1})})
\|\delta u\|_{\widetilde L^1_t(B^0_{2,\infty})} \lesssim 2^{m}\|a^2\|_{\widetilde L^\infty_t(\dot{B}^{1}_{2,1})}\|\delta u\|_{\widetilde L^1_t(B^0_{2,\infty})}.
\end{align*}
Along the same line, one has
\begin{align*}
\|(\mbox{I}_3,\,\dv\,&\mbox{I}_3)\|_{\widetilde L^1_t(B^{-2}_{2,\infty})}+\|\mbox{I}_4\|_{\widetilde L^1_t(B^{-2}_{2,\infty})}
+\|\dv\mbox{I}_4\|_{\widetilde L^1_t(B^{-2}_{2,\infty})}\\
\lesssim &
\|\mbox{I}_3\|_{\widetilde L^1_t(B^{-1}_{2,\infty})}
+\|T_{u^2}\nabla\delta u\|_{\widetilde L^1_t(B^{-2}_{2,\infty})}
+\|T_{\nabla\delta u}u^2\|_{\widetilde L^1_t(B^{-2}_{2,\infty})}+\|R(u^2, \delta u)\|_{\widetilde L^1_t(B^{-1}_{2,\infty})}
\\&+\|T_{\nabla u^2}\nabla\delta u\|_{\widetilde L^1_t(B^{-2}_{2,\infty})}
+\|T_{\nabla\delta u}\nabla u^2\|_{\widetilde L^1_t(B^{-2}_{2,\infty})}
+\|R(\partial_{\ell}u^2, \delta u_{\ell})\|_{\widetilde L^1_t(B^{-1}_{2,\infty})}
\\
\lesssim &
\int_0^t\|\delta u\|_{B^{-1}_{2,\infty}}
\bigl(\|u^1\|_{B^1_{\infty,1}}+\|u^2\|_{B^1_{\infty,1}}\bigr)\mathrm{d}\tau.
\end{align*}
To deal with $\mbox{I}_5$, we write by using Bony's decomposition that
\begin{align*}
\mbox{I}_5&=\delta a(\Delta u^1-\nabla \Pi^1)=T_{\Delta u^1-\nabla \Pi^1}\delta a+T_{\delta a}(\Delta u^1-\nabla \Pi^1)+R(\delta a,\Delta u^1-\nabla \Pi^1).
\end{align*}
It follows from product laws in Besov spaces in Proposition \ref{prop2.2} again that
$$
\begin{aligned}
\|T_{\Delta u^1-\nabla \Pi^1}\delta a\|_{\widetilde L^1_t(B^{-1}_{2,\infty})}
&\lesssim
\int_0^t\|\delta a\|_{B^0_{2,\infty}}
\bigl(\|\Delta u^1\|_{B^{-1}_{\infty,\infty}}
+\|\nabla\Pi^1\|_{B^{-1}_{\infty,\infty}}\bigr)\mathrm{d}\tau
\\
&\lesssim
\int_0^t\|\delta a\|_{B^0_{2,\infty}}
\bigl(\|\Delta\, u^1\|_{B^{-1+\frac{2}{p}}_{p,1}}
+\|\nabla\Pi^1\|_{B^{-1+\frac{2}{p}}_{p,1}}\bigr)\mathrm{d}\tau,
\end{aligned}
$$
and
\begin{align*}
\|T_{\delta a}(\Delta u^1-\nabla \Pi^1)\|_{\widetilde L^1_t(B^{-1}_{2,\infty})}
\lesssim
\int_0^t\|\delta a\|_{B^0_{2,\infty}}
\bigl(\|\Delta u^1\|_{B^{-1+\frac{2}{p}}_{p,1}}
+\|\nabla\Pi^1\|_{B^{-1+\frac{2}{p}}_{p,1}}\bigr)\,\mathrm{d}\tau,
\end{align*}
and
$$
\begin{aligned}
\|R(\delta a,\Delta u^1-\nabla \Pi^1)\|_{\widetilde L^1_t(B^{-1}_{2,\infty})}
&\lesssim
\|R(\delta a,\Delta u^1-\nabla \Pi^1)\|_{L^1_t(B^{0}_{1,\infty})}
\\
&\lesssim
\int_0^t\|\delta a\|_{B^{1-\frac{2}{p}}_{\frac{p}{p-1},\infty}}
\bigl(\|\Delta u^1\|_{B^{-1+\frac{2}{p}}_{p,1}}
+\|\nabla\Pi^1\|_{B^{-1+\frac{2}{p}}_{p,1}}\bigr)\,\mathrm{d}\tau.
\end{aligned}
$$
Since $p>2,$  $\dot B^{-1+\frac{2}{p}}_{p,1}\hookrightarrow
B^{-1+\frac{2}{p}}_{p,1},$
as a consequence, we have
$$
\begin{aligned}
\|(\mbox{I}_5,\,\dv\,\mbox{I}_5)\|_{\widetilde L^1_t(B^{-2}_{2,\infty})}&\lesssim
\|\delta a(\Delta u^1-\nabla \Pi^1)\|_{\widetilde L^1_t(B^{-1}_{2,\infty})}
\lesssim
\int_0^t\|\delta a\|_{B^{1-\frac{2}{p}}_{\frac{p}{p-1},\infty}}
 \|(\Delta u^1,\,\nabla\Pi^1)\|_{\dot B^{-1+\frac{2}{p}}_{p,1}} \,d\tau.
\end{aligned}
$$
where we used the fact that ($p \geq 2$)
$$
\|\Delta u^1\|_{B^{-1+\frac{2}{p}}_{p,1}}
+\|\nabla\Pi^1\|_{B^{-1+\frac{2}{p}}_{p,1}}
\lesssim
\|\Delta u^1\|_{\dot B^{-1+\frac{2}{p}}_{p,1}}
+\|\nabla\Pi^1\|_{\dot B^{-1+\frac{2}{p}}_{p,1}}.
$$
By summarizing the above estimates, we obtain
\begin{equation*}\label{uniqueness-pressure-elli-1-a}
\begin{split}
\|&G \|_{\widetilde L^1_t(B^{-2}_{2,\infty})}
 +\|\dv\,G\|_{\widetilde L^1_t(B^{-2}_{2,\infty})}
\lesssim
 \|a^2-S_ma^2\|_{\widetilde L^\infty_t(B^1_{2,1})}
\|\delta u\|_{\widetilde L^1_t(B^1_{2,\infty})}+2^m\|\delta u\|_{\widetilde L^1_t(B^0_{2,\infty})}
\\&
 +\int_0^t\|\delta u\|_{B^{-1}_{2,\infty}}
 \|(u^1,\,u^2)\|_{B^1_{\infty,1}} \,d\tau+\int_0^t\|\delta a\|_{B^{1-\frac{2}{p}}_{\frac{p}{p-1},\infty}}
 \|(\Delta u^1,\,\nabla\Pi^1)\|_{\dot{B}^{-1+\frac{2}{p}}_{p,1}} \,d\tau.
\end{split}
\end{equation*}
By substituting the above estimates into \eqref{4.10}, we arrive at
\begin{equation}\label{uniqueness-pressure-elli-1-a}
\begin{split}
\|\grad&\delta\Pi\|_{\widetilde{L}^{1}_{t} ({B}^{-1}_{2, \infty})}
\lesssim
 \bigl(1+2^{j}\|a^2\|_{\widetilde L^\infty_t(\dot{B}^{1}_{2,1})}\bigl(1+\|a^2\|_{\widetilde L^\infty_t(\dot{B}^{1}_{2,1})}\bigr)\bigr)\\
 &\times
 \Bigl(\|a^2-S_ma^2\|_{\widetilde L^\infty_t(B^1_{2,1})}
\|\delta u\|_{\widetilde L^1_t(B^1_{2,\infty})}+2^m\|\delta u\|_{\widetilde L^1_t(B^0_{2,\infty})}
\\&
 +\int_0^t\|\delta u\|_{B^{-1}_{2,\infty}}
 \|(u^1,\,u^2)\|_{B^1_{\infty,1}} \,d\tau+\int_0^t\|\delta a\|_{B^{1-\frac{2}{p}}_{\frac{p}{p-1},\infty}}
 \|(\Delta u^1,\,\nabla\Pi^1)\|_{\dot{B}^{-1+\frac{2}{p}}_{p,1}} \,d\tau\Bigr).
\end{split}
\end{equation}

To handle the estimate of $\|\d\cF_k\|_{\widetilde L^1_t(B^{-1}_{2,\infty})}$, we get, by applying the law of product in Besov spaces, Proposition \ref{prop2.2}, that
\begin{equation}\label{unique-pressure-elli-222}
\begin{split}
\|\d\cF_k\|_{\widetilde L^1_t(B^{-1}_{2,\infty})}
&\lesssim
\|a^2-S_k a^2\|_{\widetilde L^\infty_t(B^1_{2,1})}
\bigl(\|\Delta\delta u\|_{\widetilde L^1_t(B^{-1}_{2,\infty})}+
\|\nabla\delta \Pi\|_{\widetilde L^1_t(B^{-1}_{2,\infty})}\bigr)
\\&
+\int_0^t\|\delta u\|_{B^{-1}_{2,\infty}}\|u^1\|_{B^1_{\infty,1}}\,d\tau
+\int_0^t\|\delta a\|_{B^{1-\frac{2}{p}}_{\frac{p}{p-1},\infty}}
\bigl(\|u^1\|_{\dot B^{1+\frac{2}{p}}_{p,1}}
+\|\nabla\Pi^1\|_{\dot B^{-1+\frac{2}{p}}_{p,1}}\bigr)\,d\tau.
\end{split}
\end{equation}
For $m,\,k\geq j_0$ and $t\leq T_2,$ we get,
by substituting  \eqref{small-refe-1}, \eqref{uniqueness-pressure-elli-1-a} and \eqref{unique-pressure-elli-222} into  \eqref{estimate-uniqueness-velosity-1},  that
\begin{equation}\label{unique-pressure-elli-Pi-1}
\begin{aligned}
&\|\delta u\|_{{L}^{\infty}_t(B^{-1}_{2,\infty})}+\|\delta
u\|_{\widetilde{L}^1_t(B^{1}_{2, \infty})}+\|\nabla\delta\Pi\|_{\widetilde{L}^{1}_{t} ({B}^{-1}_{2, \infty})}\\
&
\leq
Ce^{Ct2^k}
  (2^{j} \|a^2-S_ma^2\|_{\widetilde L^\infty_t(B^1_{2,1})}+\|a^2-S_k a^2\|_{\widetilde L^\infty_t(B^1_{2,1})})
\bigl(\|\delta u\|_{\widetilde L^1_t(B^1_{2,\infty})}+\|\nabla\delta\Pi\|_{\widetilde{L}^{1}_{t} ({B}^{-1}_{2, \infty})}\bigr) \\
&\quad
 +Ce^{Ct2^k}2^{j}\bigg(\int_0^t\|\delta u\|_{B^{-1}_{2,\infty}}
 \|(u^1,\,u^2)\|_{B^1_{\infty,1}} \,d\tau+\int_0^t\|\delta a\|_{B^{1-\frac{2}{p}}_{\frac{p}{p-1},\infty}}
 \|(\Delta u^1,\,\nabla\Pi^1)\|_{\dot{B}^{-1+\frac{2}{p}}_{p,1}} \,d\tau\bigg)\\
 &\quad +Ce^{Ct2^k}\,2^{m+j}\|\delta u\|_{\widetilde L^1_t(B^0_{2,\infty})}.
\end{aligned}
\end{equation}
Notice that
 $ \|\delta u\|_{\widetilde L^1_t(B^0_{2,\infty})}
 \lesssim
 \|\delta u\|_{\widetilde L^1_t(B^{-1}_{2,\infty})}^{\f12}
\|\delta u\|_{\widetilde L^1_t(B^{1}_{2,\infty})}^{\f12}$,  we infer
\begin{equation}\label{unique-pressure-elli-Pi-2}
Ce^{Ct2^k}\,2^{m+j}
 \|\delta u\|_{\widetilde L^1_t(B^0_{2,\infty})}
\leq
\frac{1}{2}\|\delta u\|_{\widetilde L^1_t(B^{1}_{2,\infty})}
+Ce^{Ct2^k}\,2^{2m+2j}
\int_0^t\|\delta u\|_{B^{-1}_{2,\infty}}\,\mathrm{d}\tau,
\end{equation}
On the other hand, in view of the first equation in \eqref{u.3}, we deduce from  the classical estimate of the transport equation, that
\begin{align*}
\|\delta a\|_{L_t^\infty({B}^{1-\frac{2}{p}}_{\frac{p}{p-1},\infty})}
&\le
\|\delta u\cdot\nabla a_1\|_{\widetilde{L}_t^1({B}^{1-\frac{2}{p}}_{\frac{p}{p-1},\infty})}
e^{C\|\nabla u^2\|_{L^1_t(L^\infty)}}
\lesssim
\|\delta u\cdot\nabla a^1\|_{L_t^1(B^{1-\frac{2}{p}}_{\frac{p}{p-1},\infty})}.
\end{align*}
Thanks to Proposition \ref{prop2.2}, one has
\begin{align*}
&\|\delta u\cdot\nabla a^1\|_{\widetilde{L}_t^1(\dot{B}^{1-\frac{2}{p}}_{\frac{p}{p-1},\infty})} \lesssim (\|\delta u\|_{L_t^1(L^\infty)}
+\|\delta u\|_{\widetilde{L}_t^1(\dot{B}^1_{2,\infty})})\|\nabla a^1\|_{L_t^{\infty}(\dot{B}^{1-\frac{2}{p}}_{\frac{p}{p-1},\infty})},\\
&\|\delta u\cdot\nabla a^1\|_{\widetilde{L}_t^1(\dot{B}^{0}_{\frac{p}{p-1},1})} \lesssim  \|\delta u\|_{L_t^1(\dot{B}^{\frac{2}{p}}_{2, 1})} \|\nabla a^1\|_{L_t^{\infty}(\dot{B}^{1-\frac{2}{p}}_{\frac{p}{p-1},\infty})},
\end{align*}
which along with \eqref{interpo-complex-1} yields
\begin{align*}
&\|\delta u\cdot\nabla a^1\|_{\widetilde{L}_t^1({B}^{1-\frac{2}{p}}_{\frac{p}{p-1},\infty})} \lesssim \|\delta u\cdot\nabla a^1\|_{\widetilde{L}_t^1(\dot{B}^{1-\frac{2}{p}}_{\frac{p}{p-1},\infty})} +\|\delta u\cdot\nabla a^1\|_{\widetilde{L}_t^1(\dot{B}^{0}_{\frac{p}{p-1},1})}\\
&\lesssim (\|\delta u\|_{L_t^1(L^\infty)}
+\|\delta u\|_{\widetilde{L}_t^1(\dot{B}^1_{2,\infty})}+\|\delta u\|_{L_t^1(\dot{B}^{\frac{2}{p}}_{2, 1})})\|\nabla a^1\|_{L_t^{\infty}(\dot{B}^{1-\frac{2}{p}}_{\frac{p}{p-1},\infty})}\\
&\lesssim (\|\delta u\|_{L_t^1(L^\infty)}
+\|\delta u\|_{\widetilde{L}_t^1({B}^1_{2,\infty})})\|\nabla a^1\|_{L_t^{\infty}(\dot{B}^{1-\frac{2}{p}}_{\frac{p}{p-1},\infty})}.
\end{align*}
Hence, we obtain
\begin{equation}\label{diff-density-esti-1}
\begin{aligned}
\|\delta a\|_{L_t^\infty(B^{1-\frac{2}{p}}_{\frac{p}{p-1},\infty})}
&\lesssim
\bigl(\|\delta u\|_{L_t^1(L^\infty)}
+\|\delta u\|_{\widetilde L_t^1(B^1_{2,\infty})}\bigr)
\|\nabla a^1\|_{L_t^\infty(\dot{B}^{1-\frac{2}{p}}_{\frac{p}{p-1},\infty})}\\
&\lesssim
\|\delta u\|_{L_t^1(L^\infty)}
+\|\delta u\|_{\widetilde L_t^1(B^1_{2,\infty})}.
\end{aligned}
\end{equation}
Inserting \eqref{unique-pressure-elli-Pi-2} and \eqref{diff-density-esti-1} into \eqref{unique-pressure-elli-Pi-1} yields
\begin{equation*}
\begin{aligned}
&\|\delta u\|_{{L}^{\infty}_t(B^{-1}_{2,\infty})}+\|\delta
u\|_{\widetilde{L}^1_t(B^{1}_{2, \infty})}+\|\nabla\delta\Pi\|_{\widetilde{L}^{1}_{t} ({B}^{-1}_{2, \infty})}\\
&
\leq
Ce^{Ct2^k}
  (2^{j} \|a^2-S_ma^2\|_{\widetilde L^\infty_t(B^1_{2,1})}+\|a^2-S_k a^2\|_{\widetilde L^\infty_t(B^1_{2,1})})
\bigl(\|\delta u\|_{\widetilde L^1_t(B^1_{2,\infty})}+\|\nabla\delta\Pi\|_{\widetilde{L}^{1}_{t} ({B}^{-1}_{2, \infty})}\bigr) \\
&\qquad
 +Ce^{Ct2^k}2^{j}\bigg(\int_0^t\|\delta u\|_{B^{-1}_{2,\infty}}
 \|(u^1,\,u^2)\|_{B^1_{\infty,1}} \,d\tau\\
 &\qquad \qquad\qquad\qquad+\int_0^t\bigl(\|\delta u\|_{L_\tau^1(L^\infty)}
+\|\delta u\|_{\widetilde L_\tau^1(B^1_{2,\infty})}\bigr)
 \|(\Delta u^1,\,\nabla\Pi^1)\|_{\dot{B}^{-1+\frac{2}{p}}_{p,1}} \,d\tau \bigg)\\
 &\qquad +Ce^{Ct2^k}\,2^{2m+2j}
\int_0^t\|\delta u\|_{B^{-1}_{2,\infty}}\,\mathrm{d}\tau.
\end{aligned}
\end{equation*}
 Therefore, for given $m,\,k\geq j_0$, taking $c_0>0$ in \eqref{small-refe-1} small enough, we deduce from \eqref{small-refe-1} that any $t\leq T_2,$
\begin{equation*}
\begin{aligned}
&\|\delta u\|_{{L}^{\infty}_t(B^{-1}_{2,\infty})}+\|\delta
u\|_{\widetilde{L}^1_t(B^{1}_{2, \infty})}+\|\nabla\delta\Pi\|_{\widetilde{L}^{1}_{t} ({B}^{-1}_{2, \infty})}\\
&
\leq
 Ce^{Ct2^k}2^{j}\bigg(\int_0^t\|\delta u\|_{B^{-1}_{2,\infty}}
 \|(u^1,\,u^2)\|_{B^1_{\infty,1}} \,d\tau\\
 &\qquad \qquad\qquad\qquad+\int_0^t\bigl(\|\delta u\|_{L_\tau^1(L^\infty)}
+\|\delta u\|_{\widetilde L_\tau^1(B^1_{2,\infty})}\bigr)
 \|(\Delta u^1,\,\nabla\Pi^1)\|_{\dot{B}^{-1+\frac{2}{p}}_{p,1}} \,d\tau \bigg)\\
 &\qquad +Ce^{Ct2^k}\,2^{2m+2j}
\int_0^t\|\delta u\|_{B^{-1}_{2,\infty}}\,\mathrm{d}\tau.
\end{aligned}
\end{equation*}
Then for $T_3\in (0, T_2]$ being small enough, we deduce that for all $t\in [0, T_3]$,
\begin{equation}\label{lip-est-0}
\begin{aligned}
\|\delta u\|_{{L}^{\infty}_t(B^{-1}_{2,\infty})}&+\|\delta
u\|_{\widetilde{L}^1_t(B^{1}_{2, \infty})}
\\&
\lesssim
\int_{0}^{t}\bigl(\|\delta u\|_{L_\tau^1(L^\infty)}
+\|\delta u\|_{\widetilde L_\tau^1(B^1_{2,\infty})}\bigr)
\bigl(\|u^1\|_{\dot B^{1+\frac{2}{p}}_{p,1}}
+\|\nabla\Pi^1\|_{\dot B^{-1+\frac{2}{p}}_{p,1}}\bigr)\,\mathrm{d}\tau.
\end{aligned}
\end{equation}
Let $N$ be an arbitrary positive integer which will be determined later on, we write
$$
\begin{aligned}
\|\delta u\|_{L^1_{\tau}(L^\infty)}
&\le
\|\delta u\|_{L^1_{\tau}(\dot B^0_{\infty,1})}\leq
\bigg(\sum_{-1\leq q\le N}
+\sum_{N+1\le q\le 2N}
+\sum_{q\geq 2N+1}\bigg)\|\Delta_q\delta u\|_{L^1_{\tau}(L^\infty)},
\end{aligned}
$$
from which and Lemma \ref{lem2.1}, we infer
$$
\begin{aligned}
\|\delta u\|_{L^1_{\tau}(L^\infty)}
&\lesssim
2^{N}\|\delta u\|_{L^1_{\tau}(L^2)}
+N\|\delta u\|_{\widetilde L^1_{\tau}( B^1_{2,\infty})}
+2^{-N}\|\nabla\delta u\|_{\widetilde{L}^1_{\tau}(L^\infty)}.
\end{aligned}
$$
If we choose $N$ so that
$$
N
\thicksim
\ln\Bigl(e+\frac{\|\delta u\|_{L^1_{\tau}(L^2)}
+\|\nabla\delta u\|_{L^1_{\tau}(L^\infty)}}
{\|\delta u\|_{\widetilde L^1_{\tau}(B^1_{2,\infty})}}\Bigr),
$$
we obtain
\beq\label{infty-est-1}
\begin{aligned}
\|\delta u\|_{L^1_{\tau}(L^\infty)}
&\lesssim
\|\delta u\|_{\widetilde L^1_{\tau}(B^1_{2,\infty})}
\ln\bigl(e+\frac{\|\delta u\|_{L^1_{\tau}(L^2)}+\|\nabla\delta u\|_{L^1_{\tau}(L^\infty)}}
{\|\delta u\|_{\widetilde L^1_{\tau}(B^1_{2,\infty})}}\bigr)\\
&\lesssim
\|\delta u\|_{\widetilde L^1_{\tau}(B^1_{2,\infty})}
\ln\Bigl(e+\sum\limits_{i=1}^2\frac{\tau\|u^i\|_{L^\infty_{t}(L^2)}
+\|\nabla u^i\|_{L^1_{\tau}(L^\infty)}}
{\|\delta u\|_{\widetilde L^1_{\tau}(B^1_{2,\infty})}}\Bigr).
\end{aligned}
\eeq

Notice that for $\alpha\geq 0$ and $x\in(0,1],$ there holds
$$
\ln(e+\alpha x^{-1})\le\ln(e+\alpha)(1-\ln x)\qquad\mbox{and}
\qquad x\le x(1-\ln x).
$$
Then by plugging \eqref{infty-est-1} into \eqref{lip-est-0}, we find
\begin{equation*}
\begin{aligned}
&\|\delta u\|_{{L}^{\infty}_t(B^{-1}_{2,\infty})}
+\|\delta u\|_{\widetilde{L}^1_t(B^{1}_{2, \infty})}\\
&\quad\lesssim\int_{0}^{t}\|\delta u\|_{\widetilde{L}^1_{\tau}(B^{1}_{2, \infty})}
\bigl(1-\ln\|\delta u\|_{\widetilde{L}^1_{\tau}(B^{1}_{2, \infty})}\bigr)
\bigl(\|u^1\|_{\dot B^{1+\frac{2}{p}}_{p,1}}
+\|\nabla\Pi^1\|_{\dot B^{-1+\frac{2}{p}}_{p,1}}\bigr)\,d\tau.
\end{aligned}
\end{equation*}
As $\int_0^1\frac{dx}{x(1-\ln x)}=+\infty,$ and $\|\Delta u^1\|_{L^2}+\|\nabla\Pi^1\|_{L^2}$ is locally integral in $t \in \mathbb{R}^+$,
we deduce from  Osgood's Lemma (see \cite{fleet} for instance) that $\delta u(t)=0$ for $t\leq T_3,$ which together with
\eqref{diff-density-esti-1} and \eqref{uniqueness-pressure-elli-1-a} implies that
$\delta a(t)=\delta\nabla \Pi(t)=0$ for all $t\in[0,T_3].$

The uniqueness of such solutions on the whole time interval $[0, +\infty)$ then follows by a bootstrap argument, which completes the proof of Theorem \ref{thm1.1}.
\end{proof}

Finally let us present the proof of Corollary \ref{DW}.

\begin{proof}[\bf Proof of Corollary \ref{DW}] We first deduce from Theorem \ref{thm1.1} that the  system \eqref{1.2}   has a unique global solution
$(\rho, u)$ so that
$$
\begin{aligned}
&\rho^{-1}-1 \in C([0,\infty[;\,\dot B^{-1+\frac{2}{p}}_{\frac{2p}{2-p},2}\cap\,L^\infty)
\cap L^\infty(\R_+;\,L^\infty),
\\&
u\in C([0,\infty[;\,\dot B^{0}_{2,1})
\cap L^1_{loc}(\R_+;\,\dot B^{2}_{2,1}),
\\&
\nabla\Pi\in L^1_{loc}(\R_+;\,\dot B^{0}_{2,1})
\andf \,
\partial_tu\in L^1_{loc}(\R_+;\,\dot B^{0}_{2,1}).
\end{aligned}
$$
It follows from \eqref{3.17} that
\beq \label{S4eq6}
\begin{aligned}
\|u\|_{\widetilde L^\infty_t(\dot B^{-1+\frac{2}{p}}_{p,1})}
+\|u\|_{L^1_t(\dot B^{1+\frac{2}{p}}_{p,1})}
\lesssim&
\|u_0\|_{\dot B^{-1+\frac{2}{p}}_{p,1}}
 +\sum_{q\in\Z}2^{q\bigl(\frac{2}{p}-1\bigr)}
 \|[\dot\Delta_q\mathbb{P}, u\cdot\nabla]u\|_{L^1_t(L^p)}
\\&
+\sum_{q\in\Z}2^{q\bigl(\frac{2}{p}-1\bigr)}
\|[\dot\Delta_q\mathbb{P}, {\rho}^{-1}]\bigl(\Delta u-\nabla\Pi\bigr)\|_{L^1_t(L^p)}.
\end{aligned}
\eeq
By applying classical commutator's estimate (see Lemma 2.100 and Remark 2.102 in \cite{BCD}), we have
\beq \label{S4eq7}
\sum_{q\in\Z}2^{q\bigl(\frac{2}{p}-1\bigr)}
 \|[\dot\Delta_q\mathbb{P}, u\cdot\nabla]u\|_{L^1_t(L^p)}
 \lesssim
 \int_0^t\|\nabla u(\tau)\|_{L^\infty}
 \|u(\tau)\|_{\dot B^{-1+\frac{2}{p}}_{p,1}}\,d\tau.
\eeq

To handle the last term in \eqref{S4eq6}, we write, by using homogeneous Bony's decomposition, that
(with $f=\Delta u-\nabla\Pi$)
\begin{align*}
[\dot\Delta_q\mathbb{P},a]f
=\dot\Delta_q\mathbb{P}T_{f}a
+\dot\Delta_q\mathbb{P}R(a,f)
-T'_{\dot\Delta_q\mathbb{P}f}a
-[\dot\Delta_q\mathbb{P},T_a]f.
\end{align*}
Due to $p\in]1,2[,$ it follows from  Lemma \ref{lem2.1} and the law of product in Besov spaces that
\begin{align*}
&\|\dot\Delta_q\mathbb{P}T_{f}a\|_{L^1_t(L^p)}
\lesssim d_q2^{(1-\frac{2}{p})q}
\|a\|_{\widetilde L^\infty_t(\dot{B}_{\frac{2p}{2-p},1}^{-1+\frac{2}{p}})}
\|f\|_{L^1_t(L^2)},\\
&\|\dot\Delta_q\mathbb{P}R(a,f)\|_{L^1_t(L^p)}
\lesssim d_q2^{(1-\frac{2}{p})q}
\|a\|_{L^\infty_t(\dot{B}_{\frac{2p}{2-p},\infty}^{-1+\frac{2}{p}})}
\|f\|_{L^1_t(\dot B^0_{2,1})}.
\end{align*}
While we observe that
\begin{align*}
\|T'_{\dot\Delta_q\mathbb{P}f}a\|_{L^1_t(L^p)}
&\lesssim
\sum_{k\geq q-3}\|\dot{S}_{k+2}\dot{\Delta}_{q}f\|_{L^1_t(L^2)}
\|\dot{\Delta}_k a\|_{L^\infty_t(L^{\frac{2p}{2-p}})}
&\lesssim d_q2^{(1-\frac{2}{p})q}
\|a\|_{L^\infty_t(\dot{B}_{\frac{2p}{2-p},\infty}^{-1+\frac{2}{p}})}
\|f\|_{L^1_t(\dot B^0_{2,1})}.
\end{align*}
Finally, it follows from Lemma \ref{lem-commutator-1} that
\begin{align*}
\|[\dot\Delta_q\mathbb{P},T_a]f\|_{L^1_t(L^p)}
&\lesssim
\sum_{|q-k|\leq4} 2^{-q}
\|\nabla \dot{S}_{k-1} a\|_{L^\infty_t(L^{\frac{2p}{2-p}})}
\|\dot{\Delta}_{k}f\|_{L^1_t(L^2)}\\
&\lesssim d_q2^{q\bigl(1-\frac{2}{p}\bigr)}
\|a\|_{L^\infty_t(\dot B^{-1+\frac{2}{p}}_{\frac{2p}{2-p},\infty})}
\|f\|_{L^1_t(\dot B^0_{2,1})}.
\end{align*}
By summarizing the above estimates, we arrive at
\begin{equation}\label{COMMU}
\sum_{q\in\Z}2^{q\bigl(\frac{2}{p}-1\bigr)}
\|[\dot\Delta_q\mathbb{P},a]f\|_{L^1_t(L^p)}
\lesssim
\|a\|_{\widetilde L^\infty_t(\dot{B}_{\frac{2p}{2-p},1}^{-1+\frac{2}{p}})}
\|f\|_{L^1_t(\dot B^0_{2,1})}.
\end{equation}

By substituting the estimates \eqref{S4eq7} and \eqref{COMMU} and then
applying Gronwall's inequality, we achieve
$$
\begin{aligned}
\|u\|_{\widetilde L^\infty_t(\dot B^{-1+\frac{2}{p}}_{p,1})
\cap L^1_t(\dot B^{1+\frac{2}{p}}_{p,1})}
\lesssim
\bigl(\|u_0\|_{\dot B^{-1+\frac{2}{p}}_{p,1}}
+\|a\|_{\widetilde L^\infty_t(\dot B^{-1+\frac{2}{p}}_{\frac{2p}{2-p},1})}
\bigl(\|u\|_{L^1_t(\dot B^2_{2,1})}
+\|\nabla\Pi\|_{L^1_t(\dot B^0_{2,1})}\bigr)\bigr)
e^{C\|\nabla u\|_{L^1_t(L^\infty)}}.
\end{aligned}
$$

While it follows from
 Theorem 3.14 of \cite{BCD} that
$$
\|a\|_{\widetilde L^\infty_t(\dot B^{-1+\frac{2}{p}}_{\frac{2p}{2-p},1})}
\lesssim
\|a_0\|_{\dot B^{-1+\frac{2}{p}}_{\frac{2p}{2-p},1}}
e^{C\|u\|_{L^1_t(\dot B^1_{\infty,1})}}.
$$

As $p>1,$ we deduce from the  inequality \eqref{presure-crit-1} that
$$
\|\nabla\Pi\|_{L^1_t(\dot B^{{2\over p}-1}_{p,1})}
\lesssim
\|u\otimes u\|_{L^1_t(\dot B^{\frac{2}{p}}_{p,1})}
+\|a\Delta u\|_{L^1_t(\dot B^{-1+\frac{2}{p}}_{p,1})}
+\sum_{q\in\mathbb{Z}}2^{q\bigl(-1+\f2p\bigr)}\|[\dot\Delta_q,a]\nabla\Pi\|_{L^1_t(L^p)},
$$
from which and \eqref{COMMU}, we infer
 $$
\|\nabla\Pi\|_{L^1_t(\dot B^{{2\over p}-1}_{p,1})}
\lesssim
\|u\|_{L^2_t(\dot B^{\frac{2}{p}}_{p,1})}^2
+\|a\|_{\widetilde L^\infty_t(\dot B^{-1+\frac{2}{p}}_{\frac{2p}{2-p},1})}
\bigl(\|u\|_{L^1_t(\dot B^{1+\frac{2}{p}}_{p,1})}
+\|\nabla\Pi\|_{L^1_t(\dot B^0_{2,1})}\bigr).
$$
Similarly, we  deduce from the momentum equation of \eqref{1.3} that $\partial_tu\in L^1_t(\dot B^{-1+\frac{2}{p}}_{p,1})$.
This completes the proof of Corollary \ref{DW}.
\end{proof}

\bigbreak \noindent {\bf Acknowledgments.}
 G. Gui is supported in part by National Natural Science Foundation of China under Grants 12371211 and 12126359.
  P. Zhang is partially  supported by National Key R$\&$D Program of China under grant 2021YFA1000800, K. C. Wong Education Foundation and by National Natural Science Foundation of China under Grant 12031006.

\end{document}